\documentclass[final,onefignum,onetabnum]{siamart171218}

\usepackage{amsfonts}
\usepackage{graphicx}
\usepackage{epstopdf}
\usepackage{algorithmic}
\usepackage{mathdots}
\usepackage{multirow}
\usepackage[shortlabels]{enumitem}
\usepackage[toc,page]{appendix}
\usepackage[counterclockwise]{rotating}

\usepackage{caption}
\usepackage{subcaption}

\newsiamremark{remark}{Remark}
\newsiamremark{example}{Example}

\newsiamremark{hypothesis}{Hypothesis}
\crefname{hypothesis}{Hypothesis}{Hypotheses}
\newsiamthm{claim}{Claim}

\newcommand{\cred}[1]{{\color{black}  #1}}

\headers{A preconditioned MINRES method for wave optimal control}{S. Hon, J. Dong, and S. Serra-Capizzano}

\title{A preconditioned MINRES method for optimal control of wave equations and its asymptotic spectral distribution theory}

\author{Sean Hon\thanks{Corresponding Author. Department of Mathematics, Hong Kong Baptist University, Hong Kong SAR (\email{seanyshon@hkbu.edu.hk}).} \and Jiamei Dong\thanks{Department of Mathematics, Hong Kong Baptist University, Kowloon Tong, Hong Kong SAR (\email{21482799@life.hkbu.edu.hk}).} \and Stefano Serra-Capizzano\thanks{Department of Science and High Technology, University of Insubria, Italy (\email{{s.serracapizzano@uninsubria.it}}).}}

\usepackage{amsopn}

\begin{document}

\maketitle

\begin{abstract}
In this work, we propose a novel preconditioned Krylov subspace method for solving an optimal control problem of wave equations, after explicitly identifying the asymptotic spectral distribution of the involved sequence of linear coefficient matrices from the optimal control problem. Namely, we first show that the all-at-once system stemming from the wave control problem {is associated to a structured coefficient matrix-sequence possessing an} eigenvalue distribution. Then, based on such a spectral distribution {of which the symbol is explicitly identified}, we develop {an ideal preconditioner and} two parallel-in-time preconditioners for the saddle point system composed of two block Toeplitz matrices. {For the ideal preconditioner, we show that the eigenvalues of the preconditioned matrix-sequence all belong to the set $\left(-\frac{3}{2},-\frac{1}{2}\right)\bigcup \left(\frac{1}{2},\frac{3}{2}\right)$ well separated from zero, leading to mesh-independent convergence when the minimal residual method is employed.} The proposed {parallel-in-time} preconditioners can be implemented efficiently using fast Fourier transforms or {discrete} sine transforms, and their effectiveness is theoretically shown in the sense that the eigenvalues of the preconditioned matrix-sequences are clustered around $\pm 1$, which leads to rapid convergence. {When these parallel-in-time preconditioners are not \cred{fast} diagonalizable, we further propose modified versions which can be efficiently inverted. Several numerical examples are reported to verify our derived localization and spectral distribution result and to support the effectiveness of our proposed preconditioners and the related advantages with respect to the relevant literature.}
\end{abstract}


\section{Introduction}

{Over the past few decades, there has been a number of qualified papers devoted to analyzing and solving optimization problem constraint by hyperbolic equations. We refer to \cite{lions_1971,Bucci_1992,Gerdts_2008,Kroner_2011,Kunisch_2013,Kunisch_2016,Zuazua_2005,2020arXiv200509158G} and the references therein. In the current} work, we are interested in solving the distributed optimal control model problem studied in \cite{liuWu_optimal}. Namely, the following quadratic cost functional is minimized:
\begin{equation}\label{eqn:Cost_functional_wave}
\min_{y,u}~ \mathcal{L}(y,u):=\frac{1}{2}\| y - g \|^{2}_{L^2(\Omega \times (0,T))} + \frac{\gamma}{2}\| u \|^{2}_{L^2(\Omega \times (0,T))},
\end{equation}
subject to a wave equation with some initial and boundary conditions
\begin{equation}\label{eqn:wave}
\left\{
\begin{array}{lc}
 y_{tt} - \Delta y = f + u, \quad (x,t)\in \Omega \times (0,T], \qquad  y = 0, \quad (x,t)\in \partial \Omega \times (0,T], \\
y(x,0)=\psi_0,\quad y_t(x,0)=\psi_1, \quad x \in \Omega,
\end{array}
	\right.\,
\end{equation}
where $u,g \in L^2$ are the distributed control and the desired tracking trajectory respectively, $\gamma>0$ is a regularization parameter, and all $f$, $\psi_0$, and $\psi_1$ are given problem dependent functions. Under appropriate assumptions, theoretical aspects such as the existence, uniqueness, and regularity of the solution were well studied in \cite{lions_1971}. Making use of the strict convexity and the linearity of the problem, the optimal solution of (\ref{eqn:Cost_functional_wave}) \& (\ref{eqn:wave}) can be characterized by the following system {\cite{lions_1971}}:
\begin{equation}\label{eqn:wave_2}
\left\{
\begin{array}{lc}
 y_{tt} - \Delta y - \frac{1}{\gamma} p= f,\quad (x,t)\in \Omega \times (0,T], \qquad y = 0,\quad  (x,t)\in \partial \Omega \times (0,T], \\
y(x,0)=\psi_0,\quad y_t(x,0)=\psi_1,\quad  x \in \Omega,\\
 p_{tt} - \Delta p + y = g, \quad (x,t)\in \Omega \times (0,T], \qquad p = 0, \quad (x,t) \in \partial \Omega \times (0,T], \\
p(x,T) = 0,\quad p_t(x,T)= 0,\quad  x \in \Omega,
\end{array}
	\right.\,
\end{equation}
where the control variable $u$ has been eliminated. 

In this work, we follow the leap-frog finite difference scheme developed in \cite{LiLiuXiao2015} for (\ref{eqn:wave_2}), which gives
\begin{eqnarray*}
\frac{\mathbf{y}_m^{(k+1)} - 2\mathbf{y}_m^{(k)} + \mathbf{y}_m^{(k-1)} }{\tau^2} -  \Delta_m    \frac{\mathbf{y}_m^{(k+1)} +  \mathbf{y}_m^{(k-1)}}{2}  - \frac{1}{\tau} \mathbf{p}_m^{(k)} &=&  \mathbf{f}_m^{(k)}, \\
\frac{\mathbf{p}_m^{(k+1)} - 2\mathbf{p}_m^{(k)} + \mathbf{p}_m^{(k-1)} }{\tau^2} -  \Delta_m    \frac{\mathbf{p}_m^{(k+1)} +  \mathbf{p}_m^{(k-1)}}{2}  + \mathbf{y}_m^{(k)} &=&  \mathbf{g}_m^{(k)},
\end{eqnarray*} 
where $\mathbf{y}_m^{(k)}, \mathbf{p}_m^{(k)} \in \mathbb{R}^{m \times 1}$ are lexicographically ordered approximations of $y(\cdot,t)$ and $p(\cdot,t)$ on the grid points, respectively{, and $-\Delta_m \in \mathbb{R}^{m \times m}$ is a discrete matrix approximation of the negative Laplacian operator $-\Delta$}. {Here we have partitioned the time interval $[0,T]$ and  the space uniformly with the mesh step size $\tau=T/n$ and $h=1/(m+1)$, respectively, for given positive integers $n$ and $m$.}

Combining the given initial and boundary conditions, one needs to solve the following linear system 
\begin{equation}\label{eqn:main_system}
\mathcal{A}\begin{bmatrix} \sqrt{\gamma}\mathbf{y}\\ \mathbf{p} \end{bmatrix} = \begin{bmatrix} \cred{\mathbf{g}}\\  \cred{\sqrt{\gamma}\mathbf{f}} \end{bmatrix},
\end{equation}
where we have $\mathbf{y} = [ \mathbf{y}_m^{(1)}, \cdots, \mathbf{y}_m^{(n)}]^T$, \cred{$\mathbf{p} = [ \mathbf{p}_m^{(0)},\cdots, \mathbf{p}_m^{(n-1)}]^T$}, \cred{$\mathbf{f} = \tau^2 [ \frac{1}{2}\mathbf{f}_m^{(0)}+\mathbf{y}_m^{(1)}/\tau + \mathbf{y}_m^{(0)}/\tau^2, \cdots, \mathbf{f}_m^{(n-1)}]^T$}, \cred{$\mathbf{g} =\tau^2 [ \mathbf{g}_m^{(1)}, \cdots, \frac{1}{2}\mathbf{g}_m^{(n)}]^T \in \mathbb{R}^{mn \times 1}$}, and 
\begin{eqnarray}\label{eqn:matrix_A}
\mathcal{A} &=&
 \begin{bmatrix} 
\alpha \check{I}_n \otimes I_m  & (  B_n^{(1)}\otimes I_m - \frac{\tau^2}{2}  B_n^{(2)} \otimes \Delta_m )^T\\
 B_n^{(1)}\otimes I_m - \frac{\tau^2}{2}  B_n^{(2)}\otimes \Delta_m  &  -\alpha \hat{I}_n \otimes I_m 
\end{bmatrix}\\\nonumber
&=&
\begin{bmatrix} 
\alpha \check{I}_n \otimes I_m  & \mathcal{T}^T \\
 \mathcal{T}  & -\alpha \hat{I}_n \otimes I_m 
\end{bmatrix}.
\end{eqnarray}
Note that {$\alpha = \frac{\tau^2}{\sqrt{\gamma}}$}, $I_m$ is the $m \times m$ identity matrix, and
\begin{equation}\label{eqn:matrix_T}
\mathcal{T}=B_n^{(1)}\otimes I_m - \frac{\tau^2}{2}  B_n^{(2)}\otimes \Delta_m
\end{equation}
with $\hat{I}_n={\rm diag}(\frac{1}{2},1,\cdots,1)$, $\check{I}_n={\rm diag}(1,\cdots,1,\frac{1}{2}) \in \mathbb{R}^{n \times n}$, 
\[
B_n^{(1)} = \begin{bmatrix}
1 &   &  &  & \\
-2  & 1    & & & \\
 1 &  -2  & 1  & &  \\
  &   \ddots & \ddots & \ddots &  \\
 &  &  1 & -2 & 1
 \end{bmatrix},~ B_n^{(2)} = \begin{bmatrix}
1 &   &  &  & \\
0  & 1    & & & \\
 1 &  0  & 1  & &  \\
  &   \ddots & \ddots & \ddots &  \\
 &  &  1 & 0 & 1
 \end{bmatrix} \in \mathbb{R}^{n \times n}.
\] 
Clearly, as will be explained in Section \ref{sec:prelim} in more detail, the singular values of 
the nonsymmetric matrix-sequence $\{\mathcal{T}\}_n$, with $\mathcal{T}$ defined by (\ref{eqn:matrix_T}), are distributed as the matrix-valued function
\begin{equation}\label{eqn:function_h}
h(\theta)=  L_m-2I_m e^{\mathbf{i}\theta} + L_m e^{2\mathbf{i}\theta}
\end{equation} 
where $L_{m} = I_m  -  \frac{\tau^2}{2} \Delta_m \in \mathbb{R}^{m\times m}$, while the eigenvalues of $\mathcal{T}$ are exactly those of $L_m$ each with multiplicity $n$, given the block lower triangular structure of $\mathcal{T}$, with the block on the diagonal being $L_m$ repeated $n$ times.

Throughout this work, the discrete negative Laplacian matrix $-\Delta_{m}$ is assumed real symmetric positive definite (SPD) and hence $L_m$ is SPD as well. Such an assumption can be easily satisfied when a finite difference method is used on a uniform grid. In a more general case where the spatial domain is irregular, the identity matrix $I_m$ and $-\Delta_{m}$ are to be respectively replaced by the mass matrix and stiffness matrix, e.g. when a finite element method is deployed. 

Our contribution in this work is twofold. First, the asymptotic spectral distribution of the saddle-point matrix-sequence $\{ \mathcal{A} \}_n$ is precisely shown. In effect, we show that the eigenvalue distribution of $\mathcal{A}$ is explicitly determined by the generating function of the block Toeplitz matrix $T$ and by the parameter $\alpha$, in a way that is closely related to {the excellent study in} \cite{liuWu_optimal}. Such knowledge regarding eigenvalues is crucial to developing efficient solvers for the {very large} system (\ref{eqn:main_system}).

{Second, we develop in this work a preconditioned minimal residual (MINRES) method for (\ref{eqn:main_system}). {Based on the spectral distribution of $\mathcal{A}$ identified in the first part, we propose the following novel SPD block diagonal preconditioner as an ideal preconditioner for $\mathcal{A}$:}
\begin{eqnarray}\label{eqn:abs_ideal_matrix_H}
|\mathcal{H}|=\begin{bmatrix} 
 \sqrt{ \mathcal{T}^T \mathcal{T} + \alpha^2 I_{mn} } & \\
   & \sqrt{ \mathcal{T} \mathcal{T}^T + \alpha^2 I_{mn}}
\end{bmatrix}.
\end{eqnarray}

{Despite its excellent preconditioning effect for $\mathcal{A}$, which will be shown in Section \ref{sub:ideal_preconditioner}, $|\mathcal{H}|$ in itself is computational expensive to invert. As a result, we further propose the following parallel-in-time (PinT) preconditioner, which mimics $|\mathcal{H}|$ and can be fast implemented}:

\begin{eqnarray}\label{eqn:matrix_P}
\mathcal{P}_{S} = \begin{bmatrix} 
\sqrt{{S^T S} +  \alpha^2 I_{mn}} 
&  \\
  & \sqrt{{S S^T} +  \alpha^2 I_{mn}} 
\end{bmatrix}
\end{eqnarray}
where 
\begin{equation}\label{eqn:matrix_S}
S=S_n^{(1)}\otimes I_m - \frac{\tau^2}{2}  S_n^{(2)}\otimes \Delta_m,
\end{equation}
with $S_n^{(1)}$ and $S_n^{(2)}$ being, respectively, the Strang circulant matrices \cite{Strang1986} 
\[
S_n^{(1)} = \begin{bmatrix}
1 &   &  & 1 & -2\\
-2  & 1    & & & 1\\
 1 &  -2  & 1  & &  \\
  &   \ddots & \ddots & \ddots &  \\
 &  &  1 & -2 & 1
 \end{bmatrix}
 \quad
 \textrm{and}
 \quad
 S_n^{(2)} = \begin{bmatrix}
1 &   &  & 1 & 0\\
0  & 1    & & &1 \\
 1 &  0  & 1  & &  \\
  &   \ddots & \ddots & \ddots &  \\
 &  &  1 & 0 & 1
 \end{bmatrix}.
\]
 It is well-known that circulant matrices can be efficiently diagonalized by the fast Fourier transforms (FFTs), which can be parallelizable over different possessors. Hence, our proposed MINRES with this PinT preconditioner is particularly advantageous in an high performance computing environment.

\begin{remark}
Beside $S_n^{(1)}$ and $S_n^{(2)}$, one can also define the {Frobenius} optimal circulant matrices that was originally proposed in \cite{doi:10.1137/0909051} for unilevel Toeplitz systems. Since the spectral symbol $h$ of $\mathcal{T}$ as mentioned above is a matrix-valued trigonometric polynomial, it is expected {that the Frobenius} optimal one will have similar performance compared with the Strang type used in $\mathcal{P}_{S}$. Furthermore, we notice that $S_n^{(1)}$ is singular and this is not surprising, since this feature is always observed when the Strang circulant approximation of a constant coefficient differential operator on uniform gridding is considered. However, such a singularity does not represent a drawback for the use of the preconditioner $\mathcal{P}_{S}$, since the nonzero parameter $\alpha$ acts {as a form of regularization, similar to} the averaging process occurring in the Frobenius optimal preconditioning.  {We remind that the eigenvalues of the Frobenius optimal circulant preconditioner can be explicitly expressed as the Cesaro sum of the corresponding generating function on the usual circulant equispaced grid $2j\pi/n$, $j=0,\ldots,n-1$, while the Cesaro sum is precisely  the arithmetic mean of the Fourier sums of degree less than $n-1$, which results in a basic averaging process \cite{DS}.} Therefore, we only consider the use of $\mathcal{P}_{S}$ in this work.
\end{remark}

In addition to the block circulant based preconditioner, the following novel {PinT} preconditioner for $\mathcal{A}$ is proposed, which is motivated by \cite{Slinear,DS}, where it was shown that the Tau preconditioning could be superior than the circulant one when real symmetric structures are taken into consideration as those in (\ref{eqn:main_system}) \& (\ref{eqn:matrix_A}). More precisely, we define
\begin{eqnarray}\label{eqn:matrix_G}
\mathcal{P}_{G} = \begin{bmatrix} 
\sqrt{{G^TG} +  \alpha^2 I_{mn}} 
&  \\
  & \sqrt{{G G^T} +  \alpha^2 I_{mn}} 
\end{bmatrix}
\end{eqnarray}
where the symmetric matrix $G$ is expressed in the form
\begin{equation}\label{eqn:matrix_tau}
G=G_n^{(1)}\otimes I_m - \frac{\tau^2}{2}  G_n^{(2)}\otimes \Delta_m,
\end{equation}
with $G_n^{(1)}$ and $G_n^{(2)}$ being the Tau matrices \cite{BC83} {
\[
G_n^{(1)} = \begin{bmatrix}
2 & -1  &  &  & \\
-1  & 2    & -1 & & \\
  &  \ddots  & \ddots  & \ddots &  \\
  &    & \ddots & \ddots & -1  \\
 &  &   & -1 & 2
 \end{bmatrix}
  \quad
 \textrm{and}
 \quad
G_n^{(2)} = \begin{bmatrix}
0 & -1   &  &  & \\
-1  & 0    & 1 & & \\
  &  \ddots  & \ddots  & \ddots &  \\
  &    & \ddots & \ddots & -1  \\
 &  &   & -1 & 0
 \end{bmatrix},
 \]}respectively. Clearly, both $G_n^{(1)}$ and $G_n^{(2)}$ are diagonalizable by the discrete sine matrix $\mathbb{S}_n = \sqrt{\frac{2}{n+1}}\Big[\sin{\left(\frac{ij \pi}{n+1}\right)}\Big]^{n}_{i,j=1}  \in \mathbb{R}^{n \times n} $. {It should be noted  that the matrix $G$ is in fact the tridiagonal block Toeplitz preconditioner proposed for linear wave equations in our early work \cite{hon_SCapizzano_2023}.}
 
In particular, we will show that the eigenvalues of the preconditioned matrix-sequences $\{\mathcal{P}_{S}^{-1}\mathcal{A}\}_n$ and $\{\mathcal{P}_{G}^{-1}\mathcal{A}\}_n$ are clustered around $\pm 1$, which can lead to fast convergence when MINRES is employed.

{Both $\mathcal{P}_{S}$ and $\mathcal{P}_{G}$ require fast diagonalizability of $\Delta_m$ in order to be efficiently implemented. When such diagonalizability is not available, we further propose the following preconditioners $\widetilde{\mathcal{P}}_{S}$ and $\widetilde{\mathcal{P}}_{G}$ as modifications of $\mathcal{P}_{S}$ and $\mathcal{P}_{G}$, respectively:

\cred{
\begin{eqnarray}\label{eqn:matrix_P_mod}
&&\widetilde{\mathcal{P}}_{S}\\ \nonumber
&=& \textrm{blockdiag}\bigg(
\sqrt{{(S_n^{(1)})^T S_n^{(1)}} +  \alpha^2 I_{n}}  \otimes I_m - \frac{\tau^2}{2}\sqrt{(S_n^{(2)})^T S_n^{(2)}} \otimes \Delta_m,\\ \nonumber
&&\sqrt{{S_n^{(1)} (S_n^{(1)})^T} +  \alpha^2 I_{n}}  \otimes I_m - \frac{\tau^2}{2}\sqrt{S_n^{(2)} (S_n^{(2)})^T} \otimes \Delta_m \bigg).
\end{eqnarray}}

\cred{
\begin{eqnarray}\label{eqn:matrix_G_mod}
&&\widetilde{\mathcal{P}}_{G}\\\nonumber
&=& \textrm{blockdiag}\bigg(
\sqrt{{(G_n^{(1)})^T G_n^{(1)}} +  \alpha^2 I_{n}}  \otimes I_m - \frac{\tau^2}{2}\sqrt{(G_n^{(2)})^T G_n^{(2)}} \otimes \Delta_m, \\ \nonumber
&&\sqrt{{G_n^{(1)} (G_n^{(1)})^T} +  \alpha^2 I_{n}}  \otimes I_m - \frac{\tau^2}{2}\sqrt{G_n^{(2)} (G_n^{(2)})^T} \otimes \Delta_m \bigg).
\end{eqnarray}}As will be mentioned in Section \ref{subsubsec:implementationP}, both $\mathcal{P}_{S}$ and $\mathcal{P}_{G}$ can be efficiently inverted. A further discussion on the case where $m$ is large or is allowed to grow with \cred{$n$, that is $m=m(n)$,} is provided in the conclusion section.
}

{Notably, our proposed MINRES {methods are} in contrast with the state-of-the-art solver established in \cite{liuWu_optimal}, where a generalized minimal residual (GMRES) method with a block-circulant type preconditioner was proposed, and the eigenvalues of the preconditioned matrix was shown clustered. However, as mentioned in \cite[Chapter 6]{ANU:9672992}, the convergence study of preconditioning strategies for nonsymmetric problems is largely heuristic, since descriptive convergence bounds are usually not available for GMRES or any of the other applicable nonsymmetric Krylov subspace iterative methods. Besides, the convergence behaviour of GMRES in general cannot be rigorously analyzed by using only eigenvalues \cite{Greenbaum_1996}. Our MINRES solver does not only circumvent these theoretical difficulties of GMRES, but also provide a positive answer to the open problem discussed in \cite[Section 3]{liuWu_optimal} by showing that an effective symmetric positive definite preconditioner can be constructed for the concerned wave optimal control problem. It is also worth noting that our preconditioning approach does not involve any approximation of Schur complement (see, e.g., \cite{LiuPearson2019}), which is a common approach in the context of preconditioning for saddle point systems.} {Also, our proposed methods, based on fast transforms, are in nature different from the constraint preconditioner proposed in \cite{LiLiuXiao2015}. Finally, compared with \cite{liuWu_optimal,LiuPearson2019}, our preconditioned MINRES proposal can be faster in CPU times due to the fewer theoretical operation counts}, and this advantage becomes more evident for larger matrix sizes.

The paper is organized as follows. In Section \ref{sec:prelim}, we review some preliminary results on block Toeplitz matrices. In Section \ref{sec:main}, we provide our main results on the spectral distribution theory and our proposed preconditioners. Numerical examples are given in Section \ref{sec:numerical} {for showing the expected performance of our proposed preconditioners, the comparison with the proposal in \cite{liuWu_optimal}, and for supporting the theoretical results on the eigenvalues of the original coefficient matrices $\mathcal{A}$ and of the associated preconditioned matrices}.

\section{Preliminaries on block Toeplitz matrices}
\label{sec:prelim}

In this section, we provide some useful background knowledge regarding block Toeplitz matrices.

We let $L^1([-\pi,\pi],\mathbb{C}^{m \times m})$ be the Banach space of all matrix-valued functions that are Lebesgue integrable over $[-\pi,\pi]$, {equipped with the following norm}
\[
\|f\|_{L^1} = \frac{1}{2\pi}\int_{-\pi}^{\pi} \|f(\theta)\|_{\text{tr}}\,d \theta < \infty,
\] 
where $\|A\|_{\text{tr}}:=\sum_{j=1}^{N}\sigma_{j}(A)$ denotes the trace norm of $A\in \mathbb{C}^{N \times N}$. The block Toeplitz matrix generated by $f \in L^1([-\pi,\pi],\mathbb{C}^{m \times m})$ is denoted by $T_{(n,m)}[f]$, namely 
\begin{eqnarray*}\label{eqn:det_blockT}
{T}_{(n,m)}[f]&=&
\begin{bmatrix}
A_{(0)} & A_{(-1)} & \cdots & A_{(-n+1)} \\
A_{(1)} & \ddots & \ddots & \vdots \\
\vdots & \ddots & \ddots  & A_{(-1)} \\
A_{(n-1)} & \cdots & A_{(1)} & A_{(0) }
\end{bmatrix} \in \mathbb{C}^{mn \times mn},
\end{eqnarray*}
where the Fourier coefficients of $f$ are
\[ A_{(k)}=\frac{1}{2\pi} \int_{-\pi} ^{\pi}f(\theta) e^{-\mathbf{i} k \theta } \,d\theta\in \mathbb{C}^{m \times m},\qquad k=0,\pm1,\pm2,\dots.
\] The function $f$ is called the {\emph{generating function}} of $T_{({n},m)}[f]$. 
For thorough discussions on the related properties of block Toeplitz matrices, we refer readers to {\cite{GaroniCapizzano_one} and  references therein; for computational features see \cite{MR2108963,Chan:1996:CGM:240441.240445,serra_block1,serra_block2,serra_indef,multigrid_block,MR1980639} and references there reported.} 

Throughout this work, we assume that ${f \in L^1([-\pi,\pi],\mathbb{C}^{m \times m})}$ and is periodically extended to the real line. Furthermore, we follow all standard notation and terminology introduced in \cite{GaroniCapizzano_one}: {let $C_c(\mathbb{K})$ ($\mathbb{K}$ being $\mathbb{R}$ or $\mathbb{C}$) be the space of complex-valued continuous functions defined on $\mathbb{K}$ with bounded support and let $\eta$ be a functional, i.e., in our case any function defined on the vector space $C_c(\mathbb{K})$, taking values in $\mathbb{C}$. Also, if $k,q$ are positive integers and $g:D\subset \mathbb{R}^k \to \mathbb{K}^{q\times q}$ is a $q\times q$ matrix-valued measurable function defined on a set $D$ with $0<\mu_k(D)<\infty$, then} the functional $\eta_g$ is denoted such that
	\[
	\eta_g:C_c(\mathbb{K})\to \mathbb{C} \quad \text{and} \quad 
	\eta_g(F)=\frac{1}{q\mu_k(D)}\int_D \sum_{j=1}^q F(\lambda_j(g(\mathbf{\theta})))\,d\mathbf{\theta}.
	\]
\begin{definition}{\rm \cite[Definition 3.1]{GaroniCapizzano_one}({two types of asymptotic distributions} for general matrix-sequences)}\label{def:spectral_distribution}
Let $\{A_n\}_n$ be a matrix-sequence.
We say that $\{A_n\}_n$ has an asymptotic eigenvalue (or spectral) distribution described by a functional $\eta:C_c(\mathbb{R})\to \mathbb{C},$ and we write $\{A_n\}_n \sim_{\lambda}\eta,$ if
\[
\lim_{n\to \infty} \frac{1}{n}\sum_{i=1}^{n}F(\lambda_i(A_n))=\eta(F),\quad \forall F \in C_c(\mathbb{C}).
\] 
{If $k,q$ are positive integers, $\eta=\eta_{f}$ for some $q\times q$ matrix-valued measurable $f:D \subset \mathbb{R}^k \to \mathbb{C}^{q\times q}$ defined on a set $D$ with $0<\mu_k(D)<\infty,$ we say that $\{A_n\}_n$ has an asymptotic eigenvalue (or spectral) distribution described by $f$ and we write $\{A_n\}_n \sim_{\lambda} f$. In this case, the function $f$ is referred to as the eigenvalue (or spectral) symbol of the matrix-sequence $\{A_n\}_n$.
With the notation as before, with $X^*$ meaning the usual transpose and conjugate of its argument $X$, and by noticing that the eigenvalues of $|f|$, $|f|^2=f^*f$, coincide with the singular values of $f$, if $\eta=\eta_{|f|}$,  then we say that $\{A_n\}_n$ has an asymptotic singular value distribution described by  $f$ and we write $\{A_n\}_n \sim_{\sigma} f$. In this case, the function $f$ is referred to as the singular value symbol of the matrix-sequence $\{A_n\}_n$.
}
\end{definition}

{Taking into consideration the previous definition, the following result for block Toeplitz matrix-sequences holds, generalizing the celebrated Szeg\"o theorem.}
{
\begin{theorem}{\rm \cite{Tilli_98}}\label{tilli th}
If $f\in L^1([-\pi,\pi],\mathbb{C}^{m \times m})$ and we consider the block Toeplitz matrix-sequence {$\{{T}_{(n,m)}[f]\}_n$}, then 
\[
\{{T}_{(n,m)}[f]\}_n \sim_{\sigma} f
\]
in the sense of Definition \ref{def:spectral_distribution} with $k=1$ and $q=m$, that is, the generating function $f$ is always the singular value symbol of $\{{T}_{(n,m)}[f]\}_n$. If $f$ is $m\times m$ {Hermitian matrix-valued}, then $T_{({n},m)}[f]$ is Hermitian for {any choice of $n$ and $m$,} and  $f$ is also the spectral symbol of the sequence $\{T_{({n},m)}[f]\}_n$ in the sense of Definition \ref{def:spectral_distribution} with $k=1$ and $q=m$, that is, $\{T_{({n},m)}[f]\}_n \sim_{\lambda} f$. 
	\end{theorem}
}

	Moreover, we introduce the following definitions and a key theorem in order to prove our main distribution results in the next {section}.
	\begin{definition}{\rm \cite[Definition 5.1]{GaroniCapizzano_one}(approximating class of sequences)}\label{def:ACS}
		Let $\{A_n\}_n$ be a matrix-sequence and let $\{\{B_{n,j}\}_n\}_j$ be a sequence of matrix-sequences. We say that $\{\{B_{n,j}\}_n\}_j$ is an \textit{approximating class of sequences (a.c.s.)} for $\{A_n\}_n$ if the following condition is met: for every $j$ there exist nonnegative $n_j$, $c(j)$, $\omega(j)$, such that, for $n \geq n_j$, we have
		\[
		A_n=B_{n,j}+R_{n,j}+N_{n,j}, 
		\]
		\[
		{\rm rank}(R_{n,j}) \leq c(j)n \quad \text{and} \quad \|N_{n,j}\|\leq\omega(j),
		\]
		where $n_j$, $c(j)$, and $\omega(j)$ depend only on $j$ and \[\lim_{j\to\infty}c(j)=\lim_{j\to\infty}\omega(j)=0.\]
	\end{definition}
	
	We use $\{B_{n,j}\}_n\xrightarrow{\text{a.c.s.\ wrt\ $j$}}\{A_n\}_n$ to denote that $\{\{B_{n,j}\}_n\}_j$ is an a.c.s. for $\{A_n\}_n$.

	\begin{definition}
		Let $f_j,f:D \subset \mathbb{R}^k \to \mathbb{C}^{q\times q}$, $j,k,q$ positive integers, be a matrix-valued measurable functions. We say that $f_j \to f$ in measure if, for every $\epsilon > 0$,
		\[
		\lim_{j \to \infty} \mu_k \{ \sigma_{\min} (f_j - f) > \epsilon \}=0,
		\]
{with $\mu_k \{\cdot\}$ denoting the standard Lebesgue measure on $\mathbb{R}^k$.}
	\end{definition}

	\begin{theorem}{\rm \cite[Corollary 5.1]{GaroniCapizzano_one} and \cite[Theorem 2.34]{GaroniCapizzano_three}}\label{lem:Corollary5.1}
		Let $\{A_n\}_n$ and $ \{B_{n,j}\}_n$ be Hermitian matrix-sequences, let $j,k,q$ be positive integers, and let $f,f_j:D \subset \mathbb{R}^k \to \mathbb{C}^{q\times q}$ be {Hermitian matrix-valued} measurable functions defined on a set $D$ with $0<\mu_k(D)<\infty$. Suppose that
		
		\begin{enumerate}
			\item $\{B_{n,j}\}_n \sim_{\lambda}  f_j$ for every $j$,
			\item $\{B_{n,j}\}_n\xrightarrow{\text{a.c.s.\ wrt\ $j$}}\{A_n\}_n$,
			\item $f_j \to f$ in measure.
		\end{enumerate}
		
		Then,
		\[
		\{A_{n}\}_n \sim_{\lambda}  f.
		\]
		Moreover, if the first assumption is replaced by $\{B_{n,m}\}_n \sim_{\sigma}  f_m$ for every $m$, given that the other two assumptions are left unchanged, and we drop the assumption that all the involved matrices are Hermitian and $f$ {Hermitian matrix-valued}, then {$\{A_{n}\}_n \sim_{\sigma}  f$}.
	\end{theorem}

Before discussing the asymptotic spectral distribution of $\mathcal{A}$ associated with $h$ {defined by (\ref{eqn:function_h})}, it is crucial to our preconditioning theory development that we introduce the following notation, which was established in \cite{Ferrari2019} for certain symmetrized block Toeplitz matrix-sequences (see also \cite{MazzaPestana2018} for a detailed analysis on the same problem in the uni-level case). Given $D \subset \mathbb{R}^k$ with Lebesgue measure $0<\mu_k(D)<\infty$, we define $\widetilde D_p$ as $D\bigcup D_p$, where $p\in  \mathbb{R}^k$ and $D_p=p+D$, with the constraint that $D$ and $D_p$ have non-intersecting interior part, i.e., $D^\circ \bigcap  D_p^\circ  =\emptyset$. In this way, we have $\mu_k(\widetilde D_p)=2\mu_k(D)$. Given any $g$ measurable and $q \times q$ {Hermitian matrix-valued} defined over $D$, we define $\psi_g$ over $\widetilde D_p$ in the following way
	\begin{equation}\label{def-psi}
	\psi_g(x)=\left\{
	\begin{array}{cc}
	g(x), & x\in D, \\
	-g(x-p), & x\in D_p, \ x \notin D.
	\end{array}
	\right.\,
	\end{equation}

If $q=1$ and $f$ has only real Fourier coefficients, we know that the symmetrized flipped matrix-sequence $\{Y_n T_{(n,1)}[f]\}_n \sim_{\lambda}  \psi_{|f|}$ (see \cite{Ferrari2019}). In Section \ref{subsec:part_one}, we will prove a statement of this type for the matrix-sequence $\{\mathcal{A}\}_n$ with $q=2m>1$. {We remind that the results in \cite{Ferrari2019,MazzaPestana2018} were motivated and stimulated by an algorithmic intuition by Pestana and Wathen \cite{PW_2015}, which lead to a fruitful research direction for efficient preconditioned MINRES solvers and for the related spectral analysis of the associated matrix-sequences.}

\subsection{Matrix analysis}

In this subsection, several useful results on matrix analysis are presented.

If $\widehat{F}$ is analytic on a simply connected open region of the complex plane containing the interval $[-1, 1]$, there exist ellipses with foci in $-1$ and $1$ such that $\widehat{F}$ is analytic in their interiors. Let $r_\alpha > 1$ and $r_\beta >0$ be the half axes of such an ellipse with $\sqrt{ {r_\alpha}^2 - {r_\beta}^2}=1$. Such an ellipse, denoted by $\mathbb{E}_{\mathcal{X}}$, is completely specified once the number $\mathcal{X}:= r_\alpha + r_\beta$ is known.

\begin{theorem}[Bernstein's theorem]{\rm \cite[Theorem 2.1]{Benzi1999}}\label{thm:Bernstein}
Let the function $\widehat{F}$ be analytic in the interior of the ellipse $\mathbb{E}_{\mathcal{X}}$ with $\mathcal{X}>1$ and continuous on $\mathbb{E}_{\mathcal{X}}$. In addition, suppose $\widehat{F}(x)$ is real for real $x$. Then the best approximation error
\[
E_k(\widehat{F}):={\rm inf}\{\| {\widehat{F}} - p \|_{\infty} : {\rm deg}(p) \leq k\} \leq \frac{2M(\mathcal{X})}{\mathcal{X}^k(\mathcal{X}-1)},
\]
where ${\rm deg}(p)$ denotes the degree of the polynomial $p(x)$ and
\[
\| \widehat{F} - p \|_{\infty} = \max_{-1 \leq x \leq 1} \vert \widehat{F}(x)-p(x) \vert, \quad M(\mathcal{X})=\max_{x \in \mathbb{E}_{\mathcal{X}}}\{ \vert \widehat{F}(x) \vert \}.
\]
\end{theorem}

Let $A_N$ be an $N \times N$ symmetric matrix and let $[\lambda_{\min}, \lambda_{\max}]$ be the smallest interval containing $\sigma(A_N)$. If we introduce the linear affine
function \[\psi(\lambda) = \frac{2 \lambda - (\lambda_{\min} + \lambda_{\max})}{\lambda_{\max} - \lambda_{\min}},\] then $\psi([\lambda_{\min},\lambda_{\max}])=[-1,1]$ and therefore the spectrum of the symmetric matrix \[B_N:=\psi(A_N)=\frac{2}{\lambda_{\max} - \lambda_{\min}}A_N - \frac{\lambda_{\min} + \lambda_{\max}}{\lambda_{\max} - \lambda_{\min}}I_N\] is contained in $[-1,1]$. Given a function $\widehat{f}$ analytic on a simply connected region containing $[\lambda_{\min}, \lambda_{\max}]$ and such that $\widehat{f}(\lambda)$ is real when $\lambda$ is real, the function $\widehat{F} = \widehat{f} \circ \psi^{-1}$ satisfies the assumptions of Bernstein's theorem.

In the special case where $A_N$ is SPD and $\widehat{f}(x) = x^{-1/2}$ that are of our interest in this work, we apply Bernstein's result to the function
\begin{equation}\label{def:function_F}
\widehat{F}(x) = \frac{1}{ \sqrt{ \frac{(b - a)}{2}x  + \frac{ a+ b}{2} } },
\end{equation}
where $a=\lambda_{\min}(A_N), b=\lambda_{\max}(A_N)$, and $1 < \mathcal{X} < \frac{\sqrt{\kappa}+1}{\sqrt{\kappa}-1}$ with the spectral condition number of $A_N$ being $\kappa=\frac{b}{a}.$

\section{Main results}\label{sec:main}

In what follows, we first discuss the spectral distribution of a special $2 \times 2$ block matrix-sequences, which includes the all-at-once matrix from the optimal control problem given by (\ref{eqn:matrix_A}) as a special case. Several subsections on our proposed preconditioners are then provided, which cover related issues such as implementation and convergence analysis. 

\subsection{Asymptotic spectral distribution of a certain $2 \times 2$ block matrix-sequence}\label{subsec:part_one}

In this section we provide three results of distributional type. The first one (i.e., Theorem \ref{thm:eig_matrix_A}) concerns a special kind of structured matrix-sequence $\{ \mathcal{M}_n[f]\}_n$, with $\mathcal{M}_n[f]=\mathcal{M}[f]$ of the form
\[
\mathcal{M}[f] = \begin{bmatrix} 
\alpha I_{mn}  & T_{(n,m)}[f]^*\\
 T_{(n,m)}[f]  &  -\alpha I_{mn}
\end{bmatrix} = \mathcal{A} -{\alpha \over 2}\mathbf{e}_{n}\mathbf{e}_{n}^T \otimes I_m + {\alpha \over 2}\mathbf{e}_{n+1}\mathbf{e}_{n+1}^T \otimes I_m
\in \mathbb{C}^{2mn \times 2mn},
\] 
where $\mathbf{e}_j$ is the $j$-th column of the $2n \times 2n$ identity matrix $I_{2n}$.

Then, using the notion of a.c.s. in Definition \ref{def:ACS} and the approximation Theorem \ref{lem:Corollary5.1}, in Corollary \ref{cor:A-distr} we obtain the eigenvalue distribution of our wave optimal control matrix-sequence $\{ \mathcal{A}_n\}_n$, with $\mathcal{A}_n=\mathcal{A}$ given in (\ref{eqn:matrix_A}).

Finally, Theorem \ref{thm:eig_matrix_gen} represents a generalization of Theorem \ref{thm:eig_matrix_A}, whose proof is essentially identical to the combination of the proof of Theorem \ref{thm:eig_matrix_A} and of Corollary \ref{cor:A-distr}.

\begin{theorem}\label{thm:eig_matrix_A}
Suppose $f \in L^1([-\pi,\pi], \mathbb{C}^{m \times m})$. Let
\[
\mathcal{M}[f] = \begin{bmatrix} 
\alpha I_{mn}  & T_{(n,m)}[f]^*\\
 T_{(n,m)}[f]  &  -\alpha I_{mn}
\end{bmatrix}
\in \mathbb{C}^{2mn \times 2mn}
\] with $\alpha \in \mathbb{R}$ and $T_{(n,m)}[f] \in \mathbb{C}^{mn \times mn}$ being the block Toeplitz matrix generated by $f$. 

Then,
\[
\{ \mathcal{M}[f]  \}_n \sim_{\lambda}  \psi_{ { g }}, \qquad g=\sqrt{|f|^2 + \alpha^2I_m } ~~ \textrm{with} ~~ |f|=(f ^*f)^{1/2},
\]
over the domain  $\widetilde D$ with $D=[0,2\pi]$ and $p=-2\pi$, where $\psi_{g }$ is defined in (\ref{def-psi}). That is
\begin{eqnarray*}
\lim_{n\to \infty} \frac{1}{2mn}\sum_{i=1}^{2mn}F(\lambda_i( \mathcal{M} [f])) & = &
 \frac{1}{2\pi}\int_{-\pi}^{\pi} {\frac{1}{m}} \sum_{i=1}^{m} F(\lambda_i(g(\theta)))\,d\theta + \\
 & & \frac{1}{2\pi}\int_{-\pi}^{\pi} {\frac{1}{m}} \sum_{i=1}^{m} F(-\lambda_i(g(\theta)))\,d\theta, 
\end{eqnarray*} where $\lambda_i( g(\theta))$, $i=1,2,\dots,{m}$, are the eigenvalue functions of $g$.
\end{theorem}
\begin{proof}
{
Considering the singular value decomposition of $T_{(n,m)}[f] = U \Sigma V^*$, the associated decomposition of $\mathcal{M}[f]$ is deduced by direct computation, that is
\begin{eqnarray}\label{eqn:matrix_Mat}
\mathcal{M}[f] &=&
\begin{bmatrix} 
\alpha I_{mn}  & V\Sigma U^*\\
 U\Sigma V^*  &  -\alpha I_{mn}
\end{bmatrix} \\\nonumber
&=&
\begin{bmatrix} 
V  & \\
   &  U
\end{bmatrix}
\begin{bmatrix} 
\alpha I_{mn}  & \Sigma\\
 \Sigma  &  -\alpha I_{mn}
\end{bmatrix}
\begin{bmatrix} 
V  & \\
   &  U
\end{bmatrix}^*  \\\nonumber
&=&\begin{bmatrix} 
V  & \\
   &  U
\end{bmatrix}
\mathcal{{\widehat{Q}}}
\begin{bmatrix} 
\sqrt{\Sigma^2+\alpha^2 I_{mn}}  & \\
   &  -\sqrt{\Sigma^2+\alpha^2 I_{mn} }
\end{bmatrix}
\mathcal{{\widehat{Q}}}^*
\begin{bmatrix} 
V  & \\
   &  U
\end{bmatrix}^*,
\end{eqnarray} 
where $\mathcal{{\widehat{Q}}}$ is real orthogonal expressed by 
\begin{eqnarray*}
\mathcal{{\widehat{Q}}}
=
\begin{bmatrix} 
 \Sigma D_1^{-1}& -\Sigma D_2^{-1}\\
 \big(\sqrt{\Sigma^2+\alpha^2 I_{mn}}  - \alpha  I_{mn}\big)D_1^{-1}   & \big(\sqrt{\Sigma^2+\alpha^2 I_{mn}}  + \alpha I_{mn}\big)D_2^{-1}
\end{bmatrix}
\end{eqnarray*}
with
\[
D_1 =  \sqrt{(-\sqrt{\Sigma^2+\alpha^2 I_{mn}}  + \alpha I_{mn})^2 + \Sigma^2} 
\]
and
\[
D_2 = \sqrt{(\sqrt{\Sigma^2+\alpha^2 I_{mn}}  + \alpha I_{mn})^2 + \Sigma^2 }.
\]
It is obvious that both diagonal matrices $D_1$ and $D_2$ are invertible. Thus $\widehat{Q}$ is well-defined.

Since both $\widehat{Q}$ and $\begin{bmatrix} 
V  & \\
   &  U
\end{bmatrix}$ are unitary, $\mathcal{Q}:=\begin{bmatrix} 
V  & \\
   &  U
\end{bmatrix}\widehat{Q}$ is also unitary. Hence, from Eq. (\ref{eqn:matrix_Mat}), we have obtained an eigendecomposition of $\mathcal{M}[f]$, i.e., 
\[
\mathcal{M}[f]=\mathcal{Q}\begin{bmatrix} 
\sqrt{\Sigma^2+\alpha^2 I_{mn}}  & \\
   &  -\sqrt{\Sigma^2+\alpha^2 I_{mn} }
\end{bmatrix}
\mathcal{Q}^*.
\]

}
Evidently, the eigenvalues of $\mathcal{M}[f]$ are given by the set of the singular values of $\sqrt{ \mathcal{T}^T \mathcal{T} + \alpha^2 I_{mn} }$ (or $\sqrt{ \mathcal{T}\mathcal{T}^T + \alpha^2 I_{mn} }$) and the negation of the set of these singular values. Thus, it is clear that
		\[
                     \{\mathcal{M}[f]\}_n \sim_{\lambda}  \psi_g.
                 \]
                 The proof is concluded.
\end{proof}

The following corollary as a consequence of Theorem \ref{thm:eig_matrix_A} is helpful to developing our preconditioning theory for $\mathcal{A}$ stemming from the concerned wave control problem.

\begin{corollary}\label{cor:A-distr}
Let $\mathcal{A} \in \mathbb{R}^{2mn \times 2mn}$ be defined by (\ref{eqn:matrix_A}) in which $\mathcal{T}=T_{(n,m)}[h]$ is generated by $h$ given in (\ref{eqn:function_h}). 
Then,
\[
\{ \mathcal{A}  \}_n \sim_{\lambda}  \psi_{ { g }}, \qquad g=\sqrt{|h|^2 + \alpha^2I_m } ~~ \textrm{with} ~~ |h|=(h^* h)^{1/2},
\]
over the domain  $\widetilde D$ with $D=[0,2\pi]$ and $p=-2\pi$, where $\psi_{g }$ is defined in (\ref{def-psi}). 

\end{corollary}
\begin{proof}
We consider the following decomposition
\begin{eqnarray}\nonumber
\mathcal{A} &=& 
\begin{bmatrix} 
\alpha \check{I}_n \otimes I_m  & {T}_{(n,m)}[h]^T \\
 {T}_{(n,m)}[h]  & -\alpha \hat{I}_n \otimes I_m 
\end{bmatrix}\\
&=&\underbrace{\begin{bmatrix} 
\alpha I_{mn}  & {T}_{(n,m)}[h]^T\\\label{eqn:matrix_H}
 {T}_{(n,m)}[h] &  -\alpha I_{mn}
\end{bmatrix}}_{=:\mathcal{H}}+
\underbrace{\begin{bmatrix} 
-\frac{\alpha}{2} \mathbf{\hat{e}}_n   \mathbf{\hat{e}}_n^T\otimes I_m  & \\
   & \frac{\alpha}{2}  \mathbf{\hat{e}}_1   \mathbf{\hat{e}}_1^T \otimes I_m 
\end{bmatrix}}_{=:\mathcal{{R}}},
\end{eqnarray}
where $ \mathbf{\hat{e}}_j$ is the $j$-th column of the $n \times n$ identity matrix.

We see that $\{ \mathcal{H} \}_n \sim_{\lambda}  \psi_{ { g }},~ g=\sqrt{|h|^2 + \alpha^2I_m }$, with $|h|=(h^* h)^{1/2}$ by Theorem \ref{thm:eig_matrix_A} and ${\rm rank} (\mathcal{ {R} }) \leq 2m.$

Finally, since all the involved matrices are Hermitian and the perturbation matrix-sequence is zero distributed, i.e, $\{\mathcal{R}\}_n \sim_{\lambda}  0$, the desired result follows directly from the second part of Theorem \ref{lem:Corollary5.1}, taking into account that $\{\{\mathcal{H}\}_n\}_j$ is a constant class of sequences, that is not depending on the variable $j=m$ and it is a fortiori an a.c.s. for $\{\mathcal{A}\}_n$, since, in accordance with Definition \ref{def:ACS}, the norm correction is $N_{n,m}\equiv 0$ and
\[
{\rm rank} (\mathcal{ {R} }) \leq 2m \leq {n\over m}
\]
for $n\ge 2 m^2$, with $R_{n,m}=\mathcal{ {R} }$, $\omega(m)=0$, $c(m)={1\over m}$, and $n_m =2 m^2$.
\end{proof}

Now we conclude with a wide generalization of Theorem \ref{thm:eig_matrix_A} and of Corollary \ref{cor:A-distr} and whose proof mimics, partly verbatim, those of Theorem \ref{thm:eig_matrix_A} and of Corollary \ref{cor:A-distr}.

\begin{theorem}\label{thm:eig_matrix_gen}
Assume that 
\[
\mathcal{A}^{\rm gen} = \begin{bmatrix} 
 Y_n(\alpha)  & X_{n}^*\\
 X_{n}  &  Z_n(\alpha)
\end{bmatrix}
\in \mathbb{C}^{2mn \times 2mn}
\] 
with $Y_n(\alpha)=Y_n^*(\alpha)$, $Z_n(\alpha)=Z_n^*(\alpha)$, $\{Y_n(\alpha)\}_n \sim_{\lambda} \alpha$, $\{Z_n(\alpha)\}_n \sim_{\lambda} -\alpha$, 
{$\{X_n\}_n \sim_{\sigma} f$, $f$ generic measurable function, and $\alpha$, $g=\sqrt{|f|^2 + \alpha^2I_m }$,} $\psi_{g }$ as in Theorem \ref{thm:eig_matrix_A}.
Then,
\[
\{ \mathcal{A}^{\rm gen} \}_n \sim_{\lambda}  \psi_{ { g }}, 
\]
over the domain  $\widetilde D$ with $D=[0,2\pi]$ and $p=-2\pi$. 
\end{theorem}
{The remark below is given for explaining the generality and the potential of Theorem \ref{thm:eig_matrix_gen}.
}
\begin{remark}\label{rem:eig_matrix_gen}
{With reference to the above theorem, we specify how the matrices could be concretely defined. Indeed, using the given assumptions, we can write
\[
Y_n(\alpha)=\alpha I_{mn}+W_{n,m}, \ \ \ \ \ Z_n(\alpha)=-\alpha I_{mn} + W_{n,m}',
\]
where $\alpha$ is real and both $\{W_{n,m}\}_n$, $\{W_{n,m}'\}_n$ are both Hermitian matrix-sequence, zero distributed in the eigenvalue sense. 
As already recalled in the proof of Corollary \ref{cor:A-distr}, the notion of eigenvalue zero distribution is that \cred{of} Definition \ref{def:spectral_distribution} with identically zero symbol in the eigenvalue sense.
In the present setting, $\{W_{n,m}\}_n$, $\{W_{n,m}'\}_n$ are like the matrix-sequence $\{A_{n}\}_n$ in Definition \ref{def:ACS} with $B_{n,j}$ being the zero matrix for every $n$ and every $j$ and with all the involved matrices being Hermitian. 
In choosing $\{X_n\}_n$ in Theorem \ref{thm:eig_matrix_gen} we have even more degrees of freedom. More specifically, by using the rearrangement theory (see \cite{BBG} and references therein), we could choose $\{X_n\}_n \sim_{\sigma} \tilde f$ with any $\tilde f$ such that
\[
\eta_{|\tilde f|}(F)=\eta_{|f|}(F)
\]
for every $F \in C_c(\mathbb{R})$. Notice that $\tilde f$ can by chosen $q\times q'$ matrix-valued while $f$ is $m\times m$ matrix-valued, where $q,q'$ are in principle different from $m$. By summarizing the degrees of freedom are numerous and the present remark emphasizes the wide potential in terms of range of applicability of Theorem \ref{thm:eig_matrix_gen}.}
\end{remark}

\subsection{Preconditioning}

After identifying the spectral distribution of $\{\mathcal{A}\}_n$, we now turn our focus to constructing effective preconditioners for $\mathcal{A}$.

\subsubsection{Ideal preconditioner}\label{sub:ideal_preconditioner}

First, we observe that the eigenvalue distribution gives us some preliminary information directly. When the parameter $\gamma$ is very small compared to $\tau$, we can assume that 
$
\alpha = \frac{\tau^2}{\sqrt{\gamma}} 
$
is very large. Indeed, for $\gamma \rightarrow 0^+$, ${g \over \alpha} \rightarrow 1$ and hence, in view of Corollary \ref{cor:A-distr}, the matrix-sequence $\left\{{\mathcal{A}\over \alpha} \right\}_n$ has {eigenvalues already clustered around $\pm 1$ and consequently the all-at-once system can be solved easily even without preconditioning.}

Conversely, in the case where we do not observe this degenerate setting, we can assume that $\alpha$ is just a positive constant and in the sequel it will be shown that an ideal preconditioner for $\mathcal{A}$ is the SPD matrix $|\mathcal{H}|=\sqrt{\mathcal{H}^2}$ defined by (\ref{eqn:abs_ideal_matrix_H}).

\begin{proposition}\label{eqn:ideal_precon}
Let $\mathcal{A}, \mathcal{H} \in \mathbb{R}^{2mn \times 2mn}$ be defined by (\ref{eqn:matrix_A}) and  (\ref{eqn:matrix_H}), respectively. Then,
\[
|\mathcal{H}|^{-1}\mathcal{A} = |\mathcal{H}|^{-1} \mathcal{H}  + \mathcal{ \widetilde{R}},
\]
where $|\mathcal{H}|^{-1} \mathcal{H} $ is both real symmetric and orthogonal, $\| \mathcal{ {E}} \|_2 \leq  \epsilon$ and ${\rm rank} (\mathcal{ \widetilde{R} }) \leq 2m.$
\end{proposition}

\begin{proof}
We have
\begin{eqnarray*}
|\mathcal{H}|^{-1}\mathcal{A} =  |\mathcal{H}|^{-1} \mathcal{H}  +  \underbrace{ |\mathcal{H}|^{-1}\mathcal{ {R}}}_{ =:\mathcal{ \widetilde{R}} },
\end{eqnarray*}
where we have by (\ref{eqn:matrix_H}) that ${\rm rank} (\widetilde{R})={\rm rank}(\mathcal{R}) \leq 2m.$

Now, we know by (\ref{eqn:matrix_Mat}) that
\begin{eqnarray*}
|\mathcal{H}|&=&\mathcal{Q} 
\begin{bmatrix} 
\sqrt{\Sigma^2 + \alpha^2 I_{mn}}  & \\
   &  \sqrt{\Sigma^2 + \alpha^2 I_{mn} }
\end{bmatrix}
  \mathcal{Q}^T ,
  \end{eqnarray*}
  where $\mathcal{Q}$ is orthogonal and  $\Sigma$ is a diagonal matrix containing the singular values of $T$.
Thus,
\begin{eqnarray*}
|\mathcal{H}|^{-1} \mathcal{H} &=&\mathcal{Q} 
\begin{bmatrix} 
I_{mn}  & \\
   &  - I_{mn}
\end{bmatrix}
  \mathcal{Q}^T,
\end{eqnarray*}
which is both real symmetric and orthogonal. The proof is complete.
\end{proof}

\begin{corollary}\label{cor:ideal_precon}
Let $\mathcal{A}, \mathcal{H} \in \mathbb{R}^{2mn \times 2mn}$ be defined by (\ref{eqn:matrix_A}) and  (\ref{eqn:matrix_H}), respectively. Then,
\[
\{ |\mathcal{H}|^{-1}\mathcal{A}  \}_n \sim_{\lambda}  \psi_{ { I_m }},
\]
over the domain  $\widetilde D$ with $D=[0,2\pi]$ and $p=-2\pi$, where $\psi_{I_m }$ is defined in (\ref{def-psi}).
\end{corollary}

In other words, Proposition \ref{eqn:ideal_precon} and Corollary \ref{cor:ideal_precon} show that the preconditioner $|\mathcal{H}|$ can render the eigenvalue clusters around $\pm 1$, providing a guide to designing effective preconditioners for $\mathcal{A}$. In fact, we can show that $|\mathcal{H}|$ can achieve mesh-independent convergence via the following theorem.

{

\begin{theorem}\label{thm:ideal_precon_properties}
Let $\mathcal{A}, |\mathcal{H}| \in \mathbb{R}^{2mn \times 2mn}$ be defined by (\ref{eqn:matrix_A}) and  (\ref{eqn:abs_ideal_matrix_H}), respectively. Then,
\begin{enumerate}[label=(\alph*)]
\item $\lambda=1$ as eigenvalue of $|\mathcal{H}|^{-1}\mathcal{A}$ has multiplicity at least {$nm-2m$} and $\lambda=-1$ as eigenvalue of $|\mathcal{H}|^{-1}\mathcal{A}$ has multiplicity at least {$nm-2m$};
\item there are at most $4m$ eigenvalues of $|\mathcal{H}|^{-1}\mathcal{A}$ not equal to $\pm1$;
\item the spectrum of $|\mathcal{H}|^{-1}\mathcal{A}$ belongs to the symmetric set $(-\frac{3}{2},-\frac{1}{2}) \cup (\frac{1}{2},\frac{3}{2})$ well separated from zero, with at most $m$ lying in the interval $(-\frac{3}{2},-1)$, with at most $m$ lying in the interval $(-1,-\frac{1}{2})$, with at most $m$ lying in the interval $(\frac{1}{2},1)$, and with at most $m$ lying in the interval $(1,\frac{3}{2})$.
\end{enumerate}
\end{theorem}

\begin{proof}
We first observe from (\ref{eqn:matrix_H}) that 
$\mathcal{A}
=\mathcal{H}
+
\mathcal{R}.$
Now, we know that $\mathcal{H}$ is Hermitian and invertible with eigenvalues given by that of $\sqrt{ \mathcal{T}^T \mathcal{T} + \alpha^2 I_{mn} }$ and that of $-\sqrt{ \mathcal{T} \mathcal{T}^T + \alpha^2 I_{mn} }$. Therefore, $|\mathcal{H}|$ is invertible and $|\mathcal{H}|^{-1}\mathcal{A}$ is similar to $|\mathcal{H}|^{-1/2}\mathcal{H}|\mathcal{H}|^{-1/2} + \mathcal{\widehat{R}}$, where $ \mathcal{\widehat{R}} = |\mathcal{H}|^{-1/2} \mathcal{R} |\mathcal{H}|^{-1/2}$.

Considering the Schur decomposition, we deduce that $|\mathcal{H}|^{-1/2}\mathcal{H}|\mathcal{H}|^{-1/2}$ has the eigenvalue $\lambda=1$ with multiplicity $nm$ and the eigenvalue $\lambda=-1$ with multiplicity $nm$. Furthermore, $\mathcal{R} $ has rank $2m$ with signature $(m,2nm-2m,m)$ and the same is true for $\mathcal{\widehat{R}}$ due to the Sylvester inertia law.

Now, suppose $A=|\mathcal{H}|^{-1/2}\mathcal{H}|\mathcal{H}|^{-1/2}$ and $B=A+\mathcal{\widehat{R}}$. By {the well-known Cauchy interlacing theorem \cite[Corollary III.1.5]{Bhatia1997}}, the eigenvalues of $B$ are ordered as follows:
{
\begin{eqnarray} \nonumber
&&\lambda_1(B) \leq  \cdots \leq \lambda_m(B) \leq -1~(\textrm{repeated $nm-2m$ times}) \leq \lambda_{nm-m+1}(B) \leq \cdots \leq \\ \label{first-g}
&&\lambda_{nm}(B) \leq  \lambda_{nm+1}(B) \leq \cdots  \leq \lambda_{nm+m}(B)  \leq 1~(\textrm{repeated $nm-2m$ times})  \leq \\ \nonumber
&& \lambda_{2nm-m+1}(B)  \leq \cdots \leq \lambda_{2nm}(B),
\end{eqnarray}
}with 
\begin{equation}\label{second-g}
\lambda_1(B) > -\frac{3}{2},\quad \lambda_{nm}(B) < -\frac{1}{2},\quad \lambda_{nm+1}(B) > \frac{1}{2},\quad \lambda_{2nm}(B)<  \frac{3}{2} 
\end{equation}
since $\lambda_{{\max}}(\mathcal{\widehat{R}})<\frac{1}{2}, \lambda_{{\min}}(\mathcal{\widehat{R}})>-\frac{1}{2}$. Indeed, the last inequalities hold true because 
\[
\lambda_{{\max}}(\mathcal{\widehat{R}}) \leq \frac{\lambda_{\max}(\mathcal{R})}{\lambda_{\min}(|H|)} =
\frac{\alpha}{2} \cdot \frac{1}{\lambda_{{\min}}(\sqrt{ \mathcal{T} \mathcal{T}^T + \alpha^2 I_{mn} })} < {\frac{\alpha}{2}} \cdot \frac{1}{\lambda_{{\min}}(\sqrt{ \alpha^2 I_{mn} })} = \frac{1}{2},
\]
with $\sigma_{{\min}}(\mathcal{T})>0$ (i.e., $\mathcal{T}$ is invertible). The inequality on $ \lambda_{{\min}}(\mathcal{\widehat{R}})$ is shown using similar arguments and in fact we have
\[
\lambda_{{\min}}(\mathcal{\widehat{R}}) \geq \frac{\lambda_{\min}(\mathcal{R})}{\lambda_{\max}(|H|)} =
-\frac{\alpha}{2} \cdot \frac{1}{\lambda_{{\min}}(\sqrt{ \mathcal{T} \mathcal{T}^T + \alpha^2 I_{mn} })} > - {\frac{\alpha}{2}} \cdot \frac{1}{\lambda_{{\min}}(\sqrt{ \alpha^2 I_{mn} })} = -\frac{1}{2},
\]
again because $\mathcal{T}$ is nonsingular. Hence, looking carefully at the inequalities in (\ref{first-g}) and (\ref{second-g}), the proof is concluded.
\end{proof}
}

{As a consequence of Theorems \ref{thm:ideal_precon_properties} \& {\cite[Theorem 6.13]{ElmanSilvesterWathen2004}}}, we conclude that the MINRES with $|\mathcal{H}|$ as a preconditioner can achieve mesh-independent convergence, i.e., a convergence rate independent of both the meshes and the regularization parameter. Namely, $|\mathcal{H}|$ can be regarded as an ideal preconditioner, due to the fact that $\Sigma(|\mathcal{H}|^{-1}\mathcal{A}) \subset (-\frac{3}{2},-\frac{1}{2}) \cup (\frac{1}{2},\frac{3}{2})$ which is very close to $\{\pm 1\}$, leading to the theorem below on the convergence of MINRES.}

{
\begin{theorem}\label{thm:convergence_ideal}
Let $\mathcal{A}, |\mathcal{H}| \in \mathbb{R}^{2mn \times 2mn}$ be defined by (\ref{eqn:matrix_A}) and  (\ref{eqn:abs_ideal_matrix_H}), respectively. Denoted by ${\bf r}_{k}$ the $k$th iteration residual produced by MINRES when applied to the linear system (\ref{eqn:main_system}). After $2k$ steps of MINRES, the residual ${\bf r}_{2k}$ satisfies the bound 
\begin{equation*}
||{\bf r}_{2k}||_{|\mathcal{H}|^{-1}}\leq 2\left( \frac{1}{2}\right)^{2k}||{\bf r}_{0}||_{|\mathcal{H}|^{-1}}.
\end{equation*}
\end{theorem}
}

{Despite the fact that $|\mathcal{H}|$ is an excellent preconditioner, its direct application has the drawback of being in general computationally expensive.} In what follows, we will show that both $\mathcal{P}_{S}$ defined by (\ref{eqn:matrix_P}) and $\mathcal{P}_{G}$ defined by (\ref{eqn:matrix_G}) are good preconditioners since they are close to $|\mathcal{H}|$ in the sense that their difference can be expressed as the sum of a small norm matrix and a low rank matrix. {In this context, we acknowledge that our analysis is reminiscent of that in \cite{serra_indef}, which contains the first symbol based positive definite preconditioning for indefinite Toeplitz problems. However, the approach considered here is different and richer for at least two reasons: a) in the present context the structure is of block type and hence the associated symbols are matrix-valued; b) the involved structure contains Toeplitz matrices, but is not in itself purely of Toeplitz type.}

\subsubsection{Block circulant based preconditioner}

The following theorem accounts for the preconditioning effect of $\mathcal{P}_{S}$.

\begin{theorem}\label{thm:absP_main}
Let $\mathcal{A}, \mathcal{P}_{S} \in \mathbb{R}^{2mn \times 2mn}$ be defined by (\ref{eqn:matrix_A}) and  (\ref{eqn:matrix_P}), respectively. Then, 
\[
\mathcal{P}_{S}^{-1}\mathcal{A} = \mathcal{\widetilde{Q}} + \mathcal{ \widetilde{R}}_0,
\]
where $ \mathcal{\widetilde{Q}}$ is both real symmetric and orthogonal and ${\rm rank} (\mathcal{ \widetilde{R} }_0) \leq 8m.$
\end{theorem}

\begin{proof}
Let $s(\mathcal{A})=
\begin{bmatrix} 
\alpha I_{mn}  & \mathcal{S}^T \\
 \mathcal{S}  & -\alpha I_{mn} 
\end{bmatrix}.$ Note that $|s(\mathcal{A})|=\sqrt{s(\mathcal{A})^2}=\mathcal{P}_{S}$.

Simple calculations show that
\[
s(\mathcal{A}) - \mathcal{A} = \underbrace{
\begin{bmatrix} 
  &   (\mathcal{S}-\mathcal{A})^T \\
\mathcal{S}-\mathcal{A} & 
 \end{bmatrix}
+
\begin{bmatrix} 
-\frac{\alpha}{2} \mathbf{\hat{e}}_n   \mathbf{\hat{e}}_n^T\otimes I_m  & \\
   & \frac{\alpha}{2}  \mathbf{\hat{e}}_1   \mathbf{\hat{e}}_1^T \otimes I_m 
\end{bmatrix}}_{=:\mathcal{ \widetilde{R}}_1},
\]
where ${\rm rank} (\mathcal{ \widetilde{R}}_1) \leq 6m +  2m = 8m.$
Thus, we have 
\begin{eqnarray*}
\mathcal{P}_{S}^{-1}\mathcal{A} = |s(\mathcal{A})|^{-1}\mathcal{A} &=& |s(\mathcal{A})|^{-1}(s(\mathcal{A})-{\mathcal{ \widetilde{R}}_1})\\
&=& \underbrace{|s(\mathcal{A})|^{-1}s(\mathcal{A})}_{=: \mathcal{\widetilde{Q}}}-\underbrace{|s(\mathcal{A})|^{-1}\mathcal{ \widetilde{R}}_1}_{=:\mathcal{ \widetilde{R}}_0},
\end{eqnarray*}
where $\mathcal{\widetilde{Q}}$ can be shown both real symmetric and orthogonal by direct calculations (or using similar arguments in the proof of Preposition \ref{eqn:ideal_precon} for $|\mathcal{H}|^{-1}\mathcal{H}$) and ${\rm rank} (\mathcal{\widetilde{R}}_0) = {\rm rank} (\mathcal{ \widetilde{R} }_1) \leq 8m$. The proof is complete.
\end{proof}

As a consequence of Theorem \ref{thm:absP_main} and \cite[Corollary 3]{BRANDTS20103100}, we know that the preconditioned matrix $ \mathcal{P}_{S}^{-1} \mathcal{A}  $ has clustered eigenvalues at $\pm1$, with a number of outliers (i.e., $16m$) independent of $n$. Hence, when MINRES is used with $ \mathcal{P}_{S}$ as a preconditioner, fast convergence can be obtained.

The corollary below follows directly from Theorem \ref{thm:absP_main}.

\begin{corollary}
Let $\mathcal{A}, \mathcal{P}_{S} \in \mathbb{R}^{2mn \times 2mn}$ be defined by (\ref{eqn:matrix_A}) and  (\ref{eqn:matrix_P}), respectively. Then,
\[
\{ \mathcal{P}_{S}^{-1}\mathcal{A}  \}_n \sim_{\lambda}  \psi_{ { I_m }},
\]
over the domain  $\widetilde D$ with $D=[0,2\pi]$ and $p=-2\pi$, where $\psi_{I_m }$ is defined in (\ref{def-psi}).
\end{corollary}

\subsubsection{Block Tau preconditioner}

First, we provide the following lemma and proposition, which will be used to show the preconditioning effect of $\mathcal{P}_{G}$.

\begin{lemma}\label{lemma:rankT_G}
Let $\mathcal{T}, G \in \mathbb{R}^{mn \times mn} $ be defined in (\ref{eqn:matrix_T}) and (\ref{eqn:matrix_tau}), respectively. Then,
\[
{\rm rank}\Big( (\mathcal{T}^T\mathcal{T} +\alpha^2 I_{mn} )^{K}- (G^TG +\alpha^2 I_{mn} )^{K} \Big) \leq 4Km,
\]
for any positive integer $K$ provided that $n>4Km$.
\end{lemma}
\begin{proof}
Routine computations give
\begin{eqnarray*}
(\mathcal{T}^T\mathcal{T} +\alpha^2 I_{mn} ) - ( {G}^T{G} +\alpha^2 I_{mn} )
&=&
\mathcal{T}^T\mathcal{T} - G^T G\\
&=& 
\begin{bmatrix}
L_m^2 &  &  & & &   & \\
   &     && & & &  \\
      &   & &  & &  &\\
      &   & &  & &  &\\
   & & &&   &  -L_m^2 & 2L_m \\
 &  && & &  2L_m   & - 4I_m\\
\end{bmatrix}.
\end{eqnarray*}
Exploiting such a simple structure of $\mathcal{T}^T\mathcal{T} - G^T G$, we can further show the following by using a computational lemma given in \cite[Lemma 3.11]{doi:10.1137/080720280}
\begin{eqnarray*}\label{eqn:structure_TP}
&&(\mathcal{T}^T\mathcal{T} +\alpha^2 I_{mn} )^{n_a} ( \mathcal{T}^T\mathcal{T} - S^T S )({S}^T{S} +\alpha^2 I_{mn} )^{n_b} \\
&=& 
\begin{bmatrix}
* & \cdots & * & & & & * & \cdots & *\\
\vdots &  & \vdots & & & & \vdots &  & \vdots\\
* & \cdots & * & &  && * & \cdots & *\\
 &  &  & &  &&  & & \\
* & \cdots & * & &  && * & \cdots & *\\
\vdots &  & \vdots & & & & \vdots &  & \vdots\\
* & \cdots & * & & & & * & \cdots & *\\
\end{bmatrix}
\end{eqnarray*}
for integer values $n_a$ and $n_b$, where $*$ represents a nonzero entry. Namely, $(\mathcal{T}^T\mathcal{T} +\alpha^2 I_{mn} )^{ n_a} ( \mathcal{T}^T\mathcal{T} - S^T S  )({S}^T{S} +\alpha^2 I_{mn} )^{ n_b}  $ is a block matrix with two blocks in its four corners, respectively, and each block is of size $(n_a+1)2m \times (n_b +1)2m$. Thus,
\begin{eqnarray*}
 &&(\mathcal{T}^T\mathcal{T} +\alpha^2 I_{mn} )^{K}-({S}^T{S} +\alpha^2 I_{mn} )^{K} \\
  &=& \sum_{i=0}^{K-1} (\mathcal{T}^T\mathcal{T} +\alpha^2 I_{mn} )^{K-i-1} ( \mathcal{T}^T\mathcal{T} - S^T S ) ({S}^T{S} +\alpha^2 I_{mn} )^i
\end{eqnarray*}
is also a block matrix with four blocks in its four corners and each of them is of size $2Km \times 2Km$, provided that $n>4Km$. Hence, we have ${\rm rank}\Big( (\mathcal{T}^T\mathcal{T} +\alpha^2 I_{mn} )^{K}- (S^TS +\alpha^2 I_{mn} )^{K} \Big) \leq 4Km.$

\end{proof}

\begin{lemma}\label{lem:mainTabsG}
Let $\mathcal{T}, G \in \mathbb{R}^{mn \times mn} $ be defined in (\ref{eqn:matrix_T}) and (\ref{eqn:matrix_G}), respectively. Then, for any $\epsilon >0$ there exists an integer $K$ such that for all $n>4Km$
\[
\Big(\sqrt{ \mathcal{T}^T\mathcal{T} +\alpha^2 I_{mn}    } \Big)^{-1}- \Big( \sqrt{S^T S +  \alpha^2 I_{mn}} 
\Big)^{-1} = \mathcal{\overline{E}}_1 + \mathcal{\overline{R}}_1,
\]
where $\|\mathcal{\overline{E}}_1\|_2 \leq \epsilon$ and ${\rm rank} (\mathcal{\overline{R}}_1) \leq 4Km $. 
\end{lemma}

\begin{proof}
Let $\widehat{f}(x)=x^{-1/2}$ and $\widehat{F}(x)$ be defined in (\ref{def:function_F}). By Theorem \ref{thm:Bernstein}, there exists a polynomial $p_K$ with degree less than or equal to $K$ such that
\begin{eqnarray*}
 \Big\| \Big(\sqrt{ \mathcal{T}^T\mathcal{T} +\alpha^2 I_{mn}    } \Big)^{-1} - p_K \Big( \mathcal{T}^T\mathcal{T} +\alpha^2 I_{mn}   \Big) \Big\|_2  &=& \max_{x \in \sigma (  \mathcal{T}^T\mathcal{T} +\alpha^2 I_{mn}    ) } \vert \widehat{F}(x) - p_K(x) \vert\\
& \leq &\| \widehat{F} -p_K(x)\|_{\infty} \\
&\leq &\frac{2M (\mathcal{X}_{\mathcal{T}^T\mathcal{T} +\alpha^2 I_{mn}   }  )}{\mathcal{X}_{  \mathcal{T}^T\mathcal{T} +\alpha^2 I_{mn}   }-1}\cdot \frac{1}{\mathcal{X}_{ \mathcal{T}^T\mathcal{T} +\alpha^2 I_{mn}    }^K},
\end{eqnarray*}

\begin{eqnarray*}
 \Big\| \Big(\sqrt{ {G}^T{G} +\alpha^2 I_{mn}    } \Big)^{-1} - p_K \Big( {G}^T{G} +\alpha^2 I_{mn}   \Big) \Big\|_2  &=& \max_{x \in \sigma (  {G}^T{G} +\alpha^2 I_{mn}    ) } \vert \widehat{F}(x) - p_K(x) \vert\\
& \leq &\| \widehat{F} -p_K(x)\|_{\infty} \\
&\leq &\frac{2M (\mathcal{X}_{{G}^T{G} +\alpha^2 I_{mn}  }  )}{\mathcal{X}_{ {G}^T{G} +\alpha^2 I_{mn}   }-1}\cdot \frac{1}{\mathcal{X}_{ {G}^T{G} +\alpha^2 I_{mn}    }^K},
\end{eqnarray*}
where
\[
1 < \mathcal{X}_{ \mathcal{T}^T\mathcal{T} +\alpha^2 I_{mn} } < \frac{\sqrt{\kappa_{\mathcal{T}^T\mathcal{T} +\alpha^2 I_{mn}}}+1}{\sqrt{\kappa_{\mathcal{T}^T\mathcal{T} +\alpha^2 I_{mn}}}-1},\qquad 1 < \mathcal{X}_{ {S}^T{S} +\alpha^2 I_{mn} } < \frac{\sqrt{\kappa_{{G}^T{G} +\alpha^2 I_{mn}}}+1}{\sqrt{\kappa_{{G}^T{G} +\alpha^2 I_{mn}}}-1},
\]
and $\kappa_{ \mathcal{T}^T\mathcal{T} +\alpha^2 I_{mn}  }$ and $\kappa_{ {G}^T{G} +\alpha^2 I_{mn} }$ denote the condition numbers of $ \mathcal{T}^T\mathcal{T} +\alpha^2 I_{mn} $ and $ {G}^T{G} +\alpha^2 I_{mn} $, respectively. Thus, for any $\epsilon >0$ there exists an integer $K$ such that
\[
\Big\| \Big(\sqrt{ \mathcal{T}^T\mathcal{T} +\alpha^2 I_{mn}    } \Big)^{-1} - p_K \Big( \mathcal{T}^T\mathcal{T} +\alpha^2 I_{mn}   \Big)  \Big\|_2 \leq \epsilon
\]
and 
\[
\Big\| \Big(\sqrt{ {G}^T{G} +\alpha^2 I_{mn}    } \Big)^{-1} - p_K \Big( {G}^T{G} +\alpha^2 I_{mn}   \Big) \Big\|_2 \leq \epsilon.
\]

Also, we have 
\begin{eqnarray*}
&&p_K \Big( \mathcal{T}^T\mathcal{T} +\alpha^2 I_{mn}   \Big) - p_K \Big( {G}^T{G} +\alpha^2 I_{mn}   \Big)  \\
&=&\underbrace{\sum_{i=0}^{K} a_i \Big( ( \mathcal{T}^T\mathcal{T} +\alpha^2 I_{mn}   )^{i} - ( {G}^T{G} +\alpha^2 I_{mn}   )^{i} \Big)}_{=:\mathcal{\overline{R}}_1}.
\end{eqnarray*}

By Lemma \ref{lemma:rankT_G}, we know that $\mathcal{{R}}$ has the same sparsity structure as that of $( \mathcal{T}^T\mathcal{T} +\alpha^2 I_{mn}   )^{K} - ( {G}^T{G} +\alpha^2 I_{mn}   )^{K} $. Consequently, we have deduced that ${\rm rank}(\mathcal{\overline{R}}_1) \leq 4Km$. 

Then, we obtain
\begin{eqnarray*}
&&\Big(\sqrt{ \mathcal{T}^T\mathcal{T} +\alpha^2 I_{mn}    } \Big)^{-1} - \Big( \sqrt{G^T G +  \alpha^2 I_{mn}} 
\Big)^{-1} \\
&=& \underbrace{ \Big(\sqrt{ \mathcal{T}^T\mathcal{T} +\alpha^2 I_{mn}    } \Big)^{-1} - p_K \Big( \mathcal{T}^T\mathcal{T} +\alpha^2 I_{mn}   \Big) + p_K \Big(  {G}^{T}G +\alpha^2 I_{mn}   \Big) - \Big( \sqrt{G^T G +  \alpha^2 I_{mn}} 
\Big)^{-1} }_{=:\mathcal{\overline{E}}_1}\\
&& + \underbrace{ p_K \Big( \mathcal{T}^T\mathcal{T} +\alpha^2 I_{mn}   \Big) - p_K \Big( {G}^T{G} +\alpha^2 I_{mn}   \Big) }_{=\mathcal{\overline{R}}_1},
\end{eqnarray*}
where $\|\mathcal{E}_0\|_2 \leq 2\epsilon$ and ${\rm rank}(\mathcal{\overline{R}}_1) \leq 4Km $. Therefore, the proof is concluded.
\end{proof}

Following the same argument in the proof of Lemma \ref{lem:mainTabsG}, we obtain the following dual result concerning $\mathcal{T}\mathcal{T}^T$ instead of $\mathcal{T}^T\mathcal{T}$:

\begin{lemma}\label{lem:mainTabsG_2}
Let $\mathcal{T}, G \in \mathbb{R}^{mn \times mn} $ be defined in (\ref{eqn:matrix_T}) and (\ref{eqn:matrix_G}), respectively. Then, for any $\epsilon >0$ there exists an integer $K$ such that for all $n>4Km$
\[
\Big(\sqrt{ \mathcal{T}\mathcal{T}^T +\alpha^2 I_{mn}    } \Big)^{-1}- \Big( \sqrt{G G^T +  \alpha^2 I_{mn}} 
\Big)^{-1} = \mathcal{\overline{E}}_2 + \mathcal{\overline{R}}_2,
\]
where $\|\mathcal{\overline{E}}_2 \|_2 \leq \epsilon$ and ${\rm rank} (\mathcal{\overline{R}}_2) \leq 4Km $. 
\end{lemma}

As last, we are ready to provide the following theorem which guarantees the effectiveness of $\mathcal{P}_{G}$.

\begin{theorem}\label{thm:absG_main}
Let $\mathcal{A}, \mathcal{P}_{G} \in \mathbb{R}^{2mn \times 2mn}$ be defined by (\ref{eqn:matrix_A}) and  (\ref{eqn:matrix_G}), respectively. Then, for any $\epsilon >0$ there exists an integer $K$ such that for all $n>4Km$
\[
\mathcal{P}_{G}^{-1}\mathcal{A} = \mathcal{\overline{Q}} + \mathcal{ \overline{E} }_0 + \mathcal{ \overline{R}}_0,
\]
where $ \mathcal{\overline{Q}}$ is both real symmetric and orthogonal, $\| \mathcal{ \overline{E}}_0 \|_2 \leq  \epsilon$, and ${\rm rank} (\mathcal{ \overline{R} }_0) \leq (8K+2)m.$
\end{theorem}

\begin{proof}
By (\ref{eqn:matrix_H}) and Lemmas \ref{lem:mainTabsG} \& \ref{lem:mainTabsG_2}, we have 
\begin{eqnarray*}
&&\mathcal{P}_{G}^{-1}\mathcal{A} \\
&=& ( \mathcal{P}_{G}^{-1} - |\mathcal{H}|^{-1} + |\mathcal{H}|^{-1} ) ( \mathcal{H} + \mathcal{{R}} ) \\
 & =&
 \cred{
 \textrm{blockdiag} \bigg(
 \Big(\sqrt{G^T G +  \alpha^2 I_{mn}} \Big)^{-1} - \Big(\sqrt{\mathcal{T}^T \mathcal{T} +  \alpha^2 I_{mn}} \Big)^{-1},
 } \\
&&  \cred{
\Big(\sqrt{G G^T +  \alpha^2 I_{mn}} \Big)^{-1}  -\Big(\sqrt{\mathcal{T} \mathcal{T}^T +  \alpha^2 I_{mn}} \Big)^{-1}\bigg)
 }
\\
&&\times ( \mathcal{H} + \mathcal{{R}} ) + |\mathcal{H}|^{-1} \mathcal{{R}} + |\mathcal{H}|^{-1} \mathcal{H}\\
&=&
\underbrace{\begin{bmatrix} 
-\mathcal{\overline{E}}_1  &  \\
  & -\mathcal{\overline{E}}_2
\end{bmatrix} ( \mathcal{H} + \mathcal{{R}} )  }_{ =:\mathcal{\overline{E}}_0 }
+
\underbrace{\begin{bmatrix} 
-\mathcal{\overline{R}}_1  &  \\
  & -\mathcal{\overline{R}}_2
\end{bmatrix} ( \mathcal{H} + \mathcal{{R}} )+ |\mathcal{H}|^{-1} \mathcal{{R}} }_{ =: \mathcal{\overline{R}}_0 }
+ 
\underbrace{ |\mathcal{H}|^{-1} \mathcal{H} }_{=:\mathcal{\overline{Q}}},
\end{eqnarray*}
where
$\mathcal{\overline{Q}}$ is both real symmetric and orthogonal by Preposition \ref{eqn:ideal_precon}, 
\begin{eqnarray*}
{\rm rank} ( \mathcal{\overline{R}}_0 ) &\leq&  {\rm rank}(\mathcal{\overline{R}}_1 ) + {\rm rank}( \mathcal{\overline{R}}_2 ) +  {\rm rank}(\mathcal{{R}} ) \\
&\leq& (8K+2)m
\end{eqnarray*}
and
\begin{eqnarray*}
\|\mathcal{ \overline{E}}_0\|_2 &\leq& \max{\{ \| \mathcal{\overline{E}}_1\|_2, \| \mathcal{\overline{E}}_2\|_2 \}}  \|  \mathcal{H} + \mathcal{{R}} \|_2
\\
&\leq&(  \|  \mathcal{H}\|_2 + \| \mathcal{{R}} \|_2 ) \epsilon \\
&\leq& \Big( \|  \mathcal{H}\|_2 + \frac{\alpha}{2} \Big) \epsilon.
\end{eqnarray*}

Lastly, the following inequalities hold
\begin{eqnarray*}
\| \mathcal{H} \|_2  \leq { \| \sqrt{ \mathcal{T}^T \mathcal{T} + \alpha^2 I_{mn}  }\|_2 }&=& { \Big\|  \sqrt{ T_{(m,n)} [h]^T T_{(m,n)} [h]  + \alpha^2 I_{mn} } \Big\|_2 }\\
&\leq& { \sqrt{  \sigma^2(T_{(m,n)} [h]) + \alpha^2 }} \\
&\leq& {  \max_{\theta\in [0,2\pi]}   \sqrt{  \sigma^2_{\max} (h(\theta)) + \alpha^2  }}\\
&=& {\max_{\theta\in [0,2\pi]}  \sqrt{ \sigma^2_{\max} (h(\theta)) +  \frac{T^4}{n^4\gamma}} }
\end{eqnarray*}
by using the general inequality in \cite[Corollary 4.2]{Serra_Tilli_2002} where the Schatten $p$ norm with $p=\infty$ equates the spectral norm $\|\cdot\|_2$. Hence, $\| \mathcal{H} \|_2$ is uniformly bounded with respect to $n$, since for any $\epsilon>0$ the quantity $\frac{T^4}{n^4\gamma}$ is bounded by $\epsilon$ choosing $n$ large enough and in any case $\frac{T^4}{n^4\gamma}\le \frac{T^4}{\gamma}$ uniformly with respect to $n$. Therefore the proof is concluded.

\end{proof}

As a consequence of Theorem \ref{thm:absG_main} and \cite[Corollary 3]{BRANDTS20103100}, we know that for sufficiently large $n$ the preconditioned matrix-sequence $\{ \mathcal{P}_{G}^{-1} \mathcal{A}  \}_n$ has clustered eigenvalues around $\pm1$ with a number of outliers independent of $n$. Hence, when MINRES is used with $ \mathcal{P}_{G}$ as a preconditioner, fast convergence can be obtained independently of $n$.

The corollary below follows directly from Theorem \ref{thm:absG_main}.

\begin{corollary}
Let $\mathcal{A}, \mathcal{P}_{G} \in \mathbb{R}^{2mn \times 2mn}$ be defined by (\ref{eqn:matrix_A}) and  (\ref{eqn:matrix_G}), respectively. Then,
\[
\{ \mathcal{P}_{G}^{-1}\mathcal{A}  \}_n \sim_{\lambda}  \psi_{ { I_m }},
\]
over the domain  $\widetilde D$ with $D=[0,2\pi]$ and $p=-2\pi$, where $\psi_{I_m }$ is defined in (\ref{def-psi}).
\end{corollary}

\subsubsection{Modified preconditioners}\label{subsubsec:ModifiedPs}

{
Following our spectral symbol matching strategy, the construction of both modified preconditioners $\widetilde{\mathcal{P}}_{S}$ and $\widetilde{\mathcal{P}}_{G}$ is motivated by the following approximation without requiring the fast diagonalizability of $\Delta_m$. 

Since the original preconditioners ${\mathcal{P}}_{S}$ and ${\mathcal{P}}_{G}$ satisfy 
\[
\{ {\mathcal{P}}_{S}  \}_n \sim_{\lambda}  \psi_{ { g }} \quad \textrm{and} \quad \{ {\mathcal{P}}_{G}  \}_n \sim_{\lambda}  \psi_{ { g }},
\]
where $g=\sqrt{|h|^2 + \alpha^2I_m } ~~ \textrm{with} ~~ h(\theta)=  L_m-2I_m e^{\mathbf{i}\theta} + L_m e^{2\mathbf{i}\theta}$ defined by (\ref{eqn:function_h}), we can approximate the spectral symbol
\begin{eqnarray*}
&&g\\
&=&\sqrt{|h|^2 + \alpha^2 I_m}\\
&=& \sqrt{2L_m^2 -8L_m\cos{\theta}+2L_m^2\cos{2\theta}+4I_m+\alpha^2 I_m}\\
&=& \sqrt{2\Big(I_m-\frac{\tau^2}{2}\Delta_m \Big)^2 -8\Big(I_m-\frac{\tau^2}{2}\Delta_m \Big)\cos{\theta}+2\Big(I_m-\frac{\tau^2}{2}\Delta_m \Big)^2\cos{2\theta}+4I_m+\alpha^2 I_m}\\
&=& \sqrt{(6 -8\cos{\theta}+2\cos{2\theta}+\alpha^2)I_m + (-2\tau^2 + 4\tau^2 \cos{\theta} -2\tau^2 \cos{2\theta})\Delta_m +\frac{\tau^4}{4}(2+2\cos{2\theta})\Delta_m^2}
\end{eqnarray*}
 via 
\begin{eqnarray*}
\widetilde{g} = \sqrt{6 -8\cos{\theta}+2\cos{2\theta}+\alpha^2}I_m - \frac{\tau^2}{2}\sqrt{2+2\cos{2\theta}}\Delta_m.
\end{eqnarray*}

Hence, we can regard $\widetilde{\mathcal{P}}_{S}$ and $\widetilde{\mathcal{P}}_{G}$ as satisfying
\[
\{ \widetilde{\mathcal{P}}_{S}  \}_n \sim_{\lambda}  \psi_{ \widetilde{g}} \quad \textrm{and} \quad \{ \widetilde{\mathcal{P}}_{G}  \}_n \sim_{\lambda}  \psi_{ \widetilde{g}}.
\] In matrix form, the corresponding preconditioners generated by $\widetilde{g}$ are precisely our proposed modified preconditioners $\widetilde{\mathcal{P}}_{S}$ and $\widetilde{\mathcal{P}}_{G}$ defined by (\ref{eqn:matrix_P_mod}) and (\ref{eqn:matrix_G_mod}), respectively.
}

\subsection{Implementation}\label{subsubsec:implementationP}

We begin by discussing the computation of $\mathcal{A}\mathbf{v}$ for any given vector $\mathbf{v}$. Since $\mathcal{A}$ is a sparse matrix, computing the matrix-vector product $\mathcal{A}\mathbf{v}$ only requires linear complexity of $\mathcal{O}(mn)$. Alternatively, due to the fact that $\mathcal{A}$ contains two block Toeplitz matrices, it is well-known that $\mathcal{A}\mathbf{v}$ can be computed in $\mathcal{O}(mn \log{n})$ operations using fast Fourier transforms and the required storage is of $\mathcal{O}(mn)$.

Now, we provide here one efficient way to invert both $\mathcal{P}_{S}$ and $\mathcal{P}_{G}$. Since circulant matrices are diagonalized by the discrete Fourier matrix $\mathbb{F}_n=\frac{1}{\sqrt{n}}[\theta_n^{(i-1)(j-1)}]_{i,j=1}^{n} \in \mathbb{C}^{n \times n}$ with $\theta_n =\exp(\frac{2\pi {\bf i}}{n})$ and ${\bf i}=\sqrt{-1}$, we can represent the matrices $S_n^{(1)}$ and $S_n^{(2)}$ defined by (\ref{eqn:matrix_S})  using their eigendecompositions $S_n^{(1)}=\mathbb{F}_n \Upsilon_n^{(1)} \mathbb{F}_n^*$ and $S_n^{(2)}=\mathbb{F}_n \Upsilon_n^{(2)} \mathbb{F}_n^*$, respectively. Note that both $\Upsilon_n^{(1)}$ and $\Upsilon_n^{(2)}$ are diagonal matrices.

 Also, recalled that $-\Delta_m$ is assumed SPD, its eigendecomposition is given by $-\Delta_m= \mathbb{U}_m \Omega_m \mathbb{U}_m^T$, where $\mathbb{U}_m$ is orthogonal and $\Omega_m$ is a diagonal matrix containing the eigenvalues of $-\Delta_m$. 

Hence, we can rewrite $\mathcal{P}_{S}$ from (\ref{eqn:matrix_P}) as follows:
\begin{eqnarray*}
&&\mathcal{P}_{S}\\
 &=& 
\begin{bmatrix} 
\sqrt{S S^T +  \alpha^2 I_{mn}} 
&  \\
  & \sqrt{S^T S +  \alpha^2 I_{mn}} 
\end{bmatrix}\\
&=&
\begin{bmatrix} 
(\mathbb{F}_n\otimes \mathbb{U}_m)\sqrt{ |\Lambda|^2  +  \alpha^2 I_{mn}}(\mathbb{F}_n \otimes \mathbb{U}_m) ^* 
&  \\
  & (\mathbb{F}_n\otimes \mathbb{U}_m)^* \sqrt{ |\Lambda|^2  +  \alpha^2 I_{mn}}(\mathbb{F}_n \otimes \mathbb{U}_m)   
\end{bmatrix}\\
&=&
\mathcal{U}
\begin{bmatrix} 
\sqrt{|\Lambda |^2 +  \alpha^2 I_{mn}} 
&  \\
  & \sqrt{ |\Lambda |^2 +  \alpha^2 I_{mn}} 
\end{bmatrix}
\mathcal{U}^*,
\end{eqnarray*}
where $\mathcal{U} = \begin{bmatrix} 
\mathbb{F}_n\otimes \mathbb{U}_m 
&  \\
  & (\mathbb{F}_n\otimes \mathbb{U}_m)^*  
\end{bmatrix}$ is a unitary matrix and $\Lambda = \Upsilon_n^{(1)} \otimes I_m + \frac{\tau^2}{2} \Upsilon_n^{(1)} \otimes \Omega_m$.

In each iteration of MINRES, it is required to compute $\mathcal{P}_{S}^{-1}\mathbf{y}$ for a given vector $\mathbf{y}$. The computations of $\mathbf{z} = \mathcal{P}_{S}^{-1}\mathbf{y}$ can be implemented via the following three steps:
\begin{eqnarray*}
&&\textrm{Step 1: Compute}~\widetilde{\mathbf{y}} =  \mathcal{U}^*\mathbf{y}; \\
&&\textrm{Step 2: Compute}~\widetilde{\mathbf{z}} = \begin{bmatrix} 
\sqrt{|\Lambda |^2 +  \alpha^2 I_{mn}}^{-1} 
&  \\
  & \sqrt{ |\Lambda |^2 +  \alpha^2 I_{mn}}^{-1}
\end{bmatrix} \widetilde{\mathbf{y}};\\
&&\textrm{Step 3: Compute}~{\mathbf{z}} =  \mathcal{U} \widetilde{\mathbf{z}}.
\end{eqnarray*}
When the spatial grid is uniformly partitioned, the orthogonal matrix $\mathbb{U}_m$ becomes the discrete sine matrix $\mathbb{S}_m$. In this case, both Steps 1 \& 3 can be computed efficiently via fast Fourier transforms and fast sine transforms in $\mathcal{O}(mn\log{n})$ operations. As for Step 2, the required computations take $\mathcal{O}(n m)$ operations since the matrix involved is a simple diagonal matrix.

The product $\mathcal{P}_{G}^{-1}\mathbf{y}$ for any vector $\mathbf{y}$ can be implemented following the above procedures in which $\mathbb{F}_n$ is replaced by $\mathbb{S}_n$.

{
We provide the following three-step procedures for the inversion computation of $\widetilde{\mathcal{P}}_{S}^{-1}\mathbf{y}$ (or $\widetilde{\mathcal{P}}_{G}^{-1}\mathbf{y}$) for any vector $\mathbf{y}$, which does not require the fast diagonalizability of $\Delta_m$.

Notice that
\begin{eqnarray*}
&&\widetilde{\mathcal{P}}_{S}\\\nonumber
&=&
\cred{\textrm{blockdiag} \bigg( \sqrt{{(S_n^{(1)})^T S_n^{(1)}} +  \alpha^2 I_{n}}  \otimes I_m - \frac{\tau^2}{2}\sqrt{(S_n^{(2)})^T S_n^{(2)}} \otimes \Delta_m,}\\ \nonumber
&& \cred{\sqrt{{S_n^{(1)} (S_n^{(1)})^T} +  \alpha^2 I_{n}}  \otimes I_m - \frac{\tau^2}{2}\sqrt{S_n^{(2)} (S_n^{(2)})^T} \otimes \Delta_m \bigg)}
\\
&=& 
\widetilde{\mathcal{U}}
\cred{
\textrm{blockdiag} \bigg(
\sqrt{|\Upsilon_n^{(1)}|^2 +  \alpha^2 I_{n}}   \otimes I_m - \frac{\tau^2}{2}|\Upsilon_n^{(2)}|  \otimes \Delta_m,} \\ \nonumber
&& \cred{ \sqrt{|\Upsilon_n^{(1)}|^2 +  \alpha^2 I_{n}}   \otimes I_m - \frac{\tau^2}{2}|\Upsilon_n^{(2)}|  \otimes \Delta_m \bigg)}
\widetilde{\mathcal{U}}^*,
\end{eqnarray*}
where $\widetilde{\mathcal{U}}= \begin{bmatrix} 
\mathbb{F}_n\otimes {I}_m 
&  \\
  & (\mathbb{F}_n\otimes {I}_m)^*  
\end{bmatrix}$ is a unitary matrix.

The computations of $\widetilde{\mathcal{P}}_{S}^{-1}\mathbf{y}$ can be implemented via the following three steps:
\begin{eqnarray*}
&&\textrm{Step 1: Compute}~\widetilde{\mathbf{y}} = \widetilde{\mathcal{U}}^*\mathbf{y}; \\
&&\textrm{Step 2: Compute}\\
\widetilde{\mathbf{z}} &=& \cred{ \textrm{blockdiag} \bigg(
\Big(\sqrt{|\Upsilon_n^{(1)}|^2 +  \alpha^2 I_{n}}   \otimes I_m - \frac{\tau^2}{2}|\Upsilon_n^{(2)}|  \otimes \Delta_m \Big)^{-1},} \\
&& \cred{ \Big(\sqrt{|\Upsilon_n^{(1)}|^2 +  \alpha^2 I_{n}}   \otimes I_m - \frac{\tau^2}{2}|\Upsilon_n^{(2)}|  \otimes \Delta_m \Big)^{-1} \bigg)
 \widetilde{\mathbf{y}};}\\
&&\textrm{Step 3: Compute}~{\mathbf{z}} =  \widetilde{\mathcal{U}} \widetilde{\mathbf{z}}.
\end{eqnarray*}
Both Steps 1 \& 3 can be computed efficiently via fast Fourier transforms in $\mathcal{O}(mn\log{n})$ operations. As for Step 2, the shifted Laplacian systems can be efficiently solved for example using the multigrid methods \cite{doi:10.1137/15M102085X}. For more details regarding such efficient implementation, we refer to \cite[Eq. 2.12]{liuWu_optimal}.

Similarly, the product $\widetilde{\mathcal{P}}_{G}^{-1}\mathbf{y}$ for any vector $\mathbf{y}$ can be implemented following the above procedures in which $\mathbb{F}_n$ is replaced by $\mathbb{S}_n$.

}

\section{Numerical examples}\label{sec:numerical}

We provide in this section a number of numerical examples regarding the asymptotic spectral distribution of $\{\mathcal{A}\}_n$ and its preconditioning methods.

All numerical experiments are carried out using MATLAB R2022b on a Mac Studio equipped with Apple M1 Ultra and 64GB RAM. 

\subsection{Asymptotic spectral distribution of $\{\mathcal{A}\}_n$}

	In order to support Corollary \ref{cor:A-distr}, we numerically show that for large enough $n$ the eigenvalues of $\mathcal{A}$ are approximately equal to the samples of $\psi_{|f|}$ over a uniform grid in $[-2\pi,2\pi]$, with the possible exception of a small number of outliers.
		
	We highlight the fact that the matrix $\mathcal{A}$ is symmetric for any $2mn$, so the quantities $\lambda_i(\mathcal{A})$ are real for $i=1,\dots,2mn$. In particular, we order the eigenvalues of $\mathcal{A}$ according to the evaluation of $\psi_{g}$ with $g=\sqrt{|h|^2 + \alpha^2I_m }$ on the following uniform grid in $[-2\pi,2\pi]$:
{ 
        \begin{equation}\label{eq:uniform_grid} 
               \theta_{i,n} = -2\pi+i\frac{4\pi}{2n}, \qquad i = 1, \dots, 2n, 
        \end{equation} 
in accordance with the formula (\ref{def-psi}) and with $D=[0,2\pi]$ and $p=2\pi$ so that $D_p=[-2\pi,2\pi]$. 

Now, in order to make the reasoning simpler and more clear we remind that $\psi_{g}$ is $m\times m$ Hermitian matrix-valued and with two main branches, one equal to $-|\psi|$ and one equal to $|\psi|$. 

Thus, in our experiments with a uniform mesh in both space and time, we compute the following:}

\begin{itemize} 

\item all the eigenvalues $\lambda_i(\mathcal{A})$, $i=1,\dots,2mn$. According to Corollary \ref{cor:A-distr}, half of the eigenvalues will be negative and half of them will be positive, so we order them nondecreasingly and we put the negative ones in correspondence with the grid point given in (\ref{eq:uniform_grid}) and belonging to $[-2\pi,0]$, while the positive one are put in correspondence with the grid point given in (\ref{eq:uniform_grid}) and belonging to $[0,2\pi]$; 

\item the quantities $-\lambda_j(g)(\theta_{i,n})$ for $j=1,\ldots,m$, $i=1,\ldots,n$ will be ordered nondecreasingly in correspondence with the grid points given in (\ref{eq:uniform_grid}) and belonging to $[-2\pi,0]$; 

\item the quantities $\lambda_j(g)(\theta_{i,n})$ for $j=1,\ldots,m$, $i=n+1,\ldots,2n$ will be ordered nondecreasingly in correspondence with the grid points given in (\ref{eq:uniform_grid}) and belonging to $[0,2\pi]$. 

\end{itemize} 

\begin{example}\label{example:one_d_eig}

We begin with the one-dimensional case in space. For illustration purposes, we choose the block size $m=32$. In Figures \ref{fig:comp_1D_gamma_4}-\ref{fig:comp_1D_gamma_8}, we observe an excellent matching between the spectrum of $\mathcal{A}$ and the evaluations of the eigenvalues of $\psi_{g}$ computed with the grid points given in (\ref{eq:uniform_grid}) with various $\gamma$, in accordance with the theoretical results. Given the asymptotic nature of Corollary \ref{cor:A-distr}, the matching improves as $n$ increases and in our experiments we consider quite moderate sizes such as $n=32,64,128$: the agreement is remarkable given the small dimensions and the fact that the proved distribution results are of asymptotical type. As described by Corollary \ref{cor:A-distr}, a number of outlying eigenvalues are present in each case.

\begin{figure}[h!]
     \centering
     \begin{subfigure}[b]{0.47\textwidth}
         \centering
         \includegraphics[width=\textwidth]{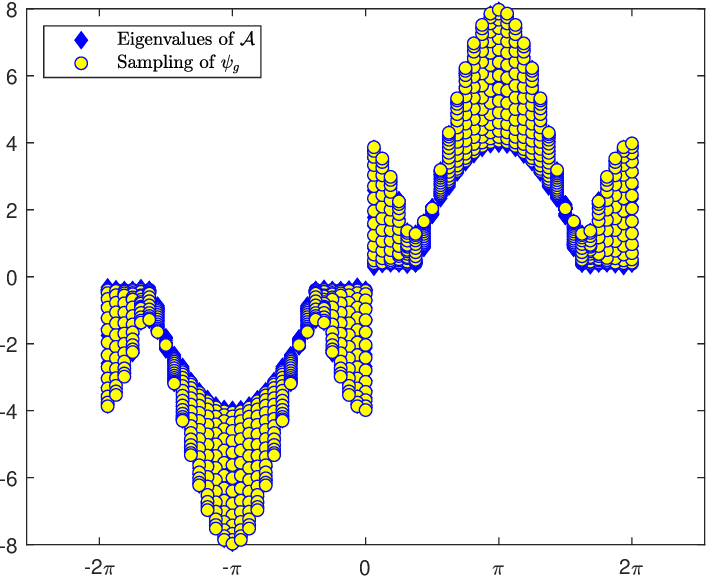}
         \caption{$n=32$}
     \end{subfigure}
     \hfill
     \begin{subfigure}[b]{0.47\textwidth}
         \centering
         \includegraphics[width=\textwidth]{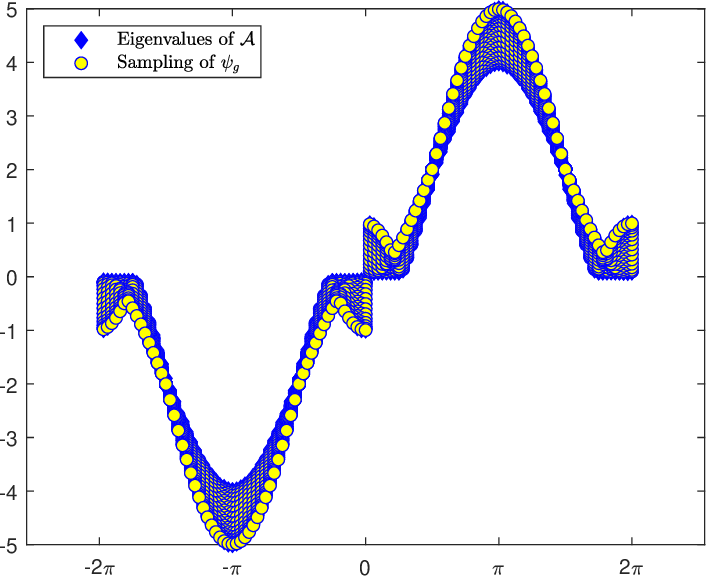}
         \caption{$n=64$}
     \end{subfigure}
     \hfill
     \begin{subfigure}[b]{0.47\textwidth}
         \centering
         \includegraphics[width=\textwidth]{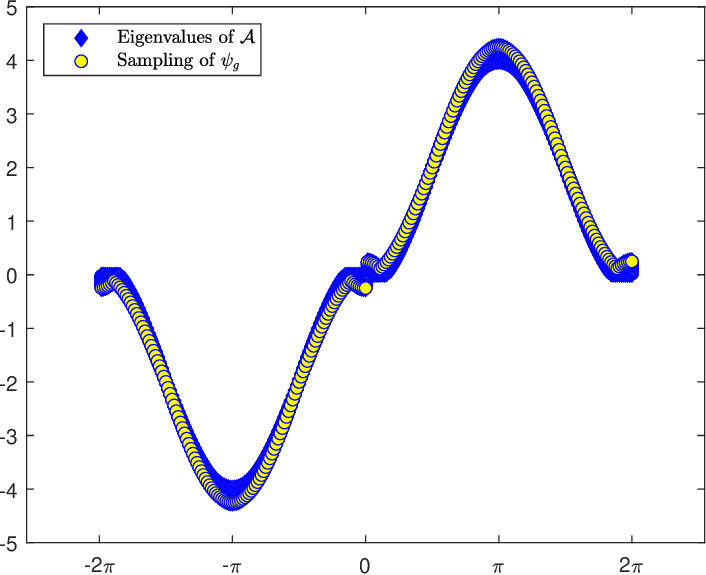}
         \caption{$n=128$}
     \end{subfigure}
        \caption{Comparison between the eigenvalues of $\mathcal{A}$ and the
sampling of $\psi_{g}$ for Example \ref{example:one_d_eig} with $m=15$ and $\gamma =10^{-4} $.}
        \label{fig:comp_1D_gamma_4}
\end{figure}

\begin{figure}[h!]
     \centering
     \begin{subfigure}[b]{0.47\textwidth}
         \centering
         \includegraphics[width=\textwidth]{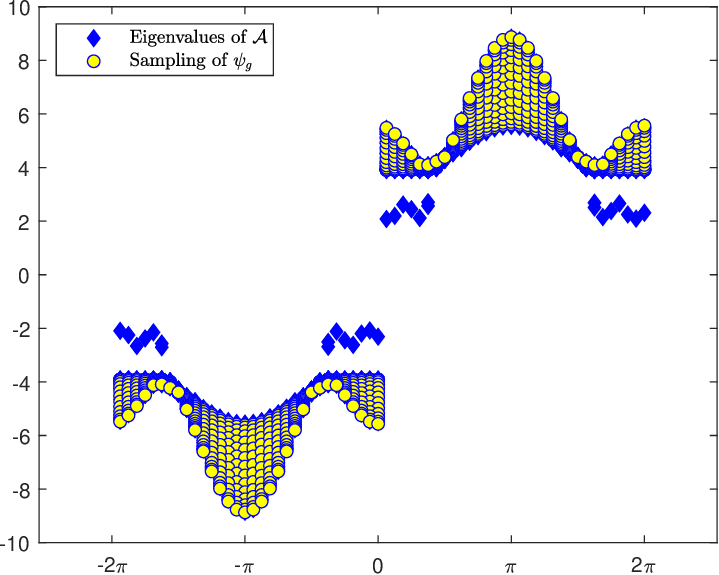}
         \caption{$n=32$}
     \end{subfigure}
     \hfill
     \begin{subfigure}[b]{0.47\textwidth}
         \centering
         \includegraphics[width=\textwidth]{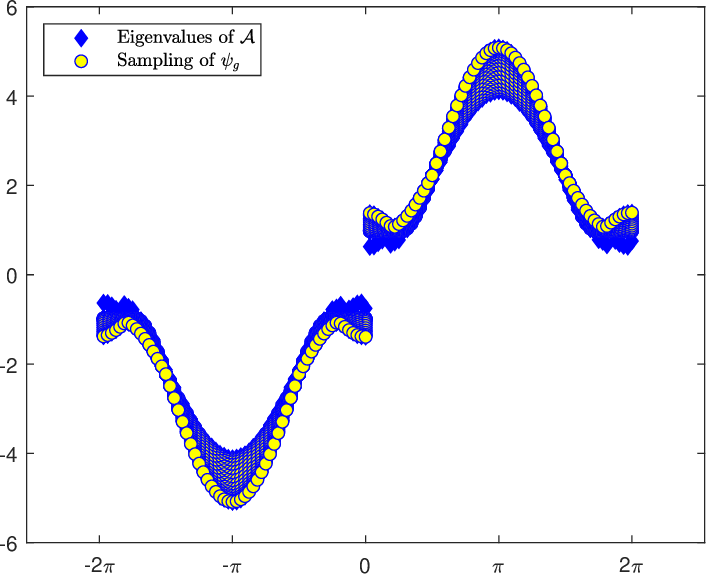}
         \caption{$n=64$}
     \end{subfigure}
     \hfill
     \begin{subfigure}[b]{0.47\textwidth}
         \centering
         \includegraphics[width=\textwidth]{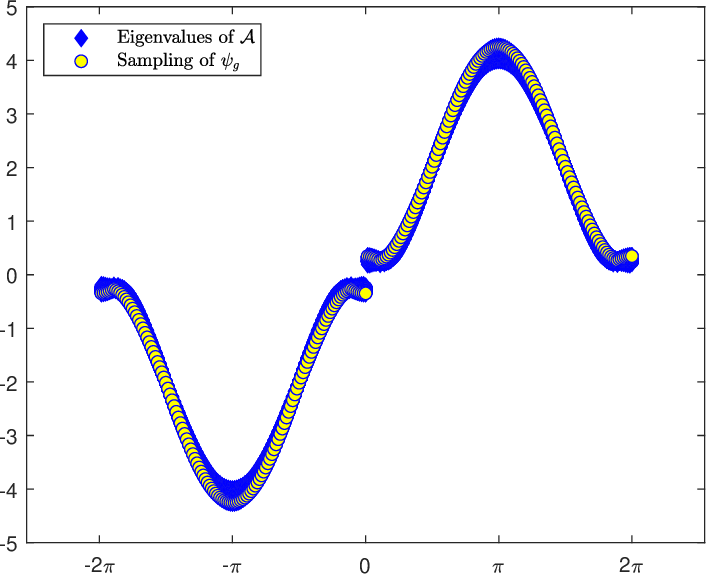}
         \caption{$n=128$}
     \end{subfigure}
        \caption{Comparison for between the eigenvalues of $\mathcal{A}$ and the
sampling of $\psi_{g}$ for Example \ref{example:one_d_eig} with $m=15$ and $\gamma =10^{-6} $.}
        \label{fig:comp_1D_gamma_6}
\end{figure}

\begin{figure}[h!]
     \centering
     \begin{subfigure}[b]{0.47\textwidth}
         \centering
         \includegraphics[width=\textwidth]{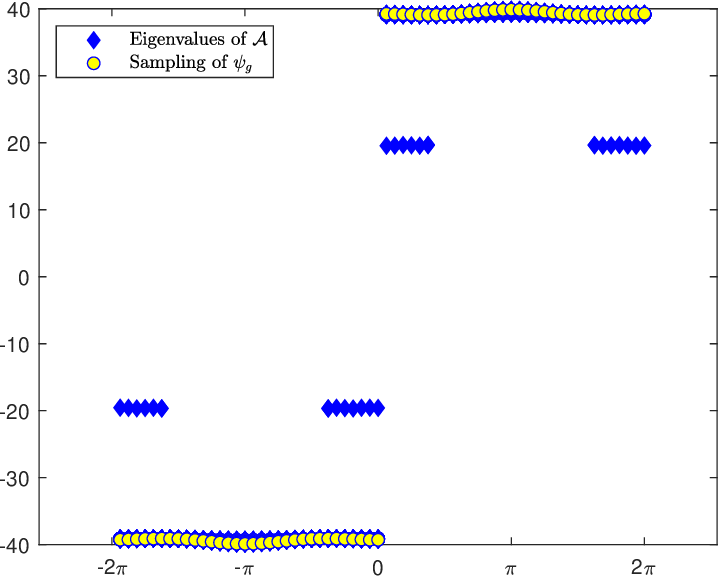}
         \caption{$n=32$}
     \end{subfigure}
     \hfill
     \begin{subfigure}[b]{0.47\textwidth}
         \centering
         \includegraphics[width=\textwidth]{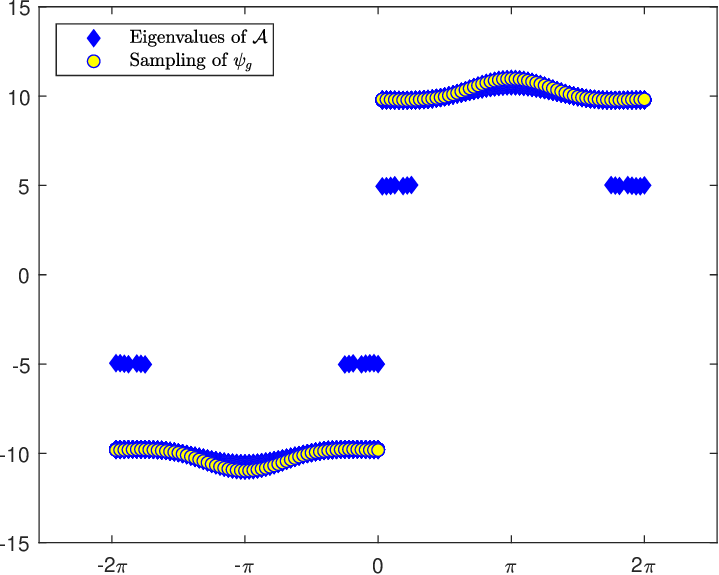}
         \caption{$n=64$}
     \end{subfigure}
     \hfill
     \begin{subfigure}[b]{0.47\textwidth}
         \centering
         \includegraphics[width=\textwidth]{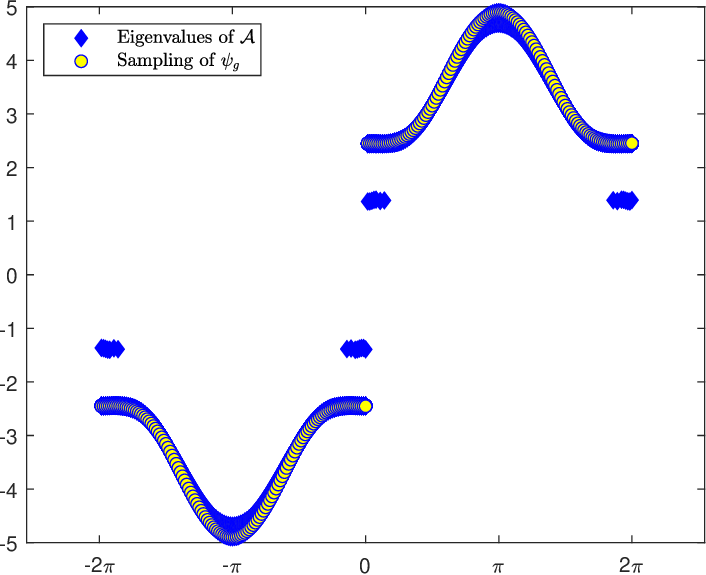}
         \caption{$n=128$}
     \end{subfigure}
        \caption{Comparison between the eigenvalues of $\mathcal{A}$ and the
sampling of $\psi_{g}$ for Example \ref{example:one_d_eig} with $m=15$ and $\gamma =10^{-8} $.}
        \label{fig:comp_1D_gamma_8}
\end{figure}

\end{example}

\begin{example}\label{example:two_d_eig}

In the second example, we consider the two-dimensional case in space. As in Example \ref{example:one_d_eig}, we report excellent matching between the spectrum of $\mathcal{A}$ and the evaluations of the eigenvalues of $\psi_{g}$ computed with the grid points with various $\gamma$ in Figures \ref{fig:comp_2D_gamma_4}-\ref{fig:comp_2D_gamma_8}.

\begin{figure}[h!]
     \centering
     \begin{subfigure}[b]{0.47\textwidth}
         \centering
         \includegraphics[width=\textwidth]{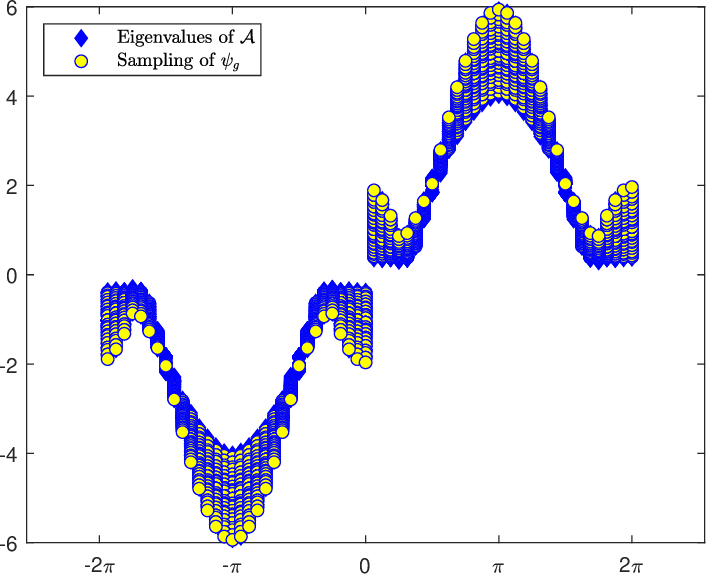}
         \caption{$n=32$}
     \end{subfigure}
     \hfill
     \begin{subfigure}[b]{0.47\textwidth}
         \centering
         \includegraphics[width=\textwidth]{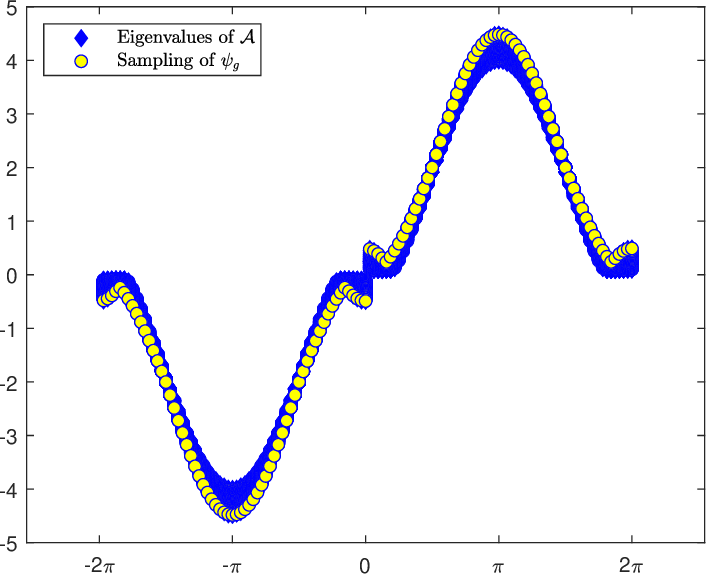}
         \caption{$n=64$}
     \end{subfigure}
     \hfill
     \begin{subfigure}[b]{0.47\textwidth}
         \centering
         \includegraphics[width=\textwidth]{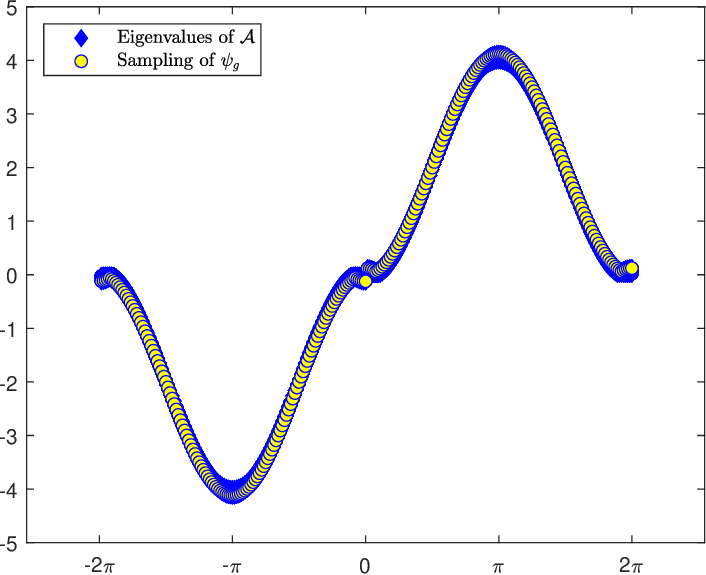}
         \caption{$n=128$}
     \end{subfigure}
        \caption{Comparison between the eigenvalues of $\mathcal{A}$ and the
sampling of $\psi_{g}$ for Example \ref{example:two_d_eig} with $m=7^2$ and $\gamma =10^{-4} $.}
        \label{fig:comp_2D_gamma_4}
\end{figure}

\begin{figure}[h!]
     \centering
     \begin{subfigure}[b]{0.47\textwidth}
         \centering
         \includegraphics[width=\textwidth]{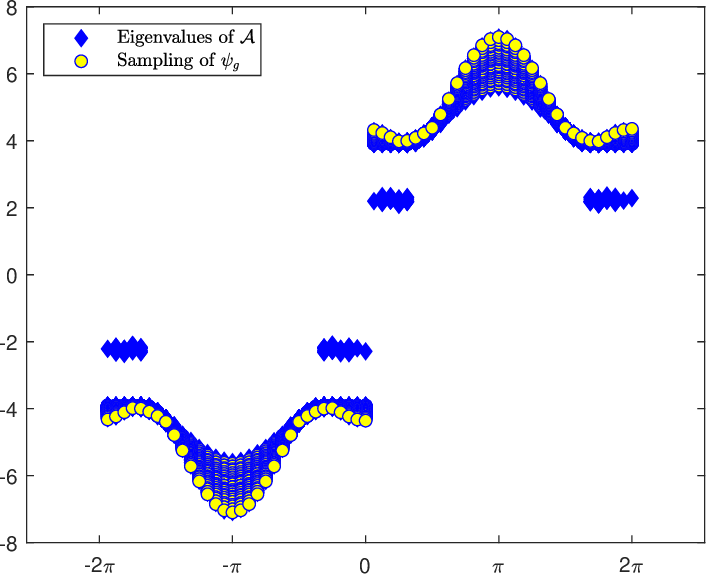}
         \caption{$n=32$}
     \end{subfigure}
     \hfill
     \begin{subfigure}[b]{0.47\textwidth}
         \centering
         \includegraphics[width=\textwidth]{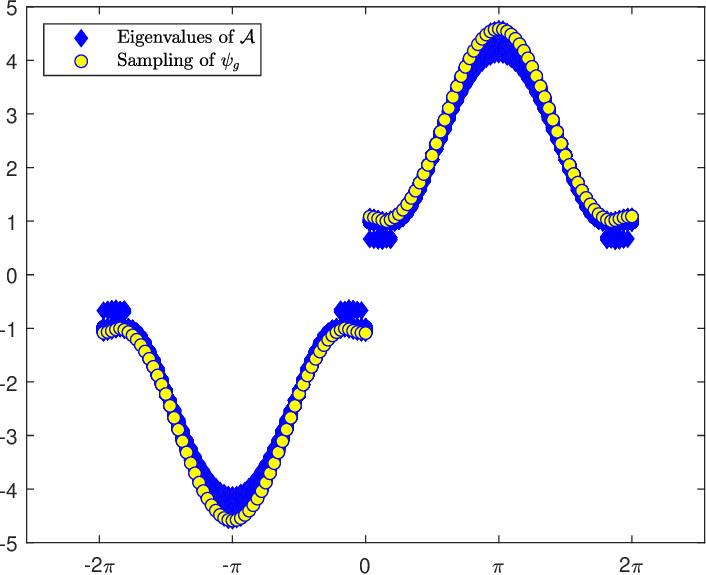}
         \caption{$n=64$}
     \end{subfigure}
     \hfill
     \begin{subfigure}[b]{0.47\textwidth}
         \centering
         \includegraphics[width=\textwidth]{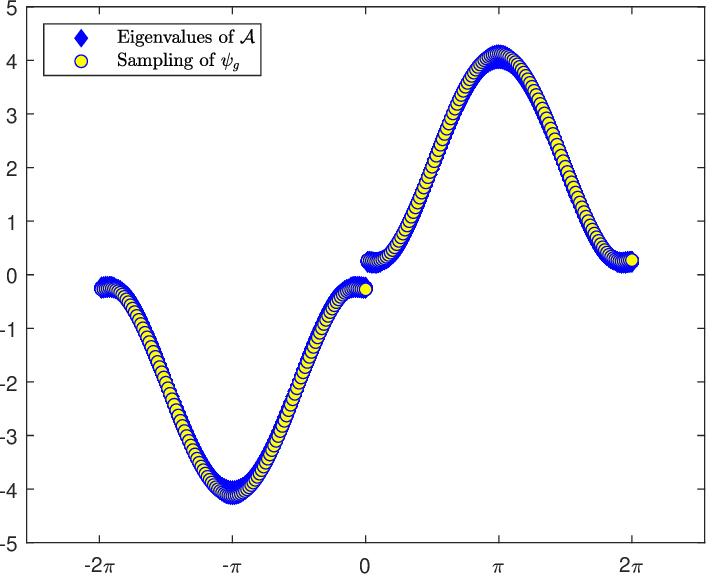}
         \caption{$n=128$}
     \end{subfigure}
        \caption{Comparison between the eigenvalues of $\mathcal{A}$ and the
sampling of $\psi_{g}$ for Example \ref{example:two_d_eig} with $m=7^2$ and $\gamma =10^{-6} $.}
        \label{fig:comp_2D_gamma_6}
\end{figure}

\begin{figure}[h!]
     \centering
     \begin{subfigure}[b]{0.47\textwidth}
         \centering
         \includegraphics[width=\textwidth]{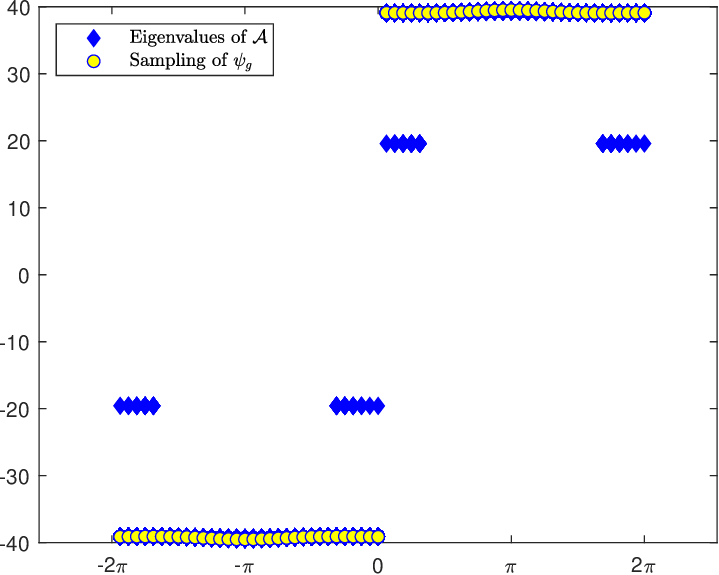}
         \caption{$n=32$}
     \end{subfigure}
     \hfill
     \begin{subfigure}[b]{0.47\textwidth}
         \centering
         \includegraphics[width=\textwidth]{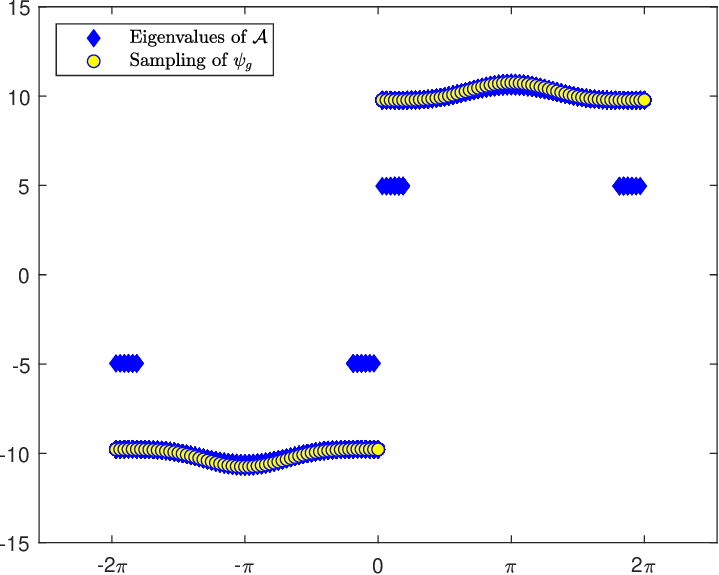}
         \caption{$n=64$}
     \end{subfigure}
     \hfill
     \begin{subfigure}[b]{0.47\textwidth}
         \centering
         \includegraphics[width=\textwidth]{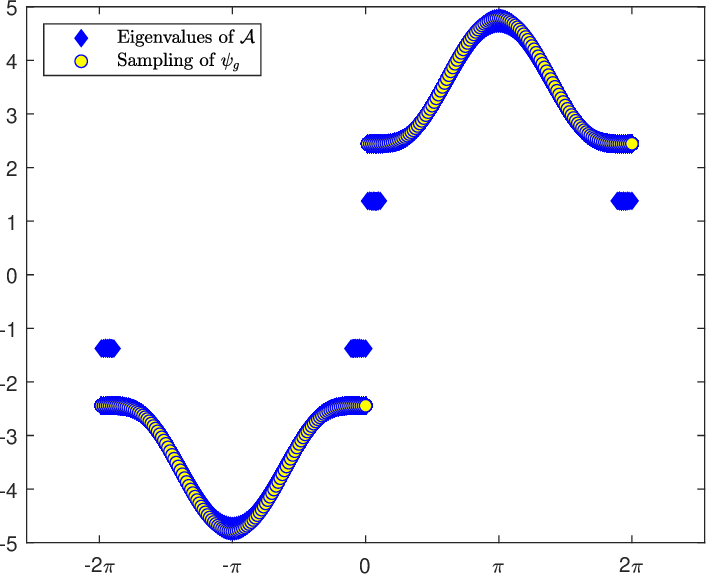}
         \caption{$n=128$}
     \end{subfigure}
        \caption{Comparison between the eigenvalues of $\mathcal{A}$ and the
sampling of $\psi_{g}$ for Example \ref{example:two_d_eig} with $m=7^2$ and $\gamma =10^{-8} $.}
        \label{fig:comp_2D_gamma_8}
\end{figure}

\end{example}

\subsection{Preconditioning for $\mathcal{A}$}
In this subsection, we demonstrate the effectiveness of our proposed preconditioners for solving (\ref{eqn:main_system}), with a uniform mesh in both space and time. The CPU time in seconds is measured using the MATLAB built-in functions \textbf{tic/toc}. Both Steps 1 \& 3 in Section \ref{subsubsec:implementationP} are implemented by the {built-in} functions \textbf{dst} (i.e., discrete sine transform) and \textbf{fft} (i.e., fast Fourier transform). Furthermore, {the MINRES and GMRES (without restarts) solvers are respectively implemented using the built-in functions \textbf{minres} and \textbf{gmres} on MATLAB}. We choose a zero initial guess and a stopping tolerance of $10^{-10}$ based on the reduction in relative residual norms for all Krylov subspace solvers tested unless otherwise indicated. {Additional preconditioning techniques using different preconditioned Krylov methods can be found, for example, in \cite{Ax,Simo}}.

{We adopt the notation MINRES-$\mathcal{P}_{S}$, MINRES-$\mathcal{P}_{G}$, MINRES-$\widetilde{\mathcal{P}}_{S}$, 
MINRES-$\widetilde{\mathcal{P}}_{G}$, GMRES-MSC, GMRES-LW to represent the MINRES solver with the proposed preconditioners $\mathcal{P}_{S}$, $\mathcal{P}_{G}$, $\widetilde{\mathcal{P}}_{S}$, $\widetilde{\mathcal{P}}_{G}$, and the state-of-the-art GMRES solvers proposed in \cite{LiuPearson2019,liuWu_optimal}, respectively. We test our proposed methods MINRES-$\mathcal{P}_{S}$, MINRES-$\mathcal{P}_{G}$, MINRES-$\widetilde{\mathcal{P}}_{S}$, MINRES-$\widetilde{\mathcal{P}}_{G}$ and compare their performance against both MINRES-$I_{2mn}$ (i.e., the non-preconditioned case) and GMRES-MSC, GMRES-LW. Note that we did not compare with the constraint preconditioner proposed in \cite{LiLiuXiao2015}, because its effectiveness deteriorates rapidly as $\gamma$ becomes small. For small $\gamma$, this method cannot outperform GMRES-MSC, as shown in the numerical tests carried out in \cite{LiuPearson2019}.}

In the related tables, we denote by `Iter' the number of iteration for solving a linear system by an iterative solver within the given accuracy. We denote by `CPU' the computational time in unit of second and we denote by `DoF' the amount of degrees of freedom, i.e., the number of unknowns. For ease of statement, we take the same spatial step size $h$ for each spatial direction in all examples appearing in this section. The norm is $L^{\infty}\left((0, T) ; L^{2}(\Omega)\right) $ is used for measuring both state and adjoint state approximation (i.e., $e_y$ and $e_p$).

\begin{example}\cite{liuWu_optimal}\label{example:one_d_test}
In this example, we consider the following one dimensional problem of solving (\ref{eqn:wave_2}), where $\Omega=(0,1)$, the time interval is $[0,2]$, $y_{0}(x)=\sin (\pi x)$, $y_{1}(x)=0$,
		$$\begin{aligned}
			f(x,t)&=-\frac{1}{\gamma} \sin (\pi x)\left(e^{t}-e^{T}\right)^{2},\text{ and } \\ 
			g(x, t)&=\left(4 e^{2 t}-2e^{T+t}\right) \sin (\pi x)+\pi^{2} \sin (\pi x)\left(e^{t}-e^{T}\right)^{2}+\sin (\pi x) \cos (\pi t).
		\end{aligned}$$ The exact solution is $y(x, t)=\sin (\pi x) \cos (\pi t)$ and $p(x, t)=\sin (\pi x)\left(e^{t}-e^{T}\right)^{2}$.
		
		Table \ref{table_1D_it_10_gamma_diff} shows the iteration numbers and CPU time of MINRES when $\mathcal{P}_{S}$ and $\mathcal{P}_{G}$ are applied, {$h = (\frac{2}{\tau} -1)^{-1}$}. For a relatively large $\gamma=10^{-4}$, we observe that the iteration number grows gradually as $m$ and $n$ increase, which can explained by the presence of outlying eigenvalues resulted from the low-rank matrices in our main preconditioning theorems (Theorems \ref{thm:absP_main} \& \ref{thm:absG_main}). Yet, when $\gamma$ gets small, the iteration numbers for convergence needed by both preconditioners are relatively stable, and appear unchanged for various $m$ and $n$. Table \ref{table_1D_error_10_gamma_4} demonstrates the accuracy when {MINRES-$\mathcal{P}_{S}$ and MINRES-$\mathcal{P}_{G}$ are used. We observe that the errors produced by all the tested solvers are almost identical, so their relevant error results will not be reported.}

\cred{In Table \ref{table_1D_CPU_IT_10_gamma_compare}}, we show the comparison between our proposed MINRES solvers and GMRES-LW with {$h = (\frac{2}{\tau} -1)^{-1}$}. We report that for larger value of $\gamma=10^{-4}$, {GMRES-MSC is preferred over MINRES-$\mathcal{P}_{S}/\mathcal{P}_{G}/\widetilde{\mathcal{P}}_{S}$ due to the smaller iteration numbers and shorter CPU times. However, when $\gamma \leq 10^{-6}$ for example, these MINRES solvers can outperform the GMRES methods in terms of total computational time, even though the iteration numbers required for convergence vary and sometimes are roughly two times larger. Particularly, we observe that our proposed MINRES-$\widetilde{\mathcal{P}}_{G}$ appears robust with respect to $\gamma$, namely, its induced iteration numbers are stable for a wide range of $\gamma$, outperforming all tested preconditioners. The numerical results show that GMRES-MSC is actually better than GMRES-LW, especially when $\gamma$ is large, so even GMRES-LW does not work very well. Our new preconditioner MINRES-$\tilde{P}_G$ appears to work well for different $\gamma$.}

{Figures \ref{fig:spectra_example_3_gamma_6} - \ref{fig:spectra_example_3_gamma_8} show the spectrum of the preconditioned matrices with $(m,n)=(32,32)$. Indeed, we observe that from the subplots the presence of the $\mathcal{O}(m)$ eigenvalue outliers predicted by Theorems \ref{thm:absP_main} and \ref{thm:absG_main}, which matches the MINRES iteration number trend accordingly. However, for relatively small $\gamma$, we emphasize that these outliers appear to be tightly clustered and hence the overall spectrum are tightly clustered around $\pm 1$, which accounts for the superb convergence rate of MINRES.} 

\begin{table}[h!]
\caption{Iteration numbers and CPU times of MINRES for Example \ref{example:one_d_test}}
\label{table_1D_it_10_gamma_diff}
\cred{
\begin{center}
\begin{tabular}{ccc|cc|cc|cc}
\hline
\multirow{2}{*}{$\gamma$}   & \multirow{2}{*}{$h$} & \multirow{2}{*}{DoF} & \multicolumn{2}{c|}{MINRES-$\mathcal{P}_{S}$} & \multicolumn{2}{c|}{MINRES-$\mathcal{P}_{G}$} & \multicolumn{2}{c}{MINRES-$I_{2mn}$} \\ \cline{4-9} 
                            &                      &                      & Iter                   & CPU                  & Iter                   & CPU                  & Iter              & CPU              \\ \hline
\multirow{4}{*}{$10^{-2}$}  & $2^{-7}$             & 32766                & 54                     & 0.22                 & 74                     & 0.32                 & $>$200            & -                \\
                            & $2^{-8}$             & 131070               & 106                    & 0.91                 & 147                    & 1.44                 & $>$200            & -                \\
                            & $2^{-9}$             & 524286               & $>$200                 & -                    & $>$200                 & -                    & $>$200            & -                \\
                            & $2^{-10}$            & 2097150              & $>$200                 & -                    & $>$200                 & -                    & $>$200            & -                \\ \hline
\multirow{4}{*}{$10^{-4}$}  & $2^{-7}$             & 32766                & 16                     & 0.075                & 33                     & 0.16                 & $>$200            & -                \\
                            & $2^{-8}$             & 131070               & 28                     & 0.24                 & 55                     & 0.54                 & $>$200            & -                \\
                            & $2^{-9}$             & 524286               & 92                     & 1.78                 & 105                    & 2.39                 & $>$200            & -                \\
                            & $2^{-10}$            & 2097150              & 199                    & 11.53                & $>$200                    & -                & $>$200            & -                \\ \hline
\multirow{4}{*}{$10^{-6}$}  & $2^{-7}$             & 32766                & 10                     & 0.046                & 11                     & 0.065                & $>$200            & -                \\
                            & $2^{-8}$             & 131070               & 10                     & 0.091                & 13                     & 0.13                 & $>$200            & -                \\
                            & $2^{-9}$             & 524286               & 14                     & 0.29                 & 21                     & 0.47                 & $>$200            & -                \\
                            & $2^{-10}$            & 2097150              & 25                     & 1.49                 & 41                     & 2.25                 & $>$200            & -                \\ \hline
\multirow{4}{*}{$10^{-8}$}  & $2^{-7}$             & 32766                & 10                     & 0.047                & 10                     & 0.057                & 107               & 0.19             \\
                            & $2^{-8}$             & 131070               & 10                     & 0.093                & 11                     & 0.11                 & $>$200            & -                \\
                            & $2^{-9}$             & 524286               & 10                     & 0.21                 & 11                     & 0.26                 & $>$200            & -                \\
                            & $2^{-10}$            & 2097150              & 10                     & 0.64                 & 13                     & 0.78                 & $>$200            & -                \\ \hline
\multirow{4}{*}{$10^{-10}$} & $2^{-7}$             & 32766                & 8                      & 0.039                & 7                      & 0.046                & 11                & 0.027            \\
                            & $2^{-8}$             & 131070               & 9                      & 0.098                & 9                      & 0.13                 & 25                & 0.10             \\
                            & $2^{-9}$             & 524286               & 10                     & 0.22                 & 11                     & 0.34                 & $>$200            & -                \\
                            & $2^{-10}$            & 2097150              & 10                     & 0.68                 & 11                     & 0.72                 & $>$200            & -                \\ \hline
\end{tabular}
\end{center}
}
\end{table}

	\begin{table}[h!]
		\caption{Error results of MINRES for Example \ref{example:one_d_test}}
		\label{table_1D_error_10_gamma_4}
		\begin{center}
\cred{
\begin{tabular}{ccc|cc|cc|cc}
\hline
\multirow{2}{*}{$\gamma$}   & \multirow{2}{*}{$h$} & \multirow{2}{*}{DoF} & \multicolumn{2}{c|}{MINRES-$\mathcal{P}_{S}$} & \multicolumn{2}{c|}{MINRES-$\mathcal{P}_{G}$} & \multicolumn{2}{c}{MINRES-$I_{2mn}$} \\ \cline{4-9} 
                            &                      &                      & $e_y$                 & $e_p$                 & $e_y$                 & $e_p$                 & $e_y$             & $e_p$            \\ \hline
\multirow{4}{*}{$10^{-2}$}  & $2^{-5}$             & 32766                & 2.56e-2               & 1.10e-3               & 2.56e-2               & 1.10e-3               & -                 & -                \\
                            & $2^{-6}$             & 131070               & 6.40e-3               & 2.82e-4               & 6.40e-3               & 2.82e-4               & -                 & -                \\
                            & $2^{-7}$             & 524286               & -                     & -                     & -                     & -                     & -                 & -                \\
                            & $2^{-8}$             & 2097150              & -                     & -                     & -                     & -                     & -                 & -                \\ \hline
\multirow{4}{*}{$10^{-4}$}  & $2^{-5}$             & 32766                & 3.44e-2               & 1.64e-4               & 3.44e-2               & 1.64e-4               & -                 & -                \\
                            & $2^{-6}$             & 131070               & 8.60e-3               & 4.12e-5               & 8.60e-3               & 4.12e-5               & -                 & -                \\
                            & $2^{-7}$             & 524286               & 2.20e-3               & 1.03e-5               & 2.20e-3               & 1.03e-5               & -                 & -                \\
                            & $2^{-8}$             & 2097150              & 5.39e-4               & 2.58e-6               & -               & -               & -                 & -                \\ \hline
\multirow{4}{*}{$10^{-6}$}  & $2^{-5}$             & 32766                & 6.11e-2               & 2.98e-5               & 6.11e-2               & 2.98e-5               & 6.11e-2           & 2.98e-5          \\
                            & $2^{-6}$             & 131070               & 1.47e-2               & 7.75e-6               & 1.47e-2               & 7.75e-6               & -                 & -                \\
                            & $2^{-7}$             & 524286               & 3.60e-3               & 1.96e-6               & 3.60e-3               & 1.96e-6               & -                 & -                \\
                            & $2^{-8}$             & 2097150              & 8.84e-4               & 4.91e-7               & 8.84e-4               & 4.91e-7               & -                 & -                \\ \hline
\multirow{4}{*}{$10^{-8}$}  & $2^{-5}$             & 32766                & 0.12                  & 3.36e-6               & 0.12                  & 3.36e-6               & 0.12              & 3.36e-6          \\
                            & $2^{-6}$             & 131070               & 2.97e-2               & 1.24e-6               & 2.97e-2               & 1.24e-6               & 3.00e-2           & 1.24e-6          \\
                            & $2^{-7}$             & 524286               & 6.90e-3               & 3.40e-7               & 6.90e-3               & 3.40e-7               & -                 & -                \\
                            & $2^{-8}$             & 2097150              & 1.60e-3               & 8.69e-8               & 1.60e-3               & 8.69e-8               & -                 & -                \\ \hline
\multirow{4}{*}{$10^{-10}$} & $2^{-5}$             & 32766                & 0.14                  & 6.17e-8               & 0.14                  & 6.17e-8               & 0.14              & 6.17e-8          \\
                            & $2^{-6}$             & 131070               & 4.97e-2               & 7.47e-8               & 4.97e-2               & 7.47e-8               & 4.97e-2           & 7.47e-8          \\
                            & $2^{-7}$             & 524286               & 1.40e-2               & 4.55e-8               & 1.40e-2               & 4.55e-8               & 1.40e-2           & 4.55e-8          \\
                            & $2^{-8}$             & 2097150              & 3.30e-3               & 1.44e-8               & 3.30e-3               & 1.44e-8               & 3.30e-3           & 1.45e-8          \\ \hline
\end{tabular}
}
\end{center}
\end{table}

	\begin{table}[h!]
		\caption{\cred{Comparison on iteration numbers and CPU times between various solvers for Example \ref{example:one_d_test}}}
		\label{table_1D_CPU_IT_10_gamma_compare}
		\begin{center}
\cred{
\begin{tabular}{ccc|cc|cc|cc|cc}
\hline
\multirow{2}{*}{$\gamma$}   & \multirow{2}{*}{$h$} & \multirow{2}{*}{DoF} & \multicolumn{2}{c|}{MINRES-$\widetilde{\mathcal{P}}_{S}$} & \multicolumn{2}{c|}{MINRES-$\widetilde{\mathcal{P}}_{G}$}  & \multicolumn{2}{c|}{GMRES-{MSC}} & \multicolumn{2}{c}{GMRES-{LW}} \\ \cline{4-11} 
                            &                      &                      & Iter                        & CPU                         & Iter                        & CPU     & Iter                        & CPU       & Iter                        & CPU                              \\ \hline
\multirow{4}{*}{$10^{-2}$}  & $2^{-7}$             & 32766                & 60                          & 0.26                        & 38                          & 0.16          &                13&	 0.12                         &                      11&	 0.73                      \\
                            & $2^{-8}$             & 131070               & 79                          & 0.69                        & 39                          & 0.38           &                  13&	 0.29                       &                          37&	 9.43                         \\
                            & $2^{-9}$             & 524286               & 108                         & 2.07                        & 40                          & 0.91          &                  13&	 0.87                       &                         98&	 91.61                          \\
                            & $2^{-10}$            & 2097150              & 154                         & 8.90                        & 44                          & 2.49            &                12&	 2.72                      &                         158&	 574.13                             \\ \hline
\multirow{4}{*}{$10^{-4}$}  & $2^{-7}$             & 32766                & 32                          & 0.13                        & 22                          & 0.11     &                   22&	 0.22                        &                          5&	 0.38                            \\
                            & $2^{-8}$             & 131070               & 43                          & 0.37                        & 24                          & 0.25           &                  27&	 0.66                        &                     7&	 1.91                                \\
                            & $2^{-9}$             & 524286               & 59                          & 1.15                        & 25                          & 0.58               &                   29&	 2.05                        &                   17&	 16.57                         \\
                            & $2^{-10}$            & 2097150              & 87                          & 5.04                        & 26                          & 1.49           &                 31&	 7.16                &                        51&	 190.34                                            \\ \hline
\multirow{4}{*}{$10^{-6}$}  & $2^{-7}$             & 32766                & 17                          & 0.078                       & 13                          & 0.065         &                  20&	 0.20                       &                      5&	 0.38                   \\
                            & $2^{-8}$             & 131070               & 18                          & 0.16                        & 14                          & 0.15                 &                  31&	 0.80                      &                    5&	 1.47                     \\
                            & $2^{-9}$             & 524286               & 23                          & 0.46                        & 14                          & 0.33               &             49&	 3.83                        &                        5&	 5.75                               \\
                            & $2^{-10}$            & 2097150              & 34                          & 2.02                       & 14                          & 0.83                 &       65&	 16.37                         &                         5&	 22.61                       \\ \hline
\multirow{4}{*}{$10^{-8}$}  & $2^{-7}$             & 32766                & 12                          & 0.056                       & 11                          & 0.053           &            8&	 0.07                       &                          5&	 0.38                          \\
                            & $2^{-8}$             & 131070               & 14                          & 0.13                        & 12                          & 0.12              &              12&	 0.26                       &                           5&	 1.46                          \\
                            & $2^{-9}$             & 524286               & 14                          & 0.29                        & 11                          & 0.26             &                 23&	 1.59                     &                          5&	 5.72                        \\
                            & $2^{-10}$            & 2097150              & 16                          & 0.97                        & 11                          & 0.67             &             39&	 9.21                 &                       5&	 22.70                              \\ \hline
\multirow{4}{*}{$10^{-10}$} & $2^{-7}$             & 32766                & 8                           & 0.036                       & 7                           & 0.035                &              3&	 0.03                       &                          5&	 0.37                               \\
                            & $2^{-8}$             & 131070               & 9                           & 0.086                       & 9                           & 0.098          &                  4&	 0.09                      &                      5&	 1.44                                  \\
                            & $2^{-9}$             & 524286               & 12                          & 0.26                        & 10                          & 0.24                &             6&	 0.41              &                        5&	 5.64                                        \\
                            & $2^{-10}$            & 2097150              & 12                          & 0.76                        & 11                          & 0.69            &               11&	 2.49                    &                       5&	 22.46                  \\ \hline
\end{tabular}
}
\end{center}
	\end{table}

\begin{figure}[h!]
     \centering
     \begin{subfigure}[b]{0.48\textwidth}
         \centering
         \includegraphics[width=\textwidth]{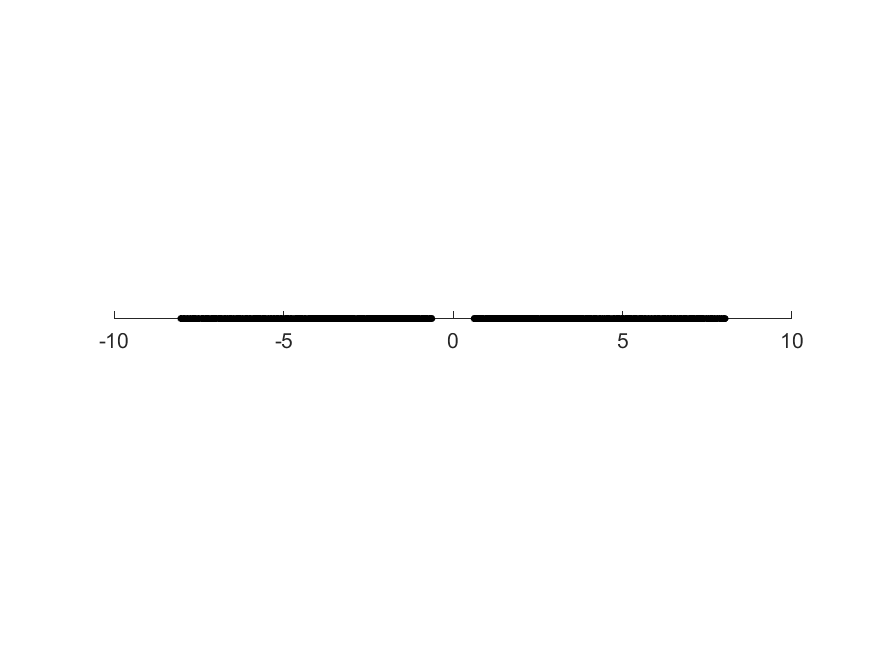}
         \caption{\footnotesize Unpreconditioned}
     \end{subfigure}
     \hfill
     \begin{subfigure}[b]{0.48\textwidth}
         \centering
         \includegraphics[width=\textwidth]{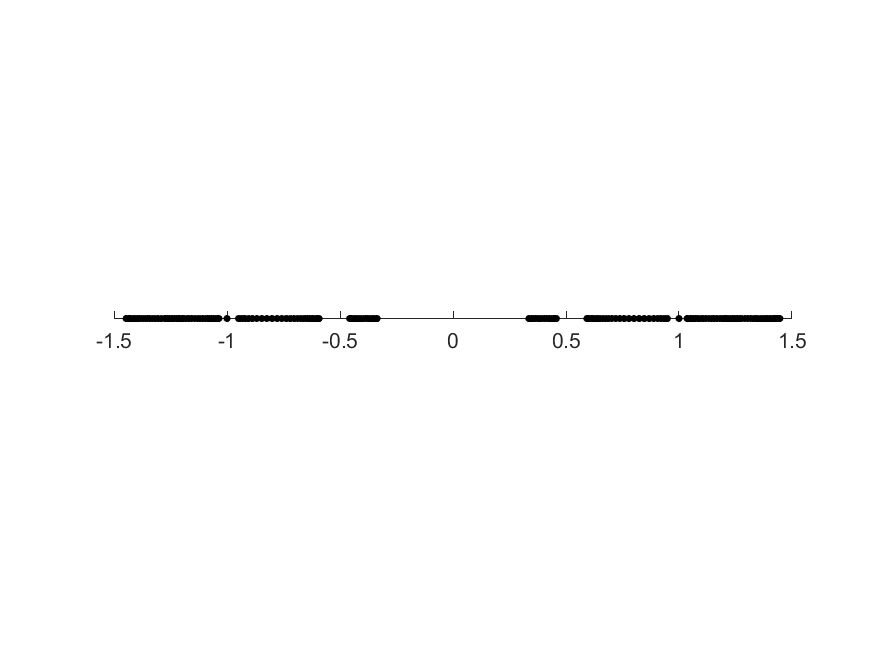}
         \caption{\footnotesize with $\mathcal{P}_{S}$}
     \end{subfigure}
     \hfill
     \begin{subfigure}[b]{0.48\textwidth}
         \centering
         \includegraphics[width=\textwidth]{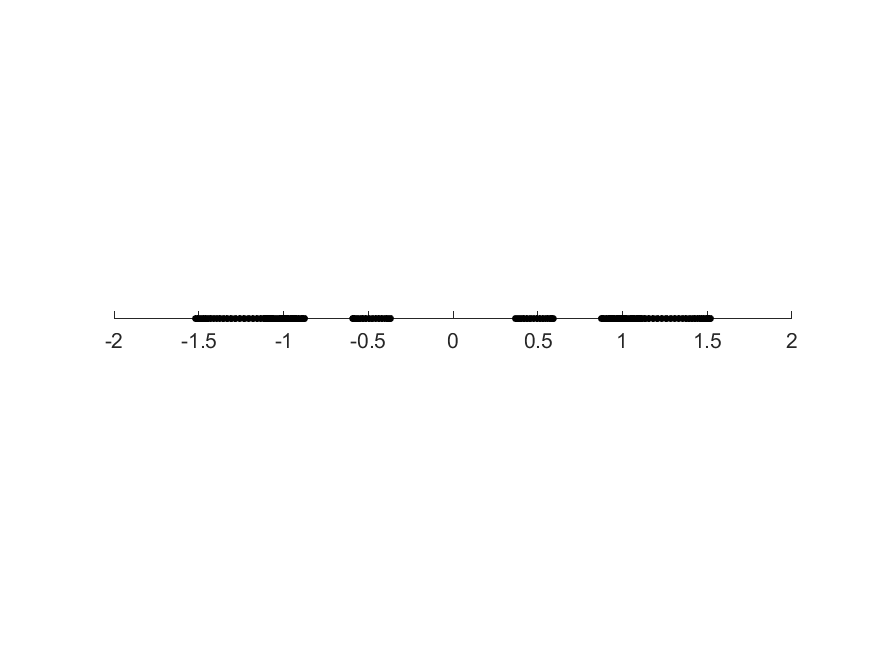}
         \caption{\footnotesize with $\mathcal{P}_{G}$}
     \end{subfigure}
      \hfill
     \begin{subfigure}[b]{0.48\textwidth}
         \centering
         \includegraphics[width=\textwidth]{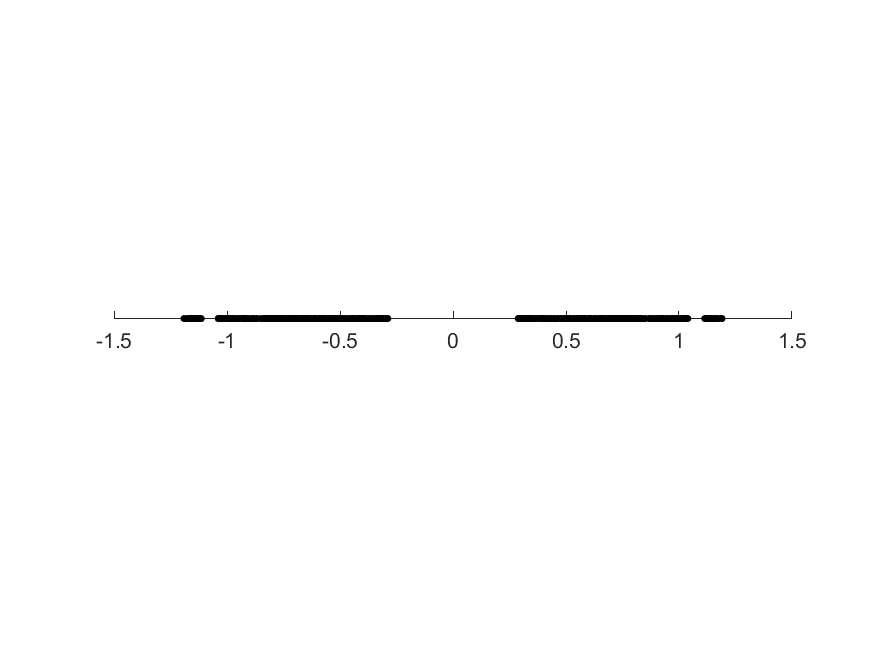}
         \caption{\footnotesize with $\widetilde{\mathcal{P}}_{S}$}
     \end{subfigure}
     \hfill
     \begin{subfigure}[b]{0.48\textwidth}
         \centering
         \includegraphics[width=\textwidth]{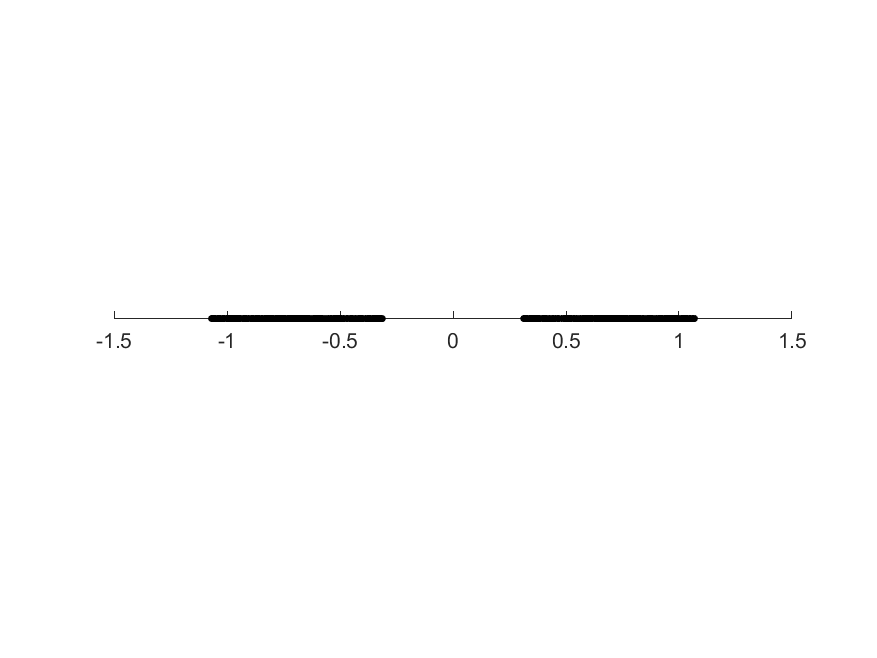}
         \caption{\footnotesize with $\widetilde{\mathcal{P}}_{G}$}
     \end{subfigure}
        \caption{{Spectrum of the preconditioned matrices for Example \ref{example:one_d_test} with $(m,n)=(32,32)$ and $\gamma =10^{-6}$.}}
        \label{fig:spectra_example_3_gamma_6}
\end{figure}

\begin{figure}[h!]
     \centering
     \begin{subfigure}[b]{0.48\textwidth}
         \centering
         \includegraphics[width=\textwidth]{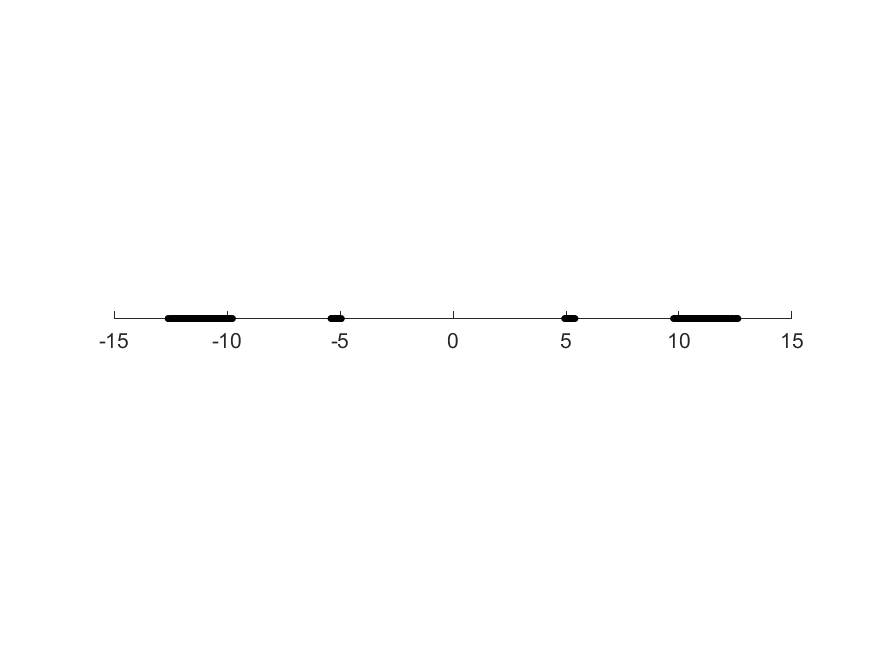}
         \caption{\footnotesize Unpreconditioned}
     \end{subfigure}
     \hfill
     \begin{subfigure}[b]{0.48\textwidth}
         \centering
         \includegraphics[width=\textwidth]{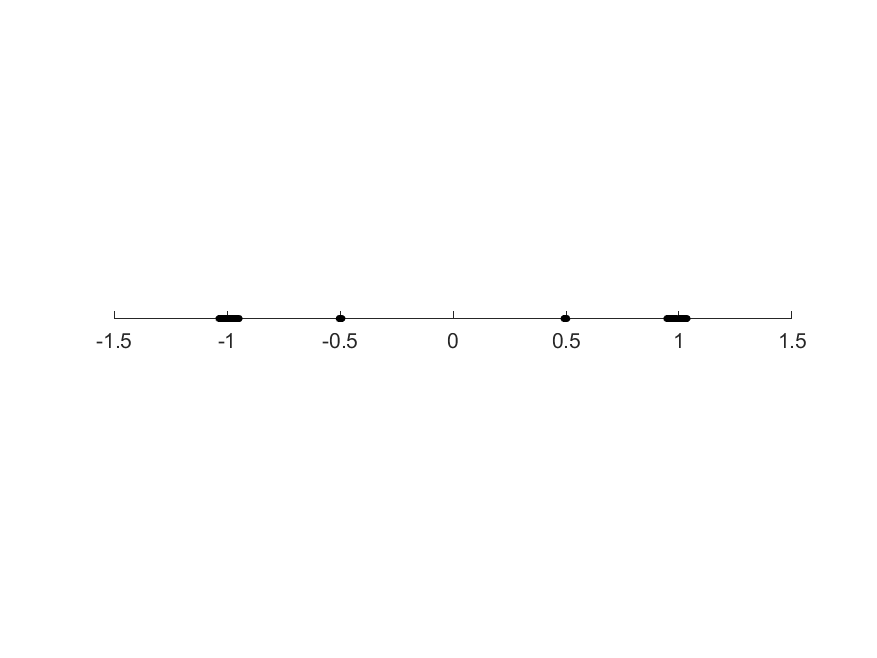}
         \caption{\footnotesize with $\mathcal{P}_{S}$}
     \end{subfigure}
     \hfill
     \begin{subfigure}[b]{0.48\textwidth}
         \centering
         \includegraphics[width=\textwidth]{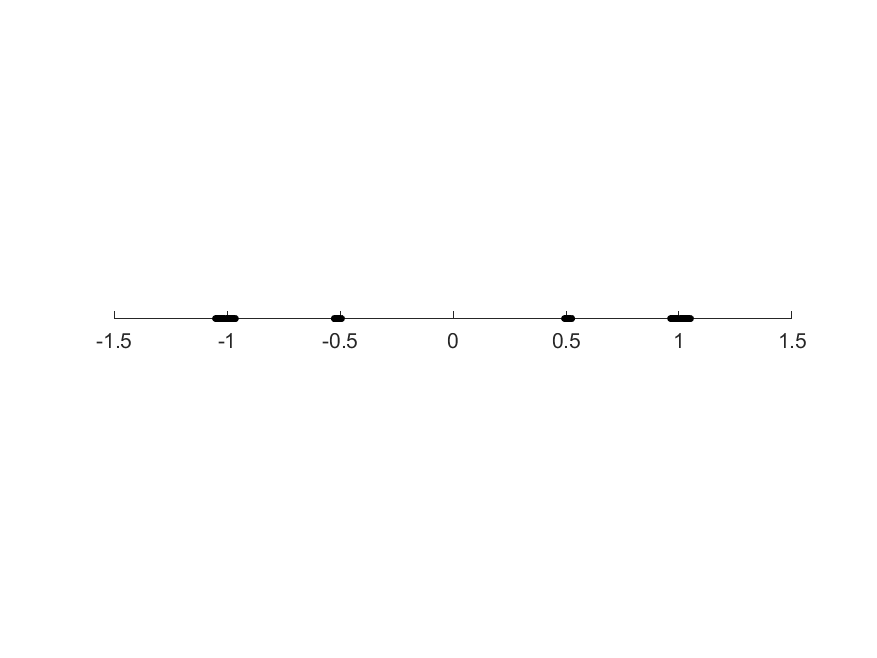}
         \caption{\footnotesize with $\mathcal{P}_{G}$}
     \end{subfigure}
     \hfill
     \begin{subfigure}[b]{0.48\textwidth}
         \centering
         \includegraphics[width=\textwidth]{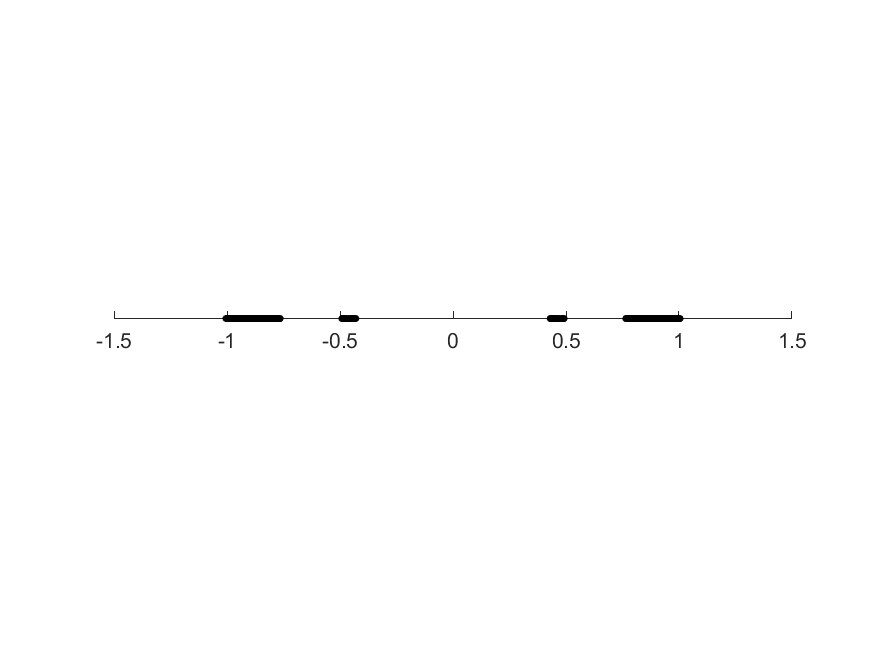}
         \caption{\footnotesize with $\widetilde{\mathcal{P}}_{S}$}
     \end{subfigure}
     \hfill
     \begin{subfigure}[b]{0.48\textwidth}
         \centering
         \includegraphics[width=\textwidth]{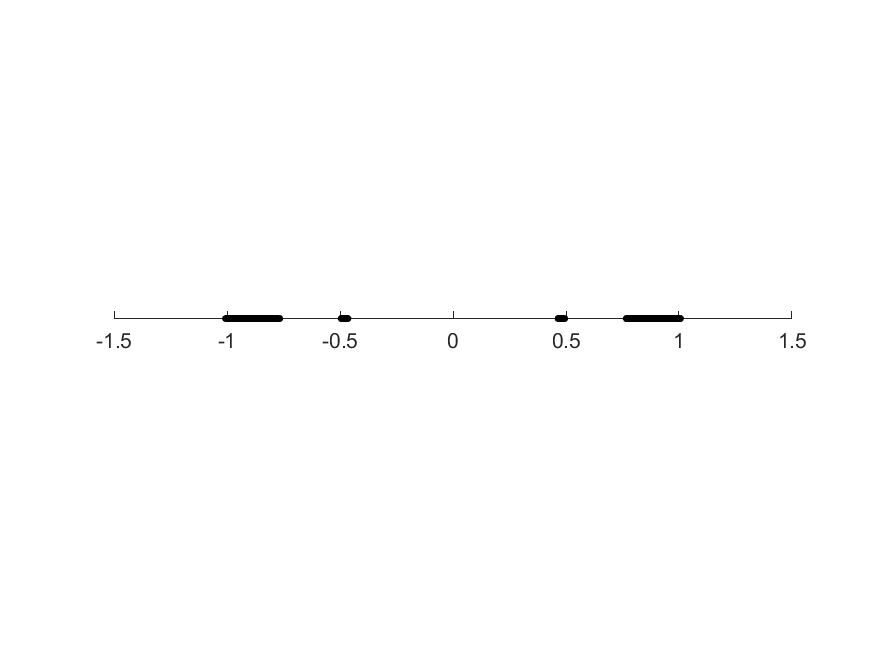}
         \caption{\footnotesize with $\widetilde{\mathcal{P}}_{G}$}
     \end{subfigure}
        \caption{{Spectrum of the preconditioned matrices for Example \ref{example:one_d_test} with $(m,n)=(32,32)$ and $\gamma =10^{-8}$.}}
        \label{fig:spectra_example_3_gamma_8}
\end{figure}

		\end{example}

	\begin{example}\cite{liuWu_optimal}\label{example:two_d_test} 
	We consider the following two dimensional problem of solving (\ref{eqn:wave_2}) with $\Omega=(0,1) \times (0,1),$ the time interval $[0,T]=[0,2]$,
		$$
		\begin{aligned}
			y_{0}\left(x_{1}, x_{2}\right)&=\sin \left(\pi x_{1}\right) \sin \left(\pi x_{2}\right), \quad y_{1}\left(x_{1}, x_{2}\right)=\sin \left(\pi x_{1}\right) \sin \left(\pi x_{2}\right), \\
			f\left(x_{1}, x_{2}, t\right)&=\left(1+2 \pi^{2}\right) e^{t} \sin \left(\pi x_{1}\right) \sin \left(\pi x_{2}\right)-\frac{1}{\gamma}(t-T)^{2} \sin \left(\pi x_{1}\right) \sin \left(\pi x_{2}\right), \\
			\text{ and }g\left(x_{1}, x_{2}, t\right)&=\left(e^{t}+2+2 \pi^{2}(t-T)^{2}\right) \sin \left(\pi x_{1}\right) \sin \left(\pi x_{2}\right) .
		\end{aligned}
		$$ The exact solution is
		$y\left(x_{1}, x_{2}, t\right)=e^{t} \sin \left(\pi x_{1}\right) \sin \left(\pi x_{2}\right) $ and $ p\left(x_{1}, x_{2}, t\right)=(t-T)^{2} \sin \left(\pi x_{1}\right) \sin \left(\pi x_{2}\right).$
	
	Tables \ref{table_2D_it_10_gamma_diff} \& \ref{table_2D_error_gamma_diff} show the results of our proposed MINRES solvers, with {$h=(\frac{2}{\tau} -1)^{-1}$}.  Furthermore, \cred{in Table \ref{table_2D_it_10_gamma_compare}}, we show the comparison between our proposed MINRES solvers and the GMRES ones, with $h=(\frac{2}{\tau} -1)^{-1}$. {As in Example \ref{example:one_d_test}, both the overall computational cost and CPU times are significantly improved when using our proposed methods MINRES-$\mathcal{P}_{S}/\mathcal{P}_{G}/\widetilde{\mathcal{P}}_{S}$ for relatively small $\gamma$, especially for larger matrix sizes. On the contrary, for larger values of $\gamma$ GMRES-{MSC} has to be preferred over these MINRES methods. However, similar to the last example, MINRES-$\widetilde{\mathcal{P}}_{G}$ outperforms all preconditioned methods implemented in terms of CPU times in all cases, which displays robustness for a wide range of $\gamma$.}

{Figures \ref{fig:spectra_example_4_gamma_6} - \ref{fig:spectra_example_4_gamma_8} show the spectrum of the preconditioned matrices with $(m,n)=(16,16)$. Similar to the last example, we observe the predicted eigenvalue distribution according to our theorems, as well as the tight clusters around $\pm 1$ for small $\gamma$ which explains the excellent performance of the MINRES method used.}


\begin{table}[h!]
\caption{Iteration numbers and CPU times of MINRES for Example \ref{example:two_d_test} }
\label{table_2D_it_10_gamma_diff}
\begin{center}
\cred{
\begin{tabular}{ccc|cc|cc|cc}
\hline
\multirow{2}{*}{$\gamma$}   & \multirow{2}{*}{$h$} & \multirow{2}{*}{DoF} & \multicolumn{2}{c|}{MINRES-$\mathcal{P}_{S}$} & \multicolumn{2}{c|}{MINRES-$\mathcal{P}_{G}$} & \multicolumn{2}{c}{MINRES-$I_{2mn}$} \\ \cline{4-9} 
                            &                      &                      & Iter                  & CPU                & Iter                  & CPU                & Iter              & CPU           \\ \hline
\multirow{4}{*}{$10^{-2}$}  & $2^{-5}$             & 63426                & 18                    & 0.12                  & 86                    & 0.66                  & $>$200            & -                \\
                            & $2^{-6}$             & 515970               & 78                    & 1.79                  & 186                   & 4.77                  & $>$200            & -                \\
                            & $2^{-7}$             & 4161282              & $>$200                & -                     & $>$200                & -                     & $>$200            & -                \\
                            & $2^{-8}$             & 33422850             & $>$200                & -                     & $>$200                & -                     & $>$200            & -                \\ \hline
\multirow{4}{*}{$10^{-4}$}  & $2^{-5}$             & 63426                & 10                    & 0.071                 & 15                    & 0.12                  & $>$200            & -                \\
                            & $2^{-6}$             & 515970               & 12                    & 0.29                  & 22                    & 0.59                  & $>$200           & -                \\
                            & $2^{-7}$             & 4161282              & 34                    & 4.09                  & 57                    & 7.56                  & $>$200           & -                \\
                            & $2^{-8}$             & 33422850             & 60                    & 58.63                 & 146                   & 150.80                & $>$200           & -                \\ \hline
\multirow{4}{*}{$10^{-6}$}  & $2^{-5}$             & 63426                & 10                    & 0.072                 & 10                    & 0.082                 & 101               & 0.19             \\
                            & $2^{-6}$             & 515970               & 10                    & 0.24                  & 12                    & 0.33                  & $>$200            & -                \\
                            & $2^{-7}$             & 4161282              & 10                    & 1.28                  & 12                    & 1.88                  & $>$200            & -                \\
                            & $2^{-8}$             & 33422850             & 11                    & 11.95                 & 20                    & 21.10                 & $>$200            & -                \\ \hline
\multirow{4}{*}{$10^{-8}$}  & $2^{-5}$             & 63426                & 8                     & 0.054                 & 8                     & 0.066                 & 10                & 0.021            \\
                            & $2^{-6}$             & 515970               & 9                     & 0.22                  & 9                     & 0.25                  & 22                & 0.14             \\
                            & $2^{-7}$             & 4161282              & 10                    & 1.28                  & 10                    & 1.63                  & 193               & 6.72             \\
                            & $2^{-8}$             & 33422850             & 10                    & 10.95                 & 11                    & 13.87                 & $>$200            & -                \\ \hline
\multirow{4}{*}{$10^{-10}$} & $2^{-5}$             & 63426                & 6                     & 0.045                 & 6                     & 0.053                 & 7                 & 0.020            \\
                            & $2^{-6}$             & 515970               & 7                     & 0.17                  & 6                     & 0.17                  & 8                 & 0.058            \\
                            & $2^{-7}$             & 4161282              & 8                     & 1.07                  & 8                     & 1.36                  & 11                & 0.61             \\
                            & $2^{-8}$             & 33422850             & 10                    & 11.14                 & 9                     & 11.77                 & 41                & 13.52            \\ \hline
\end{tabular}
}
\end{center}
\end{table}

\begin{table}[h!]
\caption{Error results of MINRES for Example \ref{example:two_d_test} }
\label{table_2D_error_gamma_diff}
\begin{center}
\cred{
\begin{tabular}{ccc|cc|cc|cc}
\hline
\multirow{2}{*}{$\gamma$}   & \multirow{2}{*}{$h$} & \multirow{2}{*}{DoF} & \multicolumn{2}{c|}{MINRES-$\mathcal{P}_{S}$} & \multicolumn{2}{c|}{MINRES-$\mathcal{P}_{G}$} & \multicolumn{2}{c}{MINRES-$I_{2mn}$} \\ \cline{4-9} 
                            &                      &                      & $e_y$                 & $e_p$                 & $e_y$                 & $e_p$                 & $e_y$            & $e_p$             \\ \hline
\multirow{4}{*}{$10^{-2}$}  & $2^{-5}$             & 63426                & 1.86e-2               & 1.90e-3               & 1.86e-2               & 1.90e-3               & -                & -                 \\
                            & $2^{-6}$             & 515970               & 4.80e-3               & 4.88e-4               & 4.80e-3               & 4.88e-4               & -                & -                 \\
                            & $2^{-7}$             & 4161282              & -                     & -                     & -                     & -                     & -                & -                 \\
                            & $2^{-8}$             & 33422850             & -                     & -                     & -                     & -                     & -                & -                 \\ \hline
\multirow{4}{*}{$10^{-4}$}  & $2^{-5}$             & 63426                & 3.64e-2               & 6.39e-5               & 3.64e-2               & 6.39e-5               & 3.64e-2          & 6.39e-5           \\
                            & $2^{-6}$             & 515970               & 9.40e-3               & 1.69e-5               & 9.40e-3               & 1.69e-5               & -                & -                 \\
                            & $2^{-7}$             & 4161282              & 2.40e-3               & 4.60e-6               & 2.40e-3               & 4.60e-6               & -                & -                 \\
                            & $2^{-8}$             & 33422850             & 6.01e-4               & 1.20e-6               & 6.01e-4               & 1.20e-6               & -                & -                 \\ \hline
\multirow{4}{*}{$10^{-6}$}  & $2^{-5}$             & 63426                & 3.63e-2               & 2.44e-6               & 3.63e-2               & 2.44e-6               & 3.63e-2          & 2.44e-6           \\
                            & $2^{-6}$             & 515970               & 9.40e-3               & 1.25e-6               & 9.40e-3               & 1.25e-6               & 9.40e-3          & 1.25e-6           \\
                            & $2^{-7}$             & 4161282              & 2.40e-3               & 4.00e-7               & 2.40e-3               & 4.00e-7               & -                & -                 \\
                            & $2^{-8}$             & 33422850             & 5.99e-4               & 1.01e-7               & 5.99e-4               & 1.01e-7               & -                & -                 \\ \hline
\multirow{4}{*}{$10^{-8}$}  & $2^{-5}$             & 63426                & 3.62e-2               & 2.99e-8               & 3.62e-2               & 2.99e-8               & 3.62e-2          & 2.99e-8           \\
                            & $2^{-6}$             & 515970               & 9.30e-3               & 3.09e-8               & 9.30e-3               & 3.09e-8               & 9.30e-3          & 3.09e-8           \\
                            & $2^{-7}$             & 4161282              & 2.40e-3               & 2.29e-8               & 2.40e-3               & 2.29e-8               & 2.40e-3          & 2.29e-8           \\
                            & $2^{-8}$             & 33422850             & 5.98e-4               & 9.04e-9               & 6.00e-4               & 9.04e-9               & 5.98e-4          & 9.05e-9           \\ \hline
\multirow{4}{*}{$10^{-10}$} & $2^{-5}$             & 63426                & 3.62e-2               & 3.00e-10              & 3.62e-2               & 3.00e-10              & 3.62e-2          & 3.00e-10          \\
                            & $2^{-6}$             & 515970               & 9.30e-3               & 3.22e-10              & 9.30e-3               & 3.22e-10              & 9.30e-3          & 3.22e-10          \\
                            & $2^{-7}$             & 4161282              & 2.40e-3               & 3.33e-10              & 2.40e-3               & 3.33e-10              & 2.40e-3          & 3.33e-10          \\
                            & $2^{-8}$             & 33422850             & 5.98e-4               & 3.12e-10              & 5.98e-4               & 3.12e-10              & 5.98e-4          & 3.26e-10         \\\hline
\end{tabular}
}
\end{center}
\end{table}

\begin{table}[h!]
\caption{\cred{Comparison on iteration numbers and CPU times between various solvers for Example \ref{example:two_d_test}}}
\label{table_2D_it_10_gamma_compare}
\begin{center}
\cred{
\begin{tabular}{ccc|cc|cc|cc|cc}
\hline
\multirow{2}{*}{$\gamma$}   & \multirow{2}{*}{$h$} & \multirow{2}{*}{DoF} & \multicolumn{2}{c|}{MINRES-$\widetilde{\mathcal{P}}_{S}$} & \multicolumn{2}{c|}{MINRES-$\widetilde{\mathcal{P}}_{G}$} & \multicolumn{2}{c|}{GMRES-{MSC}} & \multicolumn{2}{c}{GMRES-{LW}} \\ \cline{4-11} 
                            &                      &                      & Iter                       & CPU                          & Iter                        & CPU           & Iter                        & CPU   & Iter                        & CPU                             \\ \hline
\multirow{4}{*}{$10^{-2}$}  & $2^{-5}$             & 63426                & 45                         & 0.28                         & 42                          & 0.31         &                        13&	 0.26                       &                            5&	 0.12                              \\
                            & $2^{-6}$             & 515970               & 58                         & 1.29                         & 49                          & 1.24        &                          14&	 1.46                         &                           13&	 1.34                                  \\
                            & $2^{-7}$             & 4161282              & 78                         & 9.25                         & 51                          & 6.51             &                           14&	 10.56                         &                            63&	 42.39                                \\
                            & $2^{-8}$             & 33422850             & 106                        & 104.42                       & 54                          & 55.88       &                           14&	 96.13                        &                            109&	 906.09                                     \\ \hline
\multirow{4}{*}{$10^{-4}$}  & $2^{-5}$             & 63426                & 28                         & 0.17                         & 24                          & 0.18        &                         16&	 0.32                        &                           5&	 0.13                                       \\
                            & $2^{-6}$             & 515970               & 31                         & 0.71                         & 26                          & 0.71       &                         21&	 2.17                          &                            5&	 0.56                                    \\
                            & $2^{-7}$             & 4161282              & 41                         & 4.95                         & 28                          & 3.76      &                          28&	 20.57                         &                            7&	 4.84                      \\
                            & $2^{-8}$             & 33422850             & 56                         & 55.26                        & 32                          & 33.57              &                         35&	 233.78                   &                            11&	 42.98                   \\ \hline
\multirow{4}{*}{$10^{-6}$}  & $2^{-5}$             & 63426                & 14                         & 0.096                        & 14                          & 0.12       &                      12&	 0.24                        &                            5&	 0.11                                  \\
                            & $2^{-6}$             & 515970               & 16                         & 0.38                         & 15                          & 0.40         &                         16&	 1.63                         &                            5&	 0.56                                   \\
                            & $2^{-7}$             & 4161282              & 18                         & 2.25                         & 15                          & 2.03          &                          24&	 17.85                         &                            5&	 3.58                          \\
                            & $2^{-8}$             & 33422850             & 22                         & 22.74                        & 15                          & 16.61           &                         36&	 242.29                         &                            5&	 22.21                            \\ \hline
\multirow{4}{*}{$10^{-8}$}  & $2^{-5}$             & 63426                & 9                          & 0.062                        & 9                           & 0.054              &                        5&	 0.10                       &                             5&	 0.11                                         \\
                            & $2^{-6}$             & 515970               & 11                         & 0.27                         & 10                          & 0.26           &                          7&	 0.75                         &                            5&	 0.56                                   \\
                            & $2^{-7}$             & 4161282              & 12                         & 1.54                         & 11                          & 1.51            &                          11&	 8.29                          &                           5&	 3.60                                 \\
                            & $2^{-8}$             & 33422850             & 14                         & 14.78                        & 12                          & 13.68           &                         16&	 109.00                         &                           5&	 22.10                                       \\ \hline
\multirow{4}{*}{$10^{-10}$} & $2^{-5}$             & 63426                & 7                          & 0.051                        & 7                           & 0.057             &                         3&	 0.07                     &                           4&	 0.09                           \\
                            & $2^{-6}$             & 515970               & 7                          & 0.20                         & 7                           & 0.17            &                          3&	 0.37                      &                            5&	 0.56                                       \\
                            & $2^{-7}$             & 4161282              & 9                          & 1.19                         & 8                           & 1.16              &                          4&	 3.38                         &                            5&	 3.61                                     \\
                            & $2^{-8}$             & 33422850             & 10                         & 10.91                        & 9                           & 10.49               &                          7&	 50.73                         &                            5&	 21.08                          \\ \hline
\end{tabular}
}
\end{center}
\end{table}

\begin{figure}[h!]
     \centering
     \begin{subfigure}[b]{0.48\textwidth}
         \centering
         \includegraphics[width=\textwidth]{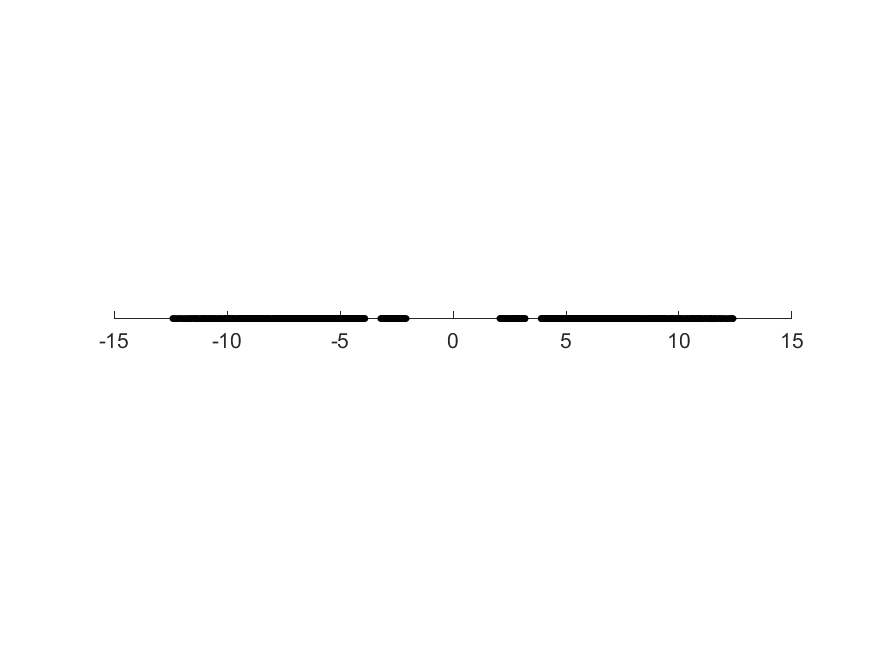}
         \caption{\footnotesize Unpreconditioned}
     \end{subfigure}
     \hfill
     \begin{subfigure}[b]{0.48\textwidth}
         \centering
         \includegraphics[width=\textwidth]{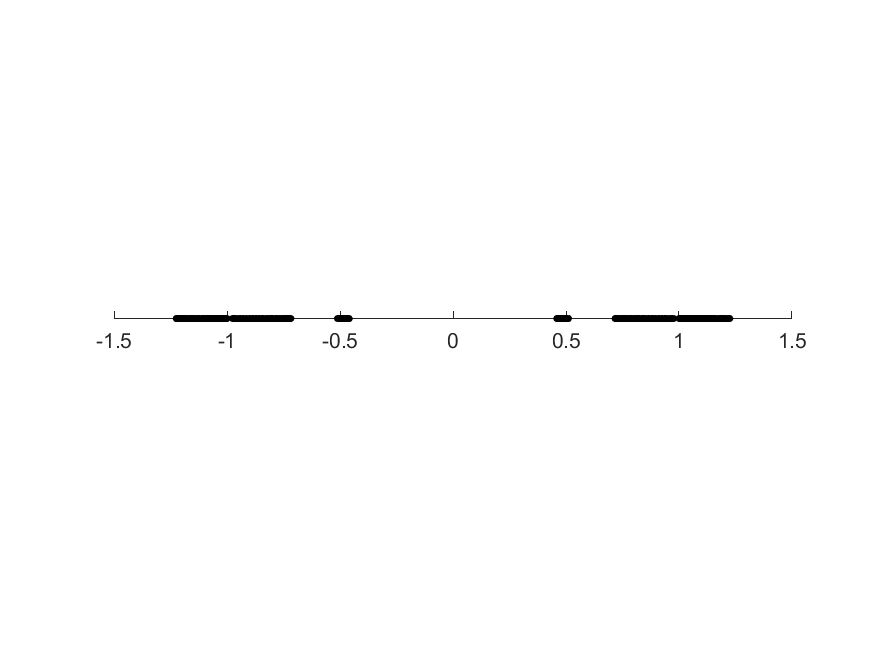}
         \caption{\footnotesize with $\mathcal{P}_{S}$}
     \end{subfigure}
     \hfill
     \begin{subfigure}[b]{0.48\textwidth}
         \centering
         \includegraphics[width=\textwidth]{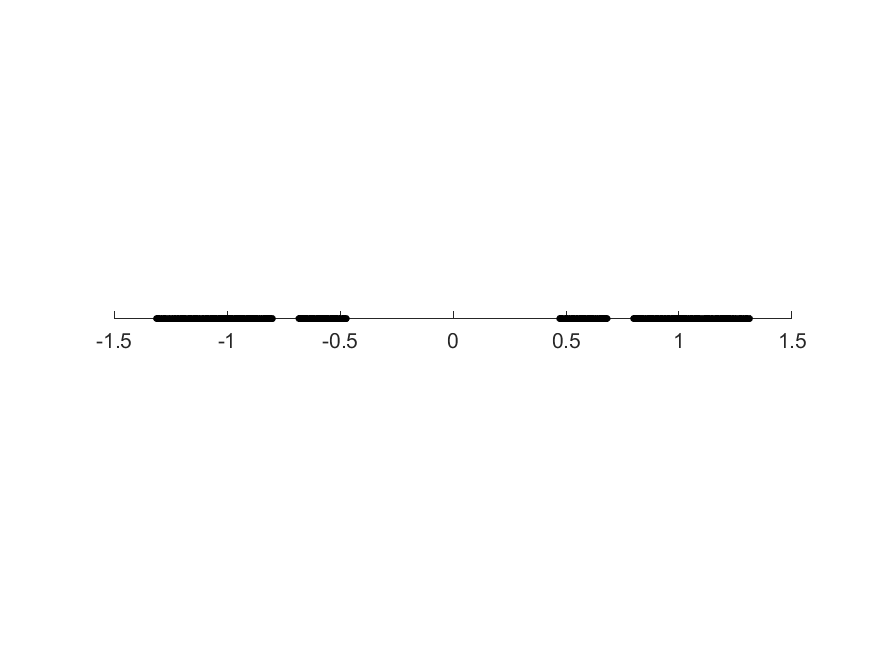}
         \caption{\footnotesize with $\mathcal{P}_{G}$}
     \end{subfigure}
       \hfill
     \begin{subfigure}[b]{0.48\textwidth}
         \centering
         \includegraphics[width=\textwidth]{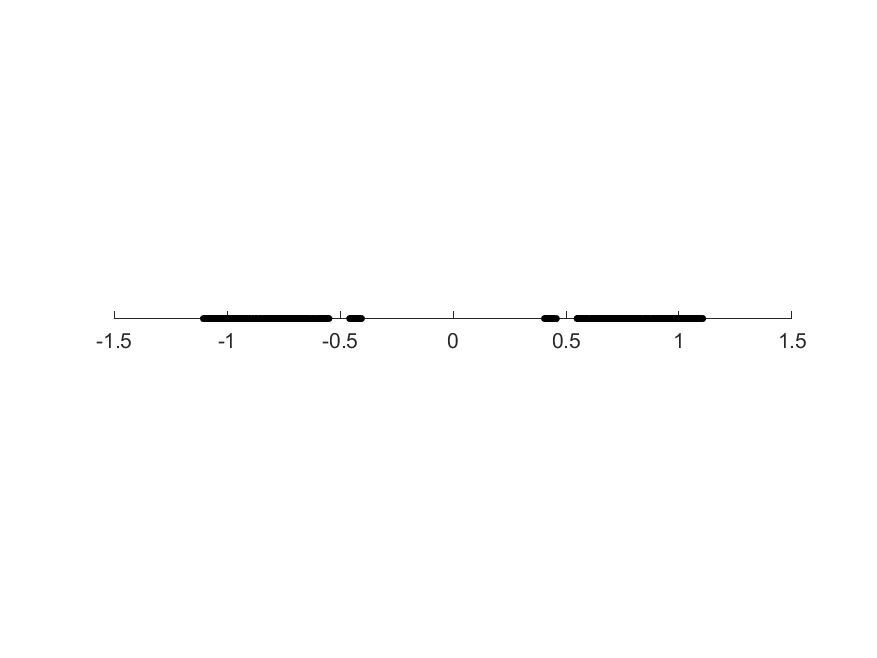}
         \caption{\footnotesize with $\widetilde{\mathcal{P}}_{S}$}
     \end{subfigure}
     \hfill
     \begin{subfigure}[b]{0.48\textwidth}
         \centering
         \includegraphics[width=\textwidth]{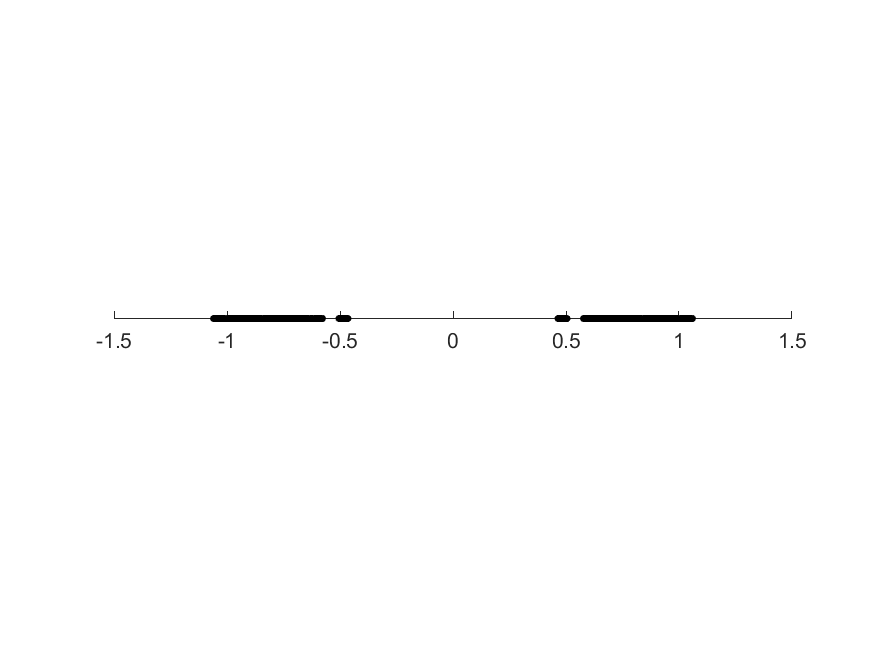}
         \caption{\footnotesize with $\widetilde{\mathcal{P}}_{G}$}
     \end{subfigure}
        \caption{{Spectrum of the preconditioned matrices for Example \ref{example:two_d_test} with $(m,n)=(16,16)$ and $\gamma =10^{-6}$.}}
        \label{fig:spectra_example_4_gamma_6}
\end{figure}

\begin{figure}[h!]
     \centering
     \begin{subfigure}[b]{0.48\textwidth}
         \centering
         \includegraphics[width=\textwidth]{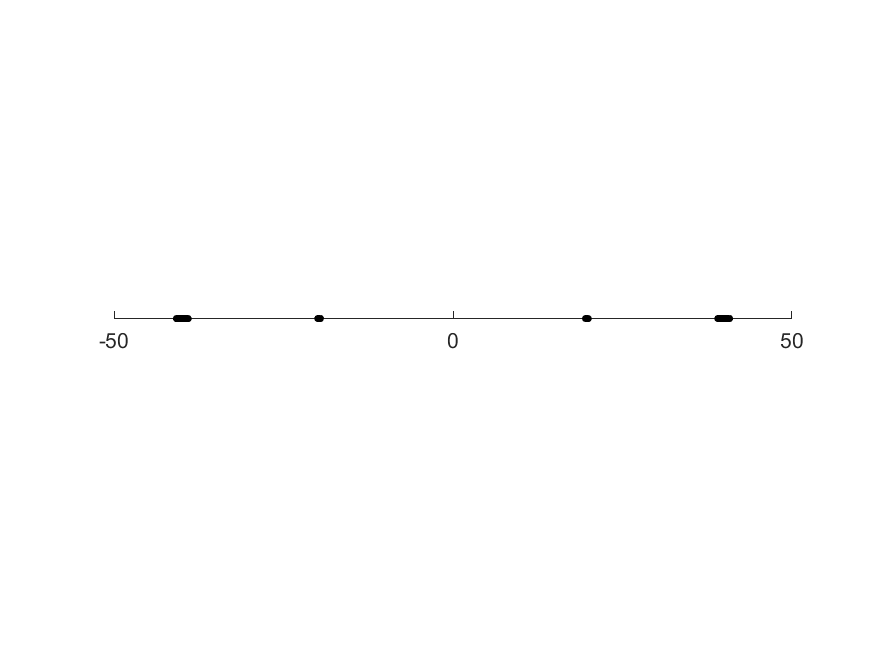}
         \caption{\footnotesize Unpreconditioned}
     \end{subfigure}
     \hfill
     \begin{subfigure}[b]{0.48\textwidth}
         \centering
         \includegraphics[width=\textwidth]{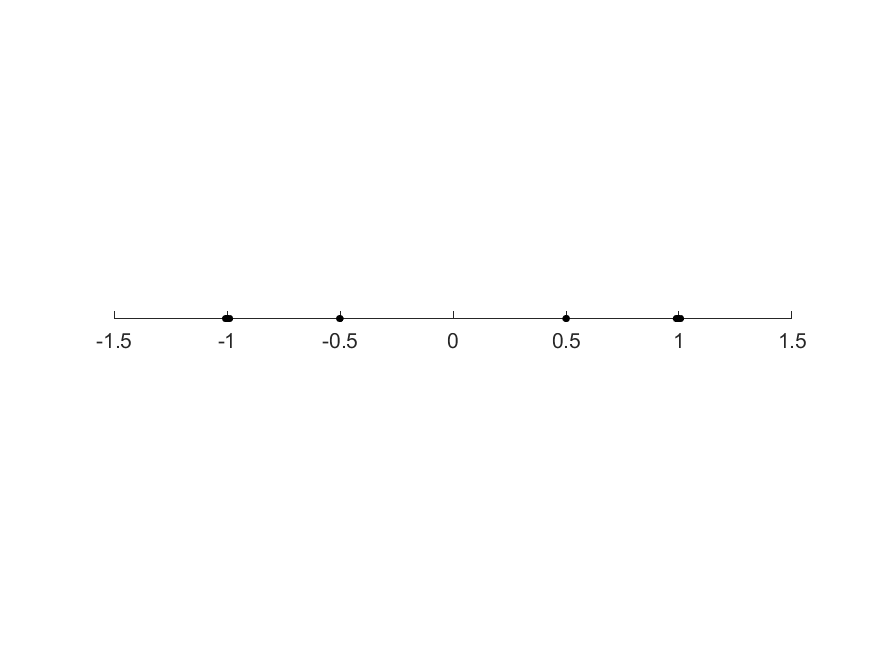}
         \caption{\footnotesize with $\mathcal{P}_{S}$}
     \end{subfigure}
     \hfill
     \begin{subfigure}[b]{0.48\textwidth}
         \centering
         \includegraphics[width=\textwidth]{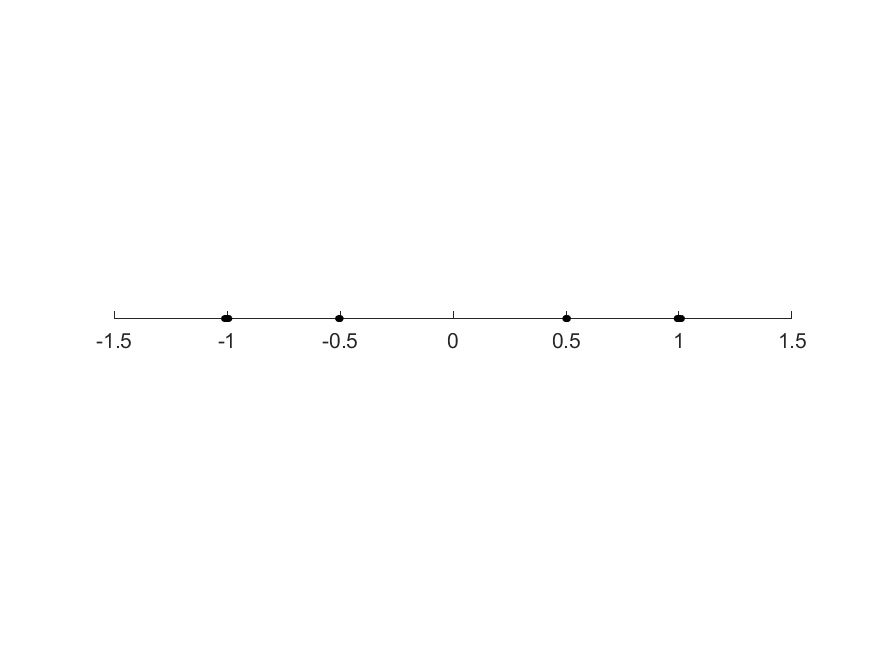}
         \caption{\footnotesize with $\mathcal{P}_{G}$}
     \end{subfigure}
     \hfill
     \begin{subfigure}[b]{0.48\textwidth}
         \centering
         \includegraphics[width=\textwidth]{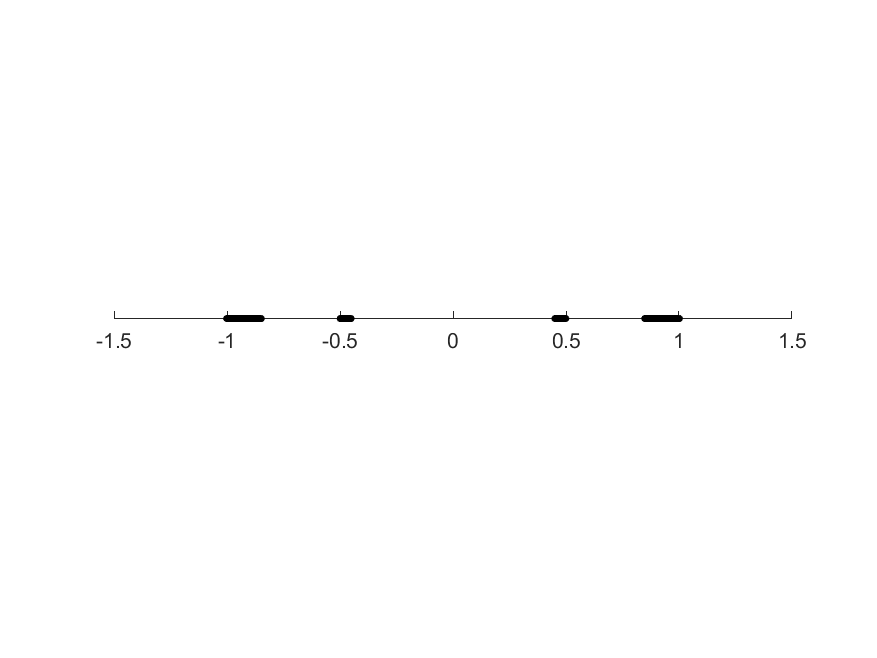}
         \caption{\footnotesize with $\widetilde{\mathcal{P}}_{S}$}
     \end{subfigure}
     \hfill
     \begin{subfigure}[b]{0.48\textwidth}
         \centering
         \includegraphics[width=\textwidth]{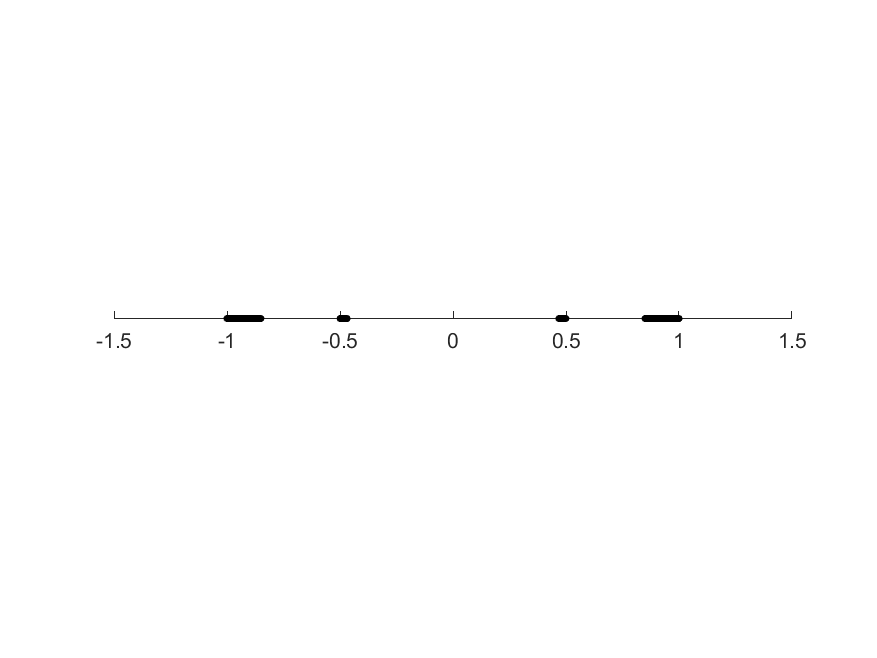}
         \caption{\footnotesize with $\widetilde{\mathcal{P}}_{G}$}
     \end{subfigure}
        \caption{{Spectrum of the preconditioned matrices for Example \ref{example:two_d_test} with $(m,n)=(16,16)$ and $\gamma =10^{-8}$.}}
        \label{fig:spectra_example_4_gamma_8}
\end{figure}

\end{example}


\section{Conclusions}\label{sec:conclusions}
We have shown the asymptotic spectral distribution result on the matrix-sequence $\{ \mathcal{A} \}_n$. {According to such a spectral study, we have proposed an ideal preconditioner $|\mathcal{H}|$, which can be used as a blueprint for designing efficient preconditioners for the all-at-once system of the wave control problem $\mathcal{A} \mathbf{u} = \mathbf{f}$. Then, we have proposed {four} novel block PinT preconditioners {($\mathcal{P}_{S}$, $\mathcal{P}_{G}$, $\widetilde{\mathcal{P}}_{S}$, and $\widetilde{\mathcal{P}}_{G}$)} for $\mathcal{A} \mathbf{u} = \mathbf{f}$, and shown its effectiveness and efficiency in the construction and in their preconditioning effect, via both via numerical evidence and a detailed theoretical study.} In particular, we have shown rapid convergence with MINRES for $\mathcal{A}$. Our preconditioners can be seen as symbol-matching preconditioners, according to the previous results on the spectral distribution of symmetrized Toeplitz matrix-sequences \cite{Ferrari2019}. {We have demonstrated the effectiveness of our PinT preconditioners in the numerical examples, and the results indicate that our PMINRES solvers, equipped with guaranteed convergence based only on eigenvalues, can provide a competitive alternative for the concerned wave optimal control problem. Noticeably, our proposed $\widetilde{\mathcal{P}}_{G}$ has shown numerical robustness for a wide range of the parameter $\gamma$, whose preconditioning effect deserves further investigation. As future work, we plan at least the next \cred{two} items:
\begin{itemize}
\item The design and the analysis of more effective PinT preconditioners that can achieve a total robustness, i.e., a convergence rate that is independent of all parameters involved in the optimal control problem;
\item An analysis of the considered approaches when $m$ is large and potentially depending on $n$\cred{, that is $m=m(n)$}: in this setting the use of a $m\times m$ matrix-valued symbol cannot be employed and different tools have to be considered such as symbols defined in higher dimensional domains \cite{GaroniCapizzano_two}.

\end{itemize}
}

{
\section*{Acknowledgments}
The authors would like to thank the Editor and the anonymous referees for their careful reading and for their helpful and pertinent suggestions. The work of Sean Hon was supported in part by the Hong Kong RGC under grant 22300921, a start-up grant from the Croucher Foundation, and a Tier 2 Start-up Grant from Hong Kong Baptist University. The work of Stefano Serra-Capizzano was supported in part by INDAM-GNCS.
The work of Stefano Serra-Capizzano is funded from the European High-Performance Computing Joint Undertaking  (JU) under grant agreement No 955701. The JU receives support from the European Union’s Horizon 2020 research and innovation programme and Belgium, France, Germany, Switzerland.
Furthermore Stefano Serra-Capizzano is grateful for the support of the Laboratory of Theory, Economics and Systems – Department of Computer Science at Athens University of Economics and Business.

}

\bibliographystyle{plain}

\end{document}